\documentclass[letterpaper]{amsart}
\pdfoutput=1
\newtheorem{theorem}{Theorem}[section]
\newtheorem{lemma}[theorem]{Lemma}

\newtheorem{proposition}[theorem]{Proposition}
\newtheorem{corollary}[theorem]{Corollary}

\theoremstyle{definition}
\newtheorem{definition}[theorem]{Definition}
\newtheorem{example}[theorem]{Example}

\theoremstyle{remark}
\newtheorem{remark}[theorem]{Remark}
\usepackage{graphicx}
\usepackage{color}
\usepackage{amsmath}
\usepackage{amsfonts}
\usepackage{amssymb}
\usepackage{epsfig}
\usepackage{rotating}
\usepackage{graphics}
\newcommand{\be}{\begin{equation}}

\newcommand{\ee}{\end{equation}}

\newcommand{\al}{\alpha}

\newcommand{\Om}{\Omega}

\newcommand{\om}{\omega}

\newcommand{\la}{\lambda}

\newcommand{\dz}{\wedge}

\newcommand{\ba}{\begin{array}}

\newcommand{\ea}{\end{array}}

\newcommand{\beq}{\begin{eqnarray}}

\newcommand{\eeq}{\end{eqnarray}}

\newtheorem{lm}{lemma}

\newtheorem{thee}{theorem}

\newtheorem{proo}{proposition}

\newtheorem{co}{corollary}

\newtheorem{rem}{remark}

\newtheorem{deff}{definition}

\newcommand{\bd}{\begin{deff}}

\newcommand{\ed}{\end{deff}}

\newcommand{\bl}{\begin{lm}}

\newcommand{\el}{\end{lm}}

\newcommand{\bp}{\begin{proo}}

\newcommand{\ep}{\end{proo}}

\newcommand{\bt}{\begin{thee}}

\newcommand{\et}{\end{thee}}

\newcommand{\bc}{\begin{co}}

\newcommand{\ec}{\end{co}}

\newcommand{\brm}{\begin{rem}}

\newcommand{\erm}{\end{rem}}

\newcommand{\der}{{\rm d}}

\usepackage{t1enc}
\def\frak{\mathfrak}

\newcommand{\newc}{\newcommand}

\let\ccdot\cdot
\def\cdot{\hbox to 2.5pt{\hss$\ccdot$\hss}}

\newc{\aR}{\mbox{\boldmath{$ R$}}}
\newc{\aS}{\mbox{\boldmath{$ S$}}}
\newc{\aT}{\mbox{\boldmath{$ T$}}}
\newc{\aW}{\mbox{\boldmath{$ W$}}}

\newc{\aK}{\mbox{\boldmath{$ K$}}}
\newc{\aL}{\mbox{\boldmath{$ L$}}}
\usepackage{amssymb}
\usepackage{amscd}

\newcommand{\hook}{\raisebox{-0.35ex}{\makebox[0.6em][r]
{\scriptsize $-$}}\hspace{-0.15em}\raisebox{0.25ex}{\makebox[0.4em][l]{\tiny
 $|$}}}

\newcommand{\bma}{\begin{pmatrix}}
\newcommand{\ema}{\end{pmatrix}}

\newcommand{\bet}{\beta}

\newc{\obstrn}[2]{B^{#1}_{#2}}

\newcommand{\rpl}                         % +) or <+
{\mbox{$
\begin{picture}(12.7,8)(-.5,-1)
\put(0,0.2){$+$}
\put(4.2,2.8){\oval(8,8)[r]}
\end{picture}$}}

\newcommand{\lpl}                         % (+ or +>
{\mbox{$
\begin{picture}(12.7,8)(-.5,-1)
\put(2,0.2){$+$}
\put(6.2,2.8){\oval(8,8)[l]}
\end{picture}$}}

\usepackage{ifthen}

\newc{\tensor}[1]{#1}
\newc{\Mvariable}[1]{\mbox{#1}}
\newc{\down}[1]{{}_{#1}}
\newc{\up}[1]{{}^{#1}}

\newc{\JulyStrut}{\rule{0mm}{6mm}}
\newc{\midtenPan}{\mbox{\sf S}}
\newc{\midten}{\mbox{\sf T}}
\newc{\midtenEi}{\mbox{\sf U}}
\newc{\ATen}{\mbox{\sf E}}
\newc{\BTen}{\mbox{\sf F}}
\newc{\CTen}{\mbox{\sf G}}

\def\sideremark#1{\ifvmode\leavevmode\fi\vadjust{\vbox to0pt{\vss% the remark
 \hbox to 0pt{\hskip\hsize\hskip1em%                          will appear only
 \vbox{\hsize3cm\tiny\raggedright\pretolerance10000%          on the side
 \noindent #1\hfill}\hss}\vbox to8pt{\vfil}\vss}}}%
\newcommand{\Span}{\mathrm{Span}}
\numberwithin{equation}{section}

\newcounter{romenumi}
\newcommand{\labelromenumi}{(\roman{romenumi})}

\begin{document}
\title{Differential equations and para-CR structures}
\vskip 1.truecm
\author{C. Denson Hill} \address{Department of Mathematics, Stony
  Brook University, Stony Brook, N.Y. 11794, USA}
\email{dhill@math.sunysb.edu}  
\author{Pawe\l~ Nurowski} \address{Instytut Fizyki Teoretycznej,
Uniwersytet Warszawski, ul. Hoza 69, Warszawa, Poland}
\email{nurowski@fuw.edu.pl}
%\thanks{kbn}

\date{\today}
%\begin{abstract}
%We consider ...
%\end{abstract}
\maketitle
%*************
\tableofcontents
\newcommand{\bbS}{\mathbb{S}}
\newcommand{\bbR}{\mathbb{R}}
\newcommand{\sog}{\mathbf{SO}}
\newcommand{\cog}{\mathbf{CO}}
\newcommand{\slg}{\mathbf{SL}}
\newcommand{\og}{\mathbf{O}}
\newcommand{\soa}{\frak{so}}
\newcommand{\coa}{\frak{co}}
\newcommand{\oa}{\frak{o}}
\newcommand{\sla}{\frak{sl}}
\newcommand{\sua}{\frak{su}}
\newcommand{\dr}{\mathrm{d}}
\newcommand{\sug}{\mathbf{SU}}
\newcommand{\gat}{\tilde{\gamma}}
\newcommand{\Gat}{\tilde{\Gamma}}
\newcommand{\thet}{\tilde{\theta}}
\newcommand{\Thet}{\tilde{T}}
\newcommand{\rt}{\tilde{r}}
\newcommand{\st}{\sqrt{3}}
\newcommand{\kat}{\tilde{\kappa}}
\newcommand{\kz}{{K^{{~}^{\hskip-3.1mm\circ}}}}
\newcommand{\bv}{{\bf v}}
\newcommand{\di}{{\rm div}}
\newcommand{\curl}{{\rm curl}}
\newcommand{\cs}{(M,{\rm T}^{1,0})}
\newcommand{\tn}{{\mathcal N}}
%*************
\section{Introduction}\label{intor}
   Aldo Andreotti liked simple ideas best. He often said "The
more simple an idea is, the better it is". He also liked explicit
provocative examples which begged for the development of a new
general theory. We think he would have enjoyed hearing the 
story we tell here.

   A para-CR structure is the real analogue of a CR structure (see
   Definition \ref{pcr}). The
main point is that $K^2 = I$, instead of $J^2 = -I$, and one does not
insist that $\dim H^+ = \dim H^-$, as in the situation of CR structures
(where $\dim H^{1,0} = \dim H^{0,1}$ happens accidentally). Here $H^\pm$
are the $\pm 1$ eigenspaces of $K$. Assuming that one is already familiar
with CR structures, then here is the simple idea: "Change the
sign and allow the dimensions of the eigenspaces to differ".

   What are the provocative examples? One of the goals of this
paper is to provide a few of them.

Rather than overburden this introduction with a lengthy
description of what is contained here, we refer the reader
to the detailed table of contents. If we were to highlight 
the Sections of the paper that in our opinion are the most
interesting, we 
would indicate Sections \ref{buur} and 
\ref{buru}. 
\section{To para-CR structures via ODEs}
\subsection{Geometry of general solutions of ODEs modulo point transformations}\label{s1}
The abstract notion of a \emph{para-CR manifold} \cite{alex1} appears 
naturally in the context of systems of differential equations considered 
modulo point transformations of variables \cite{nurdif,nurspar}. In the simplest case of a single 
ordinary differential equation of $n$th order, 
\be y^{(n)}=F(x,y,y',...,y^{(n-1)}),\label{ode}\ee
for a real function $\bbR\ni x\mapsto y(x)\in\bbR$, such an equation has a \emph{general solution}
\be y=\psi(x,a_0,a_1,...,a_{n-1}),\label{sol}\ee
depending on $n$ arbitrary real parameters
$(a_0,a_1,...,a_{n-1})$. Thus the general solution of such an equation may be considered as a \emph{hypersurface} $\Sigma$ in $\bbR^2\times \bbR^n$ defined by
\be 
\Sigma=\{\bbR^2\times\bbR^n\ni(y,x,a_0,a_1,...,a_{n-1})~|~ \Psi(y,x,a_0,a_1,...,a_{n-1})=0\},\label{set}\ee
where
$\Psi(y,x,a_0,a_1,...,a_{n-1})=y-\psi(x,a_0,a_1,...,a_{n-1})$. Now
consider a diffeomorphism of $\bbR^2\times\bbR^n$, which preserves the
split of $\bbR^{(2+n)}$ onto $\bbR^2$ and $\bbR^n$. This may mix the
variables $y$ and $x$, and, \emph{separately}, may mix the variables $a_0$, $a_1$, ...,$a_{n-1}$; 
it cannot however mix $y$ and $x$ with the $a_i$s. Explicitly it is given 
by  
$$\bbR^2\times\bbR^n\ni (y,x,a_0,a_1,...,a_{n-1})\mapsto(\bar{y},\bar{x},\bar{a}_0,\bar{a}_1,...,\bar{a}_{n-1})\in\bbR^2\times\bbR^n,$$
where
\begin{eqnarray}
&&\bar{y}=\bar{y}(y,x),\nonumber\\&& \bar{x}=\bar{x}(y,x),\label{var}\\
&&\bar{a}_i=\bar{a}_i(a_0,a_1,...,a_{n-1}),\quad\quad i=0,1,...,n-1.\nonumber
\end{eqnarray} 
This diffeomorphism transforms $\Sigma$ to another hypersurface in $\bbR^2\times\bbR^n$, which defines the general solution to an ODE which is locally 
\emph{point equivalent} to the ODE (\ref{ode}).

To understand the geometry of general solutions of such ODEs (\ref{ode}) modulo point transformations better, it is convenient to pass to a bit more general setting. Thus, without referring to any ODE, we consider $\bbR^{(2+n)}$ equipped with a \emph{linear} operator 
$$\kappa:\bbR^{(2+n)}\to\bbR^{(2+n)},\quad\quad{\rm such~that}\quad\quad 
\kappa^2={\rm id}.$$ 
The operator $\kappa$ has two eigenvalues: $+1$ and $-1$, and we \emph{assume} that the corresponding eigenspaces are, respectively, $\chi_+=\bbR^2$, with  
eigenvalue $+1$, and $\chi_-=\bbR^n$, with eigenvalue $-1$. We adapt a coordinate system $(y,x,a_0,...,a_{n-1})$ in $\bbR^{(2+n)}$, so 
that $\chi_+={\rm Span}(\partial_y,\partial_x)$ and $\chi_-={\rm Span}(\partial_{a_0},...,\partial_{a_{n-1}})$.

Given $\bbR^{(2+n)}$ with such a $\kappa$, we consider a smooth real function 
$$\Psi:\bbR^{(2+n)}\to\bbR.$$ This function is supposed to have
\emph{zero} as a  
\emph{regular value}. With this assumption the set $\Sigma$ as in
(\ref{set}) is a codimension one submanifold of $\bbR^{(2+n)}$. In
addition we assume that $\Sigma$ is \emph{generically} embedded, 
which means that its tangent space at each point, ${\rm T}_p\Sigma$, 
is spanned by the \emph{linearly independent} vectors 
\begin{eqnarray*}
X_1&=&\Psi_x\partial_y-\Psi_y\partial_x\\
Y_1&=&\Psi_1\partial_0-\Psi_0\partial_1\\
Y_2&=&\Psi_2\partial_1-\Psi_1\partial_2\\
\phantom{}&\phantom{}&\dots\\
Y_{n-1}&=&\Psi_{n-1}\partial_{n-2}-\Psi_{n-2}\partial_{n-1}\\
Z&=&\Psi_0\partial_y-\Psi_y\partial_0.
\end{eqnarray*}
Here $\partial_i=\tfrac{\partial}{\partial a_i}$, $i=0,\dots,(n-1)$, and $\Psi_x=\partial_x(\Psi)$, $\Psi_y=\partial_y(\Psi)$, $\Psi_i=\partial_i(\Psi)$.

Note that the operator $\kappa$ from the ambient space $\bbR^{(2+n)}$ defines a vector subspace $H_p$ of ${\rm T}_p\Sigma$ by
$$H_p=\kappa({\rm T}_p\Sigma)\cap {\rm T}_p\Sigma.$$
In the above basis of ${\rm T}_p\Sigma$ we have
$$H_p={\rm Span}(X_1,Y_1,\dots,Y_{n-1}).$$
Moreover, $\kappa$ restricts to $H_p$, defining an operator 
$K_p:H_p\to H_p$, $K_p=\kappa_{|H_p}$.  Since $K_p^2=id$, it splits $H_p$ onto $H_p=H^+_p\oplus H^-_p$; the spaces $H^\pm_p$ correspond to the $\pm$ eigenvalues of $K_p$. We have
$$H^+_p={\rm Span}(X_1),\quad\quad H^-_p={\rm Span}(Y_1,\dots,Y_{n-1}).$$
It further follows that the distributions
$H^+=\bigcup_{p\in\Sigma}H^+_p$ and $H^-=\bigcup_{p\in\Sigma}H^-_p$
are \emph{integrable}. They define \emph{two foliations} on $\Sigma$,
one of which has 1-dimensional leaves tangent to $X_1$, and the other
has $(n-1)$-dimensional leaves tangent to all the $Y_i$s. 
These two foliations are obtained by the intersections of $\Sigma$ with the leaves $\pi^{-1}_\pm(v_\mp)$, $v_{\mp}\in\chi_\mp$, of the respective foliations $\pi_+:\bbR^{(2+n)}\to\chi_-$ and $\pi_-:\bbR^{(2+n)}\to \chi_+$.   

Note also that although both distributions $H^+$ and $H^-$ are
automatically integrable, the distribution $H$ is \emph{in general
  not} integrable. For $H$ to be integrable the defining function
$\Psi$ would have to satisfy the $\tfrac{n(n-1)}{2}$ conditions:
$$\Psi_y\Psi_{x[i}\Psi_{j]}-\Psi_x\Psi_{y[i}\Psi_{j]}=0,$$
for all $i,j=0,1,\dots,(n-1)$. Here $\Psi_{yi}=\tfrac{\partial^2\Psi}{\partial y\partial a_i}$, and $\Psi_{x[i}\Psi_{j]}=\tfrac12 (\Psi_{xi}\Psi_j-\Psi_{xj}\Psi_i)$, etc. 

\subsection{Abstract para-CR manifolds} 
The structure on $\Sigma$ consisting of $K$ and $H=H^+\oplus H^-$ is
precisely 
the structure of a \emph{para-CR manifold}, which abstractly can be
defined, somwhat more generally, as follows:
\begin{definition}\label{pcr}
A $(k+n)$-dimensional manifold $M$ equipped 
with an $n$-dimensional distribution $H$ together with a linear operator $K:H\to H$, such that $K^2=id$, is called an \emph{almost para-CR manifold}. If in addition \emph{both} eigenspaces of $K$, $H^+=\{X\in H, KX=X\}$ and $H^-=\{X\in H, KX=-X\}$, \emph{are integrable}, $[H^\pm,H^\pm]\subset H^\pm$, then an almost para-CR manifold 
$(M,H,K)$ is called an abstract \emph{para-CR manifold}. 
The \emph{type} of the abstract para-CR manifold will be denoted
by $(k,r,s)$ where $k$ is the para-CR codimension, and $r=\dim H^+$,
$s=\dim H^-$
\end{definition}

In the following we will only consider \emph{smooth} para-CR structures, i.e. 
smooth manifolds $M$, with both $H$ and $K$ being smooth.

In the case of the hypersurfaces $\Sigma$ considered above,  $\Sigma$
has type $(1,1,n-1)$. What is more important, in this case the para-CR
structure $(H,K)$ was \emph{induced on} $\Sigma$ \emph{from the
  ambient space} $(\bbR^{(2+n)},\kappa)$. A natural question arises if
an abstractly defined para-CR manifold $(M,K,H)$, as in Definition 
\ref{pcr}, can be (locally) generically embedded as a submanifold $\Sigma$ in 
some $\bbR^{(m+n)}$ equipped with a linear operator 
$\kappa:\bbR^{(m+n)}\to\bbR^{(m+n)}$, $\kappa^2=id$, 
having $\bbR^m$ as its $+1$ eigenspace, and $\bbR^n$ as its $-1$ 
eigenspace, so that the induced para-CR structure on 
$\Sigma$ coincides with that of $(M,K,H)$.

To answer this question we need some preparations.

\begin{definition}\label{dpcrr}
Two abstract para-CR structures $(M_1,H_1,K_1)$ and $(M_2,H_2,K_2)$ 
are (locally) equivalent iff there exists a (local) diffeomorphism 
$\Phi:M_1\to M_2$ such that $\Phi_* H_1=H_2$ and $\Phi_*\circ
K_1=K_2\circ\Phi_*$. Such a $\Phi$ is called a \emph{para-CR diffeomorphism}.
\end{definition}

A \emph{dual formulation} of the para-CR definition is very useful:
\begin{definition}\label{dpcr}
An almost para-CR structure (of type $(k,r,s)$) is a $(k+n)$-dimensio\-nal manifold $M$
equipped with an equivalence class of $(k+r+s)$ one-forms 
$(\lambda_1,\dots,\lambda_k,$ $\mu_1,\dots,\mu_r,\nu_1,\dots,\nu_s)$ such that
\begin{itemize}
\item $r+s=n$,
\item $\lambda_1\dz\dots\lambda_k\dz\mu_1\dots\mu_r\dz\nu_1\dots\nu_s\neq 0$ at each point of $M$,
\item two choices of 1-forms
  $(\lambda_1,\dots,\lambda_k,\mu_1,\dots,\mu_r,\nu_1,\dots,\nu_s)$
  and $(\lambda'_1,\dots,\lambda'_k,$
  $\mu'_1,\dots,\mu'_r,\nu'_1,\dots,\nu'_s)$ are in an equivalence relation iff there exist real functions $a^i_{~j}$, $b^j_{~A}$, $c^j_{~\alpha}$, $f^A_{~B}$, $h^\alpha_{\beta}$, 
with $i,j=1,\dots k$; $A,B=1,\dots, r$; $\alpha,\beta=1,\dots, s$, on $M$ such that:
\be\lambda'_i=a^j_{~i}\lambda_j,\quad\mu'_A=f^B_{~A}\mu_B+b^j_{~A}\lambda_j,\quad
\nu'_\alpha=h^\beta_{~\alpha}\nu_\beta+c^j_{~\alpha}\lambda_j,\label{pt}\ee
and ${\rm det}(a^i_{~j}){\rm det}(f^A_{~B}){\rm det}(h^\alpha_{~\beta})\neq 0$. 
\end{itemize} 
An almost para-CR structure is an \emph{integrable} para-CR structure iff, in addition, the following equations
\begin{eqnarray}
&&\der\lambda_i\dz\lambda_1\dz\dots\dz\lambda_k\dz\mu_1\dz\dots\dz\mu_r=0\nonumber\\
&&\der\mu_A\dz\lambda_1\dz\dots\dz\lambda_k\dz\mu_1\dz\dots\dz\mu_r=0\label{integ1}
\end{eqnarray}
and 
\begin{eqnarray}
&&\der\lambda_i\dz\lambda_1\dz\dots\dz\lambda_k\dz\nu_1\dz\dots\dz\nu_s=0\nonumber\\
&&\der\nu_\alpha\dz\lambda_1\dz\dots\dz\lambda_k\dz\nu_1\dz\dots\dz\nu_s=0\label{integ2}
\end{eqnarray}
are simultaneously satisfied, for all $i=1,\dots,k$, 
$A=1,\dots,r$, $\alpha=1,\dots,s$, and for one (therefore all) 
representatives
$(\lambda_1,\dots,\lambda_k,\mu_1,\dots,\mu_r,\nu_1,\dots,\nu_s)$ of
an equivalence class $[(\lambda_1,\dots,\lambda_k,\mu_1,\dots,\mu_r,\nu_1,\dots,\nu_s)]$. 
\end{definition}
One observes that Definition \ref{dpcr} is the dual version of Definition \ref{pcr} identifying $H^-$ with the anihilator of $(\lambda_1,\dots,\lambda_k,\mu_1,\dots,\mu_r)$ and $H^+$ with the anihilator of $(\lambda_1,\dots,\lambda_k,\nu_1,\dots,\nu_s)$. Thus $H^+$ is $r$-dimensional, and $H^-$ is $s$-dimensional, with $H=H^+\oplus H^-$ being $r+s=n$-dimensional. In particular $H$ is integrable iff $\der\lambda_i\dz\lambda_1\dz\dots\dz\lambda_k=0$ for all $i=1,\dots,k$.
\begin{example}\label{el1}
Given an $n$-th order ODE (\ref{ode}) we introduce a canonical para-CR
structure on the space $\mathcal J$ of the $(n-1)$ jets. 
Parametrizing this space by 
$(x,y,y^1,\dots,y^{n-1})$ we introduce
\begin{eqnarray}
&&\lambda=\der y-y^1\der x,\nonumber\\
&&\mu=\der x,\label{cf}\\
&&\nu_i=\der y^i-y^{i+1}\der x,\quad\quad\quad\quad\quad\quad\quad\quad\quad\forall i=1,\dots,n-2,\nonumber\\
&&\nu_{n-1}=\der y^{n-1}-F(x,y,y^1,\dots,y^{n-1})\der x,\nonumber
\end{eqnarray} 
and define the class $[\lambda,\mu,\nu_\alpha]$ on $\mathcal J$ 
via: $$(\lambda,\mu,\nu_\alpha)\sim(\lambda',\mu',\nu'_\alpha)\quad\quad{\rm
  iff}\quad\quad \lambda'=a\lambda,~\mu'=f\mu+b\lambda,~{\rm and}~\nu'_\alpha=h^\beta_\alpha\nu_\beta+c_\alpha\lambda,$$
with functions $a,b,c,h^\beta_\alpha,c_\alpha$ on $\mathcal J$, such
that $af\det(h^\alpha_\beta)\neq 0$. 
Obviously $\lambda\dz\mu\dz\nu_1\dz\dots\dz\nu_{n-1}\neq 0$,
$\der\lambda\dz\lambda\dz\mu\equiv 0\equiv\der\mu\dz\lambda\dz\mu$,
and for dimensional reasons 
$\der\lambda\dz\lambda\dz\nu_1\dz\dots\dz\nu_{n-1}\equiv 0\equiv
\der\nu_\alpha\dz\lambda\dz\nu_1\dz\dots\dz\nu_{n-1}$ for all
$\alpha=1,\dots,n-1$. This shows that $({\mathcal
  J},[\lambda,\mu,\nu_\alpha])$ is an abstract para-CR structure of type
$(1,1,n-1)$. This para-CR structure is called the \emph{canonical
  para-CR structure of an ODE} $y^{(n)}=F(x,y,y',\dots,y^{(n-1)})$. 
\end{example}

Returning to the general discussion we have the following Proposition.
\begin{proposition}\label{repp}
Every abstract para-CR manifold $(M,[\lambda_i,\mu_A,\nu_\alpha])$ of type
$(k,r,s)$ locally admits two  overlaping coordinate systems
$(y_i,x_A,a_\alpha)$ and $(\bar{y}_i,x_A,a_\alpha)$ in which the forms
$(\lambda_i,\mu_A,\nu_\alpha)$ can be written either as:
\be\lambda_i=\der y_i+L^A_{~i}\der x_A,\quad\quad\mu_A=\der x_A,\quad\quad\nu_\alpha=\der a_\alpha,\label{repa}\ee
or by
\be\lambda_i=\der \bar{y}_i+\bar{L}^\alpha_{~i}\der a_\alpha,\quad\quad\mu_A=\der x_A,\quad\quad\nu_\alpha=\der a_\alpha,\label{repb}\ee
where $L^A_{~i}=L^A_{~i}(y,x,a)$ and
$\bar{L}^\alpha_{~i}=\bar{L}^\alpha_{~i}(\bar{y},x,a)$, $i=1,\dots,k$,
$A=1,\dots, r$, $\alpha=1,\dots,s$, 
are appropriate real functions of the respective variables  $(y_i,x_A,a_\alpha)$ and $(\bar{y}_i,x_A,a_\alpha)$.  
\end{proposition}
\begin{proof}
The proof is a simple application of the Frobenius theorem:

On one hand, the Frobenius theorem applied to 
the integrability conditions (\ref{integ1}), together with the use of
transformations (\ref{pt}), imply
the existence of functions $(y_i,x_A,L^A_i)$ for which $\lambda_i=\der
y_i+L^A_{~i}\der x_A$ and $\mu_A=\der x_A$ holds. On the other
hand, the same argument applied to the integrability conditions
(\ref{integ2}), imply the existence of functions
$(\bar{y}_i,a_\alpha,\bar{L}^\alpha_i)$ for which $\lambda_i=\der
\bar{y}_i+\bar{L}^\alpha_{~i}\der a_\alpha$ and $\nu_\alpha=\der
a_\alpha$ holds. But since
$\lambda_1\dz\dots\lambda_k\dz\mu_1\dz\dots\dz\mu_r\dz\nu_1\dots\dz\nu_s\neq
0$, then taking $\lambda$s and $\mu$s from the first representation,
and $\nu$s from the second we get $\der y_1\dz\dots\der y_k\dz\der x_1\dz\dots\dz\der
x_s\dz\der a_1\dz\dots\dz\der a_s\neq 0$. Similarly, taking $\lambda$s and $\nu$s from the second representation,
and $\mu$s from the first we get $\der \bar{y}_1\dz\dots\der \bar{y}_k\dz\der x_1\dz\dots\dz\der
x_s\dz\der a_1\dz\dots\dz\der a_s\neq 0$. This shows that both sets
of functions $(y_i,x_A,a_\alpha)$
and $(\bar{y}_i,x_A,a_\alpha)$ form local coordinates on $M$. In these
coordinates the
para-CR forms have the respective desired representation (\ref{repa})
and (\ref{repb}). 
\end{proof}

\subsection{The embedding problem} Once an integrable para-CR
structure is defined in terms of
$[(\lambda_1,\dots,\lambda_k,\mu_1,\dots,\mu_r,\nu_1,\dots,\nu_s)]$ it
is easy to solve the embedding problem, at least locally. 

We have the following embedding theorem.
\begin{theorem}\label{mbe}
Every smooth $(k+r+s)$-dimensional 
abstract para-CR manifold $(M,H,K)$ with $\dim H^+=r$ and $\dim H^-=s$ is locally 
embeddable in $\bbR^{(k+r)+(k+s)}$, with the embedding 
$\iota:M\to \bbR^{(k+r)+(k+s)}$ being a para-CR diffeomorphism between $(M,H,K)$ and 
the para-CR structure which $\iota(M)$ aquires from 
the ambient space $(\bbR^{(k+r)+(k+s)},\kappa)$. Here $\kappa$ is the canonical 
linear map $\kappa: \bbR^{(k+r)+(k+s)}\to  \bbR^{(k+r)+(k+s)}$, $\kappa^2=id$, having 
$\bbR^{k+r}$ and $\bbR^{k+s}$ as its respective $+1$ and $-1$ eigenspaces. 
\end{theorem}

\begin{proof} 
Choosing a representative $(\lambda_1,\dots,\lambda_k,\mu_1,\dots,\mu_r,\nu_1,\dots,\nu_s)$ we consider vector fields 
$(Z_1,\dots,Z_k,X_1,\dots,X_r,Y_1,\dots,Y_s)$ which are the respective
duals of $(\lambda_1,$ $\dots,\lambda_k,\mu_1,\dots,\mu_r,\nu_1,\dots,\nu_s)$. This in particular means that $H^+={\rm Span}(X_1,\dots,X_r)$ and $H^-={\rm Span}(Y_1,\dots,Y_s)$. Also any differentiable function $f:M\to \bbR$ has $$\der f=Z_i(f)\lambda_i+X_A(f)\mu_A+Y_\alpha(f)\nu_\alpha$$ as its differential. 
 Now, one looks for all functions $f$ and $h$ on $M$ which satsify
\begin{eqnarray}
&&\der f\dz\lambda_1\dz\dots\dz\lambda_k\dz\mu_1\dz\dots\dz\mu_r=0,\quad{\rm and}\label{embeq1}\\
&&\der h\dz\lambda_1\dz\dots\dz\lambda_k\dz\nu_1\dz\dots\dz\nu_s=0,\label{embeq2}
\end{eqnarray}
or, what is the same,
$$Y_\alpha(f)=0,\quad\forall\alpha=1,\dots,s,\quad\quad{\rm and}\quad X_A(h)=0,\quad\forall A=1,\dots,r.$$
If, for example, we choose
$(\lambda_1,\dots,\lambda_k,\mu_1,\dots,\mu_r,\nu_1,\dots,\nu_s)$ in
the local representation (\ref{repa}), then equations
(\ref{embeq1})-(\ref{embeq2}) are, respectively, 
\begin{eqnarray}
&&\frac{\partial f}{\partial
    a_\alpha}=0,\quad\forall\alpha=1,\dots,s,\label{ds1}\\
&&\frac{\partial h}{\partial x_A}-L^A_{~i}\frac{\partial h}{\partial
    y^i}=0,\quad\forall A=1,\dots,r.\label{ds2}
\end{eqnarray}
Thus in this coordinate system equations (\ref{ds1}) for the function
$f$ are trivial to solve: they 
obviously have $k+r$ independent solutions given by 
$f_1=y_1,\dots, f_k=y_k,\tilde{f}_1=x_1,\dots,\tilde{f}_r=x_r$. 
The equations (\ref{ds2}) for the function $h$ do not look very nice
in this coordinate system. To analyse them it is convenient to use the 
other coordinate system, $(\bar{y}_i,x_A,a_\alpha)$, in which equations
(\ref{embeq1})-(\ref{embeq2}) are, respectively:
\begin{eqnarray}
&&\frac{\partial f}{\partial
    a_\alpha}-\bar{L}^\alpha_{~i}\frac{\partial f}{\partial\bar{y}^i}=0,\quad\forall \alpha=1,\dots,s,\label{ds12}\\
&&\frac{\partial h}{\partial x_A}=0,\quad\forall  A=1,\dots,r.\label{ds22}
\end{eqnarray}
In this coordinate system equation (\ref{ds22}) for the function $h$
is trivial: it has $k+s$ independent solutions,
$h_1=\bar{y}_1,\dots,h_k=\bar{y}_k,\tilde{h}_1=a_1,\dots,\tilde{h}_s=a_s$. Now
since both coordinate systems $(y_i,x_A,a_\alpha)$ and $(\bar{y}_i,x_A,a_\alpha)$ are
defined over the same region of $M$, and because the coordinates
$(x,a)$ are the same in both systems, we have:
$$\bar{y}_i=\bar{y}_i(y,x,a),\quad\quad{\rm and}\quad\quad y_i=y_i(\bar{y},x,a).$$ 
This shows that the two maps:
$$M\ni(y_i,x_A,a_\alpha)\stackrel{\iota}{\mapsto}(f_i,\tilde{f}_A,h_j,\tilde{h}_\alpha)=(y_i,x_A,\bar{y}_j(y,x,a),a_\alpha)\in\bbR^{(k+r)+(k+s)}$$
and
$$M\ni(\bar{y}_i,x_A,a_\alpha)\stackrel{\bar{\iota}}{\mapsto}(f_i,\tilde{f}_A,h_j,\tilde{h}_\alpha)=(y_i(\bar{y},x,a),x_A,\bar{y}_j,a_\alpha)\in\bbR^{(k+r)+(k+s)}$$
give two local embeddings of the para-CR structure 
$(M,[(\lambda,\mu,\nu)])$ in $\bbR^{(k+r)+(k+s)}$ with coordinates
$(f_i,\tilde{f}_A,h_j,\tilde{h}_\alpha)$. It follows that the $\kappa$
operator in $\bbR^{(k+r)+(k+s)}$, splitting it onto
$\bbR^{(k+r)+(k+s)}=\bbR^{k+r}\times\bbR^{k+s}$, induces two para-CR
structures on the repective images of $\iota$ and $\bar{\iota}$. These
two para-CR structures are locally equivalent, and are locally equivalent to the original structure from $M$.\end{proof}

Para-CR structures with $k=1$, for obvious reasons, are called para-CR structures of \emph{hypersurface type}. 

\subsection{Para-CR equivalence a'la Cartan}
In the following a reformulation of the (local) equivalence of two 
para-CR manifolds, in the language of the differential forms
$(\lambda_1,\dots,\lambda_k,\mu_1,$ $\dots,\mu_r,\nu_1,\dots,\nu_s)$, 
will be useful. It can be seen that Definition \ref{dpcrr} is
equivalent to
\begin{definition}
Two para-CR structures $(M,[(\lambda_i,\mu_A,\nu_\alpha)])$ and 
$(M',({\lambda_i}',{\mu_A}',{\nu_\alpha}')])$, $i=1,...,k$,
  $A=1,...,r$, $\alpha=1,...,s$, on $k+r+s$ dimensional manifolds $M$
  and $M'$   
are (locally) equivalent iff there exists a (local) diffeomorphism 
$\Phi:M\to M'$ and real functions 
$a^i_{~j}$, $b^j_{~A}$, $c^j_{~\alpha}$, $f^A_{~B}$,
$h^\alpha_{\beta}$ on $M$ such that:
\begin{eqnarray}
&&\Phi^*(\lambda'_i)=a^j_{~i}\lambda_j,\nonumber\\
&&\Phi^*(\mu'_A)=f^B_{~A}\mu_B+b^j_{~A}\lambda_j,\\
&&\Phi^*(\nu'_\alpha)=h^\beta_{~\alpha}\nu_\beta+c^j_{~\alpha}\lambda_j,\nonumber
\label{ptt}\end{eqnarray}
and $${\rm det}(a^i_{~j}){\rm det}(f^A_{~B}){\rm
  det}(h^\alpha_{~\beta})\neq 0$$ for all
$i,j=1,\dots k$; $A,B=1,\dots, r$; $\alpha,\beta=1,\dots, s$.
\end{definition}
\section{Para-CR structures of type $(1,1,n-1)$}
In Example \ref{el1} we associated a para-CR structure of type
$(1,1,n-1)$ with every $n$-th order ODE in the form (\ref{ode}). A
natural question arises: is every para-CR structure of type
$(1,1,n-1)$, 
at least locally, para-CR equivalent to a canonical type 
$(1,1,n-1)$ para-CR structure of some $n$-th order ODE
(\ref{ode})? Since all canonical para-CR structures of $n$-th order
ODEs, as in Example \ref{el1}, satisfy $\der\lambda\dz\lambda=\der
x\dz\der y^1\dz\der y\neq 0$, and since nonvanishing of
$\der\lambda\dz\lambda$ is invariant under any para-CR map
$\lambda\to\lambda'=a\lambda$, then we have
\begin{proposition}
A type $(1,1,n-1)$ para-CR structure $[\lambda,\mu,\nu_\alpha]$ which is locally equivalent to
the canonical para-CR structure of an $n$-th order ODE
$y^{(n)}=F(x,y,y',\dots,y^{(n-1)})$ has $\der\lambda\dz\lambda\neq 0$. 
\end{proposition} 
In view of this proposition, we now ask if every type $(1,1,n-1)$
para-CR structure with $\der\lambda\dz\lambda\neq 0$ is locally
equivalent to a structure from Example \ref{el1}. To illustrate
the problems associated with this question we consider low dimensions first.
\subsection{Para-CR structures of type $(1,1,1)$}\label{111}
This case, in a bit different context, was studied by one of us in
\cite{nurspar}. We have the following proposition.
\begin{proposition}\label{vry}
Every type $(1,1,1)$ para-CR structure $(M,[\lambda,\mu,\nu])$ with
$\der\lambda\dz\lambda\neq 0$ is locally para-CR equivalent to a type
$(1,1,1)$ para-CR structure associated with a point equivalence class
of second order ODEs. 
\end{proposition} 
\begin{proof}
This Proposition was proved in \cite{nurspar}. For completness we
present this proof also here. 

Choosing any representative $(\lambda,\mu,\nu)$ of $[\lambda,\mu,\nu]$, 
due to the low dimension of $M$, we have
$\der\lambda\dz\lambda\dz\mu\equiv 0$ and
$\der\mu\dz\lambda\dz\mu\equiv 0$. Thus, by the Frobenius theorem, 
we have functions $(x,y,A,B,C,E)$ on $M$ such that $\lambda=A\der
x+B\der y$, and $\mu=C \der x+E\der y$. Considering the allowed para-CR gauge
of $\lambda$ and $\mu$, we can rescale $\lambda$ to the form
$\lambda=\der y-p\der x$, with some function $p$ on $M$, and shift and 
rescale $\mu$ to the form $\mu=\der x$. Now our assumption
$0\neq\der\lambda\dz\lambda$ shows that $0\neq \der x\dz\der y\dz\der
p$ and, thus, $(x,y,p)$ can be considered a coordinate system on
$M$. In this coordinate system 
the form $\mu$ is locally $\mu=\al\der x+\bet\der y+\gamma\der p$,
where $\al,\bet,\gamma$ are some functions on $M$. Because of the
allowed para-CR transformations for $\mu$, we can, without loss of
generality, take $\mu=\der p-Q(x,y,p)\der x$, with $Q=Q(x,y,p)$ being 
some function on $M$. Thus our type $(1,1,1)$ para-CR structure
$(M,[\lambda,\mu,\nu])$ with $\der\lambda\dz\lambda\neq 0$ 
is locally para-CR equivalent to $(M,[\lambda=\der y -p\der x,\mu=\der
  x,\nu=\der p-Q\der x])$. Therefore $M$ can be locally identified with the
first jet space of the equation $y''=Q(x,y,y')$. The $(x,y,p)$ are 
canonical coordinates
$(x,y,p)$ on this jet space and the contact forms are given by 
the para-CR forms $\lambda=\der y-p\der x$, $\nu=\der
p-Q\der x$. The para-CR structure associated with the point equivalent
class of ODEs represented by $y''=Q(x,y,y')$ is locally para-CR
equivalent to the para-CR structure we started with.      
\end{proof}
Further details about this case, including relations to the Fefferman
construction, can be found in \cite{nurspar}.
\subsection{Para-CR structures of type $(1,1,2)$}\label{112}
Let $(M,[\lambda,\mu,\nu_1,\nu_2])$ be a general para-CR manifold of
type $(1,1,2)$ with 
\be\der\lambda\dz\lambda\neq 0.\label{gr}\ee 
By Proposition \ref{repp}, we can introduce a coordinate
system $(x,y,a_1,a_2)$ on $M$ in which
$$\lambda=\der y-p(x,y,a_1,a_2)\der x,\quad\quad\mu=\der x,\quad\quad\nu_1=\der
a_1,\quad\quad\nu_2=\der a_2,$$
with some function $p$ of the variables $(x,y,a_1,a_2)$.  

Our key question is if we can find new coordinates 
$(x,y,y^1,y^2)$ on $M$, and functions $h^\alpha_\beta$,
$c_\alpha$, $F$ on $M$, so that the form 
$\nu'_1=h^1_1\nu_1+h^2_1\nu_2+c_1\lambda$ is equal to $$\nu'_1=\der
y^1-y^2\der x$$ and the form 
$\nu'_2=h^1_2\nu_1+h^2_2\nu_2+c_2\lambda$ is equal to $$\nu'_2=\der y^2-F(x,y,y^1,y^2)\der
x.$$
If this were possible, we could bring this para-CR structure, by a para-CR transformation, to the
canonical form corresponding to the third order ODE
$y'''=F(x,y,y',y'')$. 

When looking for the desired coordinates $(x,y,y^1,y^2)$ we proceed as follows:

We set $$y^1=p(x,y,a_1,a_2),$$ 
and notice that (\ref{gr}) implies $\der x\dz\der y\dz\der y^1\neq 0$.
Thus the functions $(x,y,y^1)$ can serve as three independent
coordinates on $M$. The condition $\der x\dz\der y\dz\der y^1\neq 0$ also  
means that at least one of the derivatives $\frac{\partial y^1}{\partial
  a_1}$ or $\frac{\partial y^1}{\partial a_2}$ is not equal to zero. Assuming, without loss of
generality, that $\frac{\partial y^1}{\partial a_1}\neq 0$, we can solve $y^1=p(x,y,a_1,a_2)$ for
$a_1$ obtaining $$a_1=a_1(x,y,y^1,a_2).$$ This enables us to parametrize $M$ by
$(x,y,y^1,a_2)$. In this new parametrization we have
$$\lambda=\der y-y^1\der x,\quad\quad\mu=\der x,\quad 
\nu_1=\der [a_1(x,y,y^1,a_2)],\quad\quad\nu_2=\der a_2.$$
We note that since $$\nu_1=\der [a_1(x,y,y^1,a_2)]=\frac{\partial a_1}{\partial x}\der
x+\frac{\partial a_1}{\partial y}\der y+\frac{\partial a_1}{\partial
  y^1}\der y^1+\frac{\partial a_1}{\partial a_2}\der a_2,$$ and
$\lambda\dz\mu\dz\nu_1\dz\nu_2\neq 0$, then $\der y\dz\der
x\dz\frac{\partial a_1}{\partial y^1}\der y^2\dz\der a_2\neq 0$, and
hence $\frac{\partial a_1}{\partial y^1}\neq 0$.
Thus we may replace the para-CR form $\nu_1$ by the form $$\nu'_1=\Big(\frac{\partial a_1}{\partial y^1}\Big)^{-1}\Big(\nu_1-\frac{\partial a_1}{\partial a_2}\nu_2-\frac{\partial
  a_1}{\partial y}\lambda\Big)$$ from the
same para-CR class, obtaining 
$$\nu_1'=\der y^1-y^2\der x.$$
Here the function $y^2$ is given by  
\be
y^2=-\Big(\frac{\partial
  a_1}{\partial x}+y^1\frac{\partial
  a_1}{\partial y}\Big)\Big(\frac{\partial
  a_1}{\partial y^1}\Big)^{-1}\label{y22}.\ee

Summarizing, starting with an arbitrary type $(1,1,2)$ para-CR
structure $(M,[\lambda,\mu,$ $\nu_1,\nu_2])$, with
$\der\lambda\dz\lambda\neq 0$, we can always choose the coordinate
system $(x,y,y^1,a_2)$ and the representatives
of the basis 1-forms, so that the para-CR structure is represented by
$$\lambda=\der y-y^1\der x,\quad\quad\mu=\der x,\quad\quad\nu_1=\der
y^1-y^2\der x,\quad\quad\nu_2=\der a_2,$$
with a function $y^2=q(x,y,y^1,a_2)$ given by (\ref{y22}).  

Now, two cases may occur: 
\begin{itemize}
\item the general case, when $\frac{\partial y^2}{\partial a_2}\neq 0$,
  or
\item the degenerate case, when $\frac{\partial y^2}{\partial a_2}= 0$. 
\end{itemize}

In the general case, i.e. in the case when
\be
\frac{\partial }{\partial a_2}\Big(\Big(\frac{\partial
  a_1}{\partial x}+y^1\frac{\partial
  a_1}{\partial y}\Big)\Big(\frac{\partial
  a_1}{\partial y^1}\Big)^{-1}\Big)\neq 0,\label{sd}
\ee
we can solve $y^2=q(x,y,y^1,a_2)$ for $a_2$ obtaining
$$a_2=a_2(x,y,y^1,y^2),$$
and a system of coordinates $(x,y,y^1,y^2)$ on $M$, in which 
$$\lambda=\der y-y^1\der x,\quad\quad\mu=\der x,\quad\quad\nu_1=\der
y^1-y^2\der x,\quad\quad\nu_2=\der [a_2(x,y,y^1,y^2)].$$
Now we have
$$\nu_2=\frac{\partial a_2}{\partial x}\der
x+\frac{\partial a_2}{\partial y}\der y+\frac{\partial a_2}{\partial
  y^1}\der y^1+\frac{\partial a_2}{\partial y^2}\der y^2,$$ and since
$\lambda\dz\mu\dz\nu_1\dz\nu_2\neq 0$, we get $\frac{\partial
  a_2}{\partial y^2}\neq 0$. This enables us to replace $\nu_2$ by
another representative
$$\nu'_2=\Big(\frac{\partial a_2}{\partial y^2}\Big)^{-1}\Big(\nu_2-\frac{\partial a_2}{\partial
  y^1}\nu_1-\frac{\partial a_2}{\partial y}\lambda\Big),$$
which can be written as:
$$\nu'_2=\der y^2-F(x,y,y^1,y^2)\der x,$$
with $$F(x,y,y^1,y^2)=-\Big(\frac{\partial a_2}{\partial
  x}+y^1\frac{\partial a_2}{\partial y}+y^2\frac{\partial
  a_2}{\partial y^1}\Big)\Big(\frac{\partial a_2}{\partial
  y^2}\Big)^{-1}.$$
 Summarizing we have the following proposition.
\begin{proposition}
Every type $(1,1,2)$ para-CR structure $(M,[\lambda,\mu,\nu_1,\nu_2])$
with $\der\lambda\dz\lambda\neq 0$ can be locally represented by
1-forms
$$\lambda=\der y-y^1\der x,\quad\quad\mu=\der x,\quad\quad\nu_1=\der
y^1-y^2\der x,\quad\quad\nu_2=\der a_2,$$
with a function $y^2=q(x,y,y^1,a_2)$ of coordinates $(x,y,y^1,a_2)$ on
$M$. If, in addition, the function $y^2$ satisfies $\frac{\partial
  y^2}{\partial a_2}\neq 0$ in $\mathcal U\subset M$, one can
introduce a coordinate system $(x,y,y^1,y^2)$ in $\mathcal U$ such
that the para-CR structure can be represented by
$$\lambda=\der y-y^1\der x,\quad\quad\mu=\der x,\quad\quad\nu_1=\der
y^1-y^2\der x,\quad\quad\nu_2=\der y^2-F(x,y,y^1,y^2)\der x.$$
In such case the para-CR structure is locally para-CR equivalent to
the canonical para-CR structure associated with a third order ODE $y'''=F(x,y,y',y'').$
\end{proposition}
The nongeneric case in which (\ref{sd}) is not satisfied can be
realized in several ways. The simplest of them is if $\frac{\partial
  y^2}{\partial a_2}\equiv 0$ in the neighbourhood ${\mathcal U}\subset M$. In such a case we
have  
$$\lambda=\der y-y^1\der x,\quad\quad\mu=\der x,\quad\quad\nu_1=\der
y^1-q(x,y,y^1)\der x,\quad\quad\nu_2=\der a_2,$$
and locally ${\mathcal U}={\mathcal U}_3\times \bbR$, where ${\mathcal
U}_3$, parametrized by $(x,y,y^1)$, is equipped with a canonical
$(1,1,1)$ type para-CR structure of the second order ODE
$y''=q(x,y,y')$. Thus in such a case the type $(1,1,2)$ para-CR
structure is obtained by extending the canonical $(1,1,1)$ type
para-CR structure of the equation $y''=q(x,y,y')$, from the first jet
space $\mathcal J$ with the canonical forms $\lambda=\der y-y^1\der x$,
$\mu=\der x$, $\nu_1=\der y^1-q(x,y,y^1)\der x$ to the Cartesian product ${\mathcal
  J}\times \bbR\stackrel{\pi}{\to}\mathcal J$. If $\bbR$ in ${\mathcal
  J}\times \bbR$ is parametrized by $a_2$, then the type $(1,1,2)$
para-CR structure on ${\mathcal
  J}\times \bbR$ is given by the class of
para-CR forms $[\pi^*(\lambda),\pi^*(\mu),\pi^*(\nu_1),\nu_2=\der
  a_2]$. So also in this nongeneric case the para-CR structure
$(M,[\lambda,\mu,\nu_1,\nu_2])$ is related to the canonical para-CR
structure of an ODE, the only difference with the generic case is that
now, the ODE is of lower order.

This discussion shows that, the structure of type $(1,1,2)$ para-CR
manifolds may change from point to point: in some regions it is
locally equivalent to a para-CR structure of a third order ODE, in
some regions, to a para-CR which is Cartesian product of a para-CR
structure of second order ODE and a real line.

To illustrate the discussion of this section we consider the following  example.\begin{example}\label{ecx}
Consider $\bbR^4$ parametrized by $(x,y,a_1,a_2)$ and a type $(1,1,2)$
para-CR structure on it given in terms of a function $$p(x,y,a_1,a_2)=x
a_1+y a_2.$$ By this we mean that the para-CR structure is defined in
terms of the class of para-CR 1-forms $[\lambda,\mu,\nu_1,\nu_2]$ with
representatives
\be\lambda=\der y-(xa_1+ya_2)\der x,\quad\quad\mu=\der
x,\quad\quad\nu_1=\der a_1,\quad\quad\nu_2=\der a_2.\label{rex}\ee
Proceeding as in our discussion above we define $y^1=x a_1+y a_2$,
solve it for $a_1$, $$a_1=\frac{y^1-y a_2}{x},$$
and use $(x,y,y^1,a_2)$ as new coordinates, in which
$$\lambda=\der y-y^1\der x,\quad\quad\mu=\der
x,\quad\quad\nu_1=\der[\frac{y^1-y a_2}{x}] ,\quad\quad\nu_2=\der a_2.$$
Now because $\nu_1=\frac{a_2 y-y^1}{x^2}\der x-\frac{a_2}{x}\der
y+\frac{\der y^1}{x}-\frac{y}{x}\der a_2$, we can replace $\nu_1$ by a
new form 
$$\nu_1=\der y^1-\frac{a_2y^1x+y^1-a_2 y}{x}\der x.$$
Introducing the function 
\be 
y^2=a_2y^1+\frac{y^1-a_2y}{x},\label{y2}\ee
we see that we are in the situation $\frac{\partial y^2}{\partial
  a_2}=y^1-\frac{y}{x}\neq 0$. So we can solve for $a_2$ obtaining:
$$a_2=\frac{y^2 x-y^1}{y^1x -y}.$$
Using the coordinates $(x,y,y^1,y^2)$ we get
$$\lambda=\der y-y^1\der x,\quad\quad\mu=\der x,\quad\quad\nu_1=\der
y^1-y^2\der x,\quad\quad\nu_2=\der[\frac{y^2 x-y^1}{y^1x -y}].$$
Expanding the differential we have
\begin{eqnarray*}
&&\nu_2=\frac{x}{y^1 x-y}\der y^2+\frac{y-y^2 x^2}{(y^1x -y)^2}\der
  y^1+\frac{(y^1)^2-y^2y}{(y^1x-y)^2}\der
  x+\frac{y^2x-y^1}{(y^1x-y)^2}\der y\sim\\
&&\frac{x}{y^1x-y}\Big(\der y^2+\frac{y^2(y^1-y^2 x)}{y^1x-y}\der x\Big).
\end{eqnarray*}
This means that locally the starting para-CR structure is equivalent
to
$$\lambda=\der y-y^1\der x,\quad\quad\mu=\der x,\quad\quad\nu_1=\der
y^1-y^2\der x,\quad\quad\nu_2=\der
y^2-\frac{y^2(y^2x-y^1)}{y^1x-y}\der x,$$
and thus it comes from the third order ODE
\be y'''=\frac{y''(y''x-y')}{y'x-y}.\label{odex}\ee
To solve this equation we may use our result on local
embeddability. We can start with any representation of the class
$[\lambda,\mu,\nu_1,\nu_2]$, then find an embedding, and finally interpret it as a
general solution to (\ref{odex}). It turns out that the simplest
calculations are in the representation (\ref{rex}):

Obviously the two independent solutions $(f_1,\tilde{f}_1)$ of 
the embedding equations $\der f\dz\lambda\dz\mu\equiv 0$ are $f_1=x$
and $\tilde{f}_1=y$. Also, two independent solutions of the embedding
equation $\der h\dz\lambda\dz\nu_1\dz\nu_2\equiv \der h\dz(\der y-(x a_1+y a_2)\der x)\dz\der a_1\dz\der a_2\equiv
0$ are obviously $h_1=a_2$ and
$h_2=a_2$. The third independent solution of this equation can
be taken as: $\tilde{h}_1={\rm e}^{-a_2
  x}\Big(y+\frac{a_1}{a_2}x+\frac{a_1}{a_2^2}\Big)$. Thus the
embedding is given by:
$$\bbR^4\ni(x,y,a_1,a_2)\to(x,y,a_0,a_1,a_2)=(x,y,{\rm e}^{-a_2
  x}(y+\frac{a_1}{a_2}x+\frac{a_1}{a_2^2}),a_1,a_2)\in\bbR^{2+3},$$
which is a hypersurface in $\bbR^5$ with coordinates
$(x,y,a_0,a_1,a_2)$, given by 
$a_0{\rm e}^{a_2 x}=y+\frac{a_1}{a_2}x+\frac{a_1}{a_2^2}$. 
It is easy to check that, magically, 
$$y=a_0{\rm e}^{a_2x}-\frac{a_1}{a_2}x-\frac{a_1}{a_2^2}$$
is the general solution to (\ref{odex}). 

We end this example with a comment that if we had started with a
function $p(x,y,a_1,a_2)=x a_1$, then our procedure would change after
equation (\ref{y2}). In such case, the function $y^2$ would be
independent of $a_2$ everywhere, and we would end up with 
$$\lambda=\der y-y^1\der x,\quad\quad\mu=\der x,\quad\quad\nu_1=\der
y^1-\frac{y^1}{x}\der x,\quad\quad\nu_2=\der a_2.$$
Thus the $(1,1,2)$ type para-CR structure $[\lambda=\der y-x a_1\der x,\mu=\der
  x,\nu_1=\der a_1,\nu_2=\der a_2]$ would be equivalent to a Cartesian
product of the canonical type $(1,1,1)$ para-CR structure of the second
order ODE $y''=\frac{y'}{x}$, and the real line represented by $a_2$. 
\end{example}
\section{Para-CR structures of type $(n-1,1,1)$}\label{n111}
Returning to Example \ref{el1}, and using the contact 1-forms
(\ref{cf}) defining the canonical para-CR structure of type
$(1,1,n-1)$ corresponding to an ODE
$y^{(n)}=F(x,y,y',\dots,y^{(n-1)})$, we can define another para-CR
structure on the space $\mathcal J$ of $(n-1)$ jets. This para-CR
structure is of type $(n-1,1,1)$, and is obtained from the contact
forms
$(l_1=\lambda,l_2=\nu_1,\dots,l_{n-1}=\nu_{n-2},m=\mu,n=\nu_{n-1})$ as
in (\ref{cf}) by extending them to a class
$[l_1,\dots,l_{n-1},m,n]$ via
\begin{eqnarray}
&&l_i\to l_i'=a_{ij}l_j,\nonumber\\
&&m\to m'=f m+b_il_i,\nonumber\\
&&n\to n'=h n+c_i l_i\quad\quad i,j=1,\dots,n-1\nonumber,\nonumber
\end{eqnarray}
where the functions $a_{ij}$, $b_i$, $c_i$, $f$, $h$ on $\mathcal J$
satisfy $\det (a)f h\neq 0$. Since this para-CR structure has $\dim
H^+=\dim H^-=1$, the integrability conditions $[H^\pm,H^\pm]\subset
H^\pm$ are automatically satisfied here.

\begin{example}
It is instructive to examine this para-CR structure in case of
$n=3$. In such case we have 
\be l_1=\der y-y^1\der x,\quad l_2=\der
y^1-y^2\der x,\quad n=\der y^2-F(x,y,y^1,y^2)\der x, \quad m=\der
x\label{qi}
\ee and
\begin{eqnarray}
&&l_1'=a_{11} l_1+a_{12}l_2,\nonumber\\
&&l_2'=a_{21} l_1+a_{22}l_2,\label{cipa}\\
&&n'=h n+c_1l_1+c_2l_2,\nonumber\\
&&m'=f m+b_1l_1+b_2 l_2,\nonumber
\end{eqnarray}
We now consider a \emph{contact} transformation
$(x,y,y^1,y^2)\to(\bar{x},\bar{y},\bar{y}^1,\bar{y}^2)$ of the
variables of the corresponding third order ODE
$y^{(3)}=F(x,y,y',y'')$. This changes the ODE to a new form 
$\bar{y}^{(3)}=\bar{F}(\bar{x},\bar{y},\bar{y}',\bar{y}'')$. It
follows that, if we
started with this equation and calculated the corresponding forms 
$(\bar{l}_1,\bar{l}_2,\bar{m},\bar{n})$ as in (\ref{qi}), then these
forms would be expressible in terms of forms (\ref{qi}) via
\begin{eqnarray}
&&\bar{l}_1=a_{11} l_1,\nonumber\\
&&\bar{l}_2=a_{21} l_1+a_{22}l_2,\label{cip}\\
&&\bar{n}=h n+c_1l_1+c_2l_2,\nonumber\\
&&\bar{m}=f m+b_1l_1+b_2 l_2,\nonumber
\end{eqnarray}
with functions $a_{ij}$, $b_i$, $c_i$, $f$ and $h$ 
which would depend on the particular form of the contact 
transformation we considered, and which would satisfy $\det(a)f
h\neq 0$. Although transformation (\ref{cip}) seems to be more restrictive than
the one in (\ref{cipa}), it turns out that they are equivalent. Actually,
it follows that starting with a general transformation (\ref{cipa})
and forms (\ref{qi}) there is unique way of killing $a_{12}$ in
(\ref{cipa}). This is done by observing that the most general forms
$(l_1',l_2',m',n')$ from (\ref{cipa}) satisfy 
%/home/pawel/notebooks/paracr/paracr211.nb
$$\der l_1'\dz l_1'\dz l_2'=\frac{a_{12}}{fh}m'\dz n'\dz l_1'\dz l_2'.$$
Thus we can alsways normalize the transformation (\ref{cipa}) to one in
which $a_{12}=0$. This proves the following proposition.  
\end{example}
\begin{proposition}
The local geometry of the type $(2,1,1)$ para-CR structure defined in
(\ref{qi})-(\ref{cipa}) is identical to the local geometry of a
general third order ODE $y^{(3)}=F(x,y,y',y'')$ considered modulo
\emph{contact} transformation of variables.
\end{proposition}
The geometry described by the above proposition was studied by Chern
\cite{chern} in the context of ODEs, and by Tanaka \cite{tanaka} in
the context of para-CR structures. Actually Tanaka in \cite{tanaka}
showed that the natural geometry associated with an $n$-th order ODE
$y^{(n)}=F(x,y,y',\dots,y^{(n-1)})$, considered modulo \emph{contact} transformations, is the geometry of
type $(n-1,1,1)$ para-CR structures, which he called
\emph{pseudo-product} structures. 
\begin{remark}
It is interesting to note that the passage from a $(1,1,n-1)$ para-CR
structure to a type $(n-1,1,1)$ para-CR structure, in the context
of para-CR structures associated with an ODE
$y^{(n)}=F(x,y,y',\dots,y^{(n-1)})$, corresponds to the passage from
the geometry of an ODE given modulo \emph{point} transformations to
the 
geometry of an ODE given modulo \emph{contact} transformations. This
is a first instance of a more general phenomenon, which will be
discussed in Section \ref{321}.
\end{remark}
\section{Invariants}
We are interested in objects naturally associated with a given para-CR 
manifold which are not changed under (local) para-CR diffeomorphisms. We call
such 
objects (local) \emph{invariants}. Clearly the simplest invariants of a 
para-CR manifold 
$(M,[(\lambda_1,\dots,\lambda_k,\mu_1,$ $\dots,\mu_r,\nu_1,\dots,\nu_s)])$ 
are the integers $(k,r,s)$. If $k=1$, we have also another obvious
invariant. This is defined as follows:

\begin{remark}\label{vr}
Note that the canonical $(1,1,n-1)$ type 
para-CR structures corresponding to $n$-th
order ODEs satisfy $\der\lambda\dz\lambda\neq 0$ and
$\der\lambda\dz\der\lambda\dz\lambda\equiv 0$. These  conditions are  
invariant under para-CR transformations, since any such transformation brings 
$\lambda\to\lambda'=a\lambda$, with some $a\neq 0$. If we have a
general $(1,r,s)$ type para-CR structure, a simple local
invariant, is the \emph{rank} of the para-CR form $\lambda$, i.e. the integer
$t$, such that
$$\underbrace{\der\lambda\dz\dots\dz\der\lambda}_{t~\mathrm{times}}\dz\lambda\neq 0\quad\quad
{\rm and}\quad\quad\underbrace{\der\lambda\dz\dots\dz\der\lambda}_{(t+1)~\mathrm{times}}\dz\lambda\equiv 0.$$
\end{remark}

Immediately there are two questions:
\begin{itemize}
\item Are the numbers $(k,r,s)$, (or $t$ when $k=1$), the \emph{only}
  local invariants of $(M,[(\lambda_1,\dots,\lambda_k,\mu_1,\dots,\mu_r,\nu_1,$ $\dots,\nu_s)])$? 
\item And if the answer to the above question is negative, 
how does one construct the system of \emph{all} local invariants of
$(M,[(\lambda_1,\dots,\lambda_k,\mu_1,\dots,\mu_r,\nu_1,\dots,\nu_s)])$? 
\end{itemize}

It is rather obvious that the answer for the first question above is
`no'. We are thus led to discuss how to construct the invariants. We
do not make an exhaustive discussion in the following. Instead we
concentrate on low dimensional cases, producing invariants for
structures of type $(1,1,2)$ and $(1,2,3)$. These examples are
complicated enough to illustrate the basic features that the general
case can have.
\subsection{Local invariants for para-CR structures of type
  $(1,1,2)$}
In Section \ref{111} we proved that every para-CR structure of
type $(1,1,1)$ for which $\der\lambda\dz\lambda\neq 0$, is locally para-CR
equivalent to a second order ODE considered modulo point
transformation of variables. Thus all local invariants for such para-CR
structures are in one-to-one correspondence with the local invariants of 
second oder ODEs considered modulo point transformations. All such
invariants are known since the times of the classical papers of Lie
\cite{lieodes}, Tresse \cite{tresse} and Cartan \cite{cartanode}. We
refer an interested reader to the para-CR treatment of these
invariants in \cite{nurspar}. Since we will need some results about 
the $(1,1,1)$ case in the following, we quote them here for completeness. 
\subsubsection{Brief summary of the $(1,1,1)$ case} As we know (see
Proposition \ref{repp}, or the proof of Proposition \ref{vry}) every
para-CR structure $(M,[\lambda,\mu,\nu])$ of type $(1,1,1)$ with
$\der\lambda\dz\lambda\neq 0$ can be locally represented by
$$\lambda=\der y-p(x,y,a_1)\der x,\quad\quad\mu=\der
x,\quad\quad\nu=\der a_1,$$
with a function $p=p(x,y,a_1)$ of variables $(x,y,a_1)$ on
$M$ such that $p_1=\frac{\partial p}{\partial a_1}\neq 0$. Consider 
now the most general forms
$(\theta^0,\theta^1,\theta^3)\in[\lambda,\nu,\mu]$ in the class 
$[\lambda,\nu,\mu]$. They are given on $M$ by:
$$\theta^0=a \lambda,\quad\quad
\theta^1=c_1\lambda+h_{11}\nu,\quad\quad\theta^3=b\lambda+f\mu,$$
with some functions $a$, $h_{11}$, $c_1$, $f$ and $b$ 
such that $a h_{11}f\neq 0$ 
(the strange numbering of the forms will become clear in the next
section). Extending the manifold $M$ to $M\times G$, where $G$ is
parametrized by $(a,h_{11},c_1,f,b)$, we can apply Cartan's
equivalence method to find the invariants of such structures. This was
done by Cartan in \cite{cartanode}. His result adapted to our
situation is summarized in the following proposition.    
\begin{proposition}\label{psss}Every para-CR manifold $(M,[\lambda,\mu,\nu])$ of
  type $(1,1,1)$ with $\der\lambda\dz\lambda\neq 0$ uniquely defines
  an 8-dimensional manifold $P$ with a unique coframe
  $(\theta^0,\theta^1,\theta^3,$ $\Om_1,\Om_2,\Om_3,\Om_7,\Om_8)$ on it, which
  satisfies the following equations
%/home/pawel/notebooks/paracr/paracr111podobniedozdeg112.nb
\begin{eqnarray}
&&\der\theta^0=\Om_1\dz\theta^0+\theta^3\dz\theta^1\nonumber\\
&&\der\theta^1=\Om_2\dz\theta^0+\Om_3\dz\theta^1\nonumber\\
&&\der\theta^3=(\Om_1-\Om_3)\dz\theta^3+\Om_7\dz\theta^0\nonumber\\
&&\der\Om_1=2\Om_8\dz\theta^0+\Om_7\dz\theta^1-\Om_2\dz\theta^3\nonumber\\
&&\der\Om_2=(\Om_3-\Om_1)\dz\Om_2+\Om_8\dz\theta^1+K\theta^0\dz\theta^3\label{sss}\\
&&\der\Om_3=\Om_8\dz\theta^0+2\Om_7\dz\theta^1+\Om_2\dz\theta^3\nonumber\\
&&\der\Om_7=\Om_7\dz\Om_3+\Om_8\dz\theta^3+J\theta^0\dz\theta^1\nonumber\\
&&\der\Om_8=\Om_8\dz\Om_1+\Om_7\dz\Om_2+\tfrac{\partial J}{\partial \theta^3} \theta^1\dz\theta^0+\tfrac{\partial K}{\partial \theta^1}\theta^3\dz\theta^0.\nonumber
\end{eqnarray}
Here the functions $J$ and $K$ are given by:
\begin{eqnarray*}
&&(6fh_{11}^3p_1^4)~J=\\&&
-15 p_{11}^3 p_{x1}+10 p_1 p_{11} p_{111} p_{x1}
p_{x1}+15 p_1 p_{11}^2 p_{x11}-4 p_1^2 p_{111} p_{x11}+\\&&12 p_1^2 p_{11}^2 p_{y1}-15 p p_{11}^3 p_{y1}-4 p_1^3
p_{111} p_{y1}+10 p p_1 p_{11} p_{111} p_{y1}-\\&&12 p_1^3 p_{11}
p_{y11}+15 p p_1 p_{11}^2 p_{y11}-4 p p_1^2 p_{111} p_{y11}-6 p_1^2 
p_{11} p_{x111}+\\&&4
p_1^2(p_1^2-\tfrac32 p  p_{11})p_{y111}-p_1^2(1+p p_{y1}) p_{1111}+p_1^3 (p_{x1111}+p p_{y1111})
\end{eqnarray*}
and 
\begin{eqnarray*}
&&(6f^3h_{11}p_1^4)~K=\\&&
-15 p_{11} p_{x1}^3+15 p_1 p_{x1}^2 p_{x11}+10 p_1 p_{11} p_{x1} p_{xx1}-4 p_1^2 p_{x11} p_{xx1}-\\&&6 p_1^2 p_{x1} p_{xx11}-p_1^2 p_{11} p_{xxx1}+p_1^3 p_{xxx11}-2 p_1^4 p_{xxy1}-3 p p_1^2 p_{11} p_{xxy1}+\\&&3 p p_1^3 p_{xxy11}-p_1^2 p_{11} p_{x1} p_{xy}+p_1^3 p_{x11} p_{xy}-3 p_1^2 p_{11} p_x p_{xy1}+6 p_{1}^3 p_{x1} p_{xy1}+\\&&20 p p_1 p_{11} p_{x1} p_{xy1}-8 p p_1^2 p_{x11} p_{xy1}+3 p_1^3 p_x p_{xy11}-12 p p_1^2 p_{x1} p_{xy11}+2 p_1^5 p_{xyy}-\\&&4 p p_1^4 p_{xyy1}-3 p^2 p_1^2 p_{11} p_{xyy1}+3 p^2 p_1^3 p_{xyy11}+10 p_1 p_{11} p_{x1}^2 p_y-10 p_1^2 p_{x1} p_{x11} p_{y}-\\&&3 p_{1}^2 p_{11} p_{xx1} p_{y}+3 p_{1}^3 p_{xx11} p_{y}-6 p_{1}^4 p_{xy1} p_{y}-9 p p_{1}^2 p_{11} p_{xy1} p_{y}+9 p p_{1}^3 p_{xy11} p_{y}-\\&&2 p_{1}^2 p_{11} p_{x1} p_{y}^2+2 p_{1}^3 p_{x11} p_{y}^2+10 p_{1} p_{11} p_{x} p_{x1} p_{y1}-6 p_{1}^2 p_{x1}^2 p_{y1}-45 p p_{11} p_{x1}^2 p_{y1}-\\&&4 p_{1}^2 p_{x} p_{x11} p_{y1}+30 p p_{1} p_{x1} p_{x11} p_{y1}-p_{1}^2 p_{11} p_{xx} p_{y1}+2 p_{1}^3 p_{xx1} p_{y1}+\\&&10 p p_{1} p_{11} p_{xx1} p_{y1}-6 p p_{1}^2 p_{xx11} p_{y1}-2 p_{1}^4 p_{xy} p_{y1}-3 p p_{1}^2 p_{11} p_{xy} p_{y1}+\\&&10 p p_{1}^3 p_{xy1} p_{y1}+20 p^2 p_{1} p_{11} p_{xy1} p_{y1}-12 p^2 p_{1}^2 p_{xy11} p_{y1}-4 p_{1}^2 p_{11} p_{x} p_{y} p_{y1}+\\&&8 p_{1}^3 p_{x1} p_{y} p_{y1}+30 p p_{1} p_{11} p_{x1} p_{y} p_{y1}-14 p p_{1}^2 p_{x11} p_{y} p_{y1}-4 p_{1}^4 p_{y}^2 p_{y1}-\\&&6 p p_{1}^2 p_{11} p_{y}^2 p_{y1}+2 p_{1}^3 p_{x} p_{y1}^2+10 p p_{1} p_{11} p_{x} p_{y1}^2-12 p p_{1}^2 p_{x1} p_{y1}^2-\\&&45 p^2 p_{11} p_{x1} p_{y1}^2+15 p^2 p_{1} p_{x11} p_{y1}^2+10 p p_{1}^3 p_{y} p_{y1}^2+20 p^2 p_{1} p_{11} p_{y} p_{y1}^2-6 p^2 p_{1}^2 p_{y1}^3-\\&&15 p^3 p_{11} p_{y1}^3-6 p_{1}^2 p_{x} p_{x1} p_{y11}+15 p p_{1} p_{x1}^2 p_{y11}+p_{1}^3 p_{xx} p_{y11}-4 p p_{1}^2 p_{xx1} p_{y11}+\\&&3 p p_{1}^3 p_{xy} p_{y11}-8 p^2 p_{1}^2 p_{xy1} p_{y11}+4 p_{1}^3 p_{x} p_{y} p_{y11}-16 p p_{1}^2 p_{x1} p_{y} p_{y11}+6 p p_{1}^3 p_{y}^2 p_{y11}-\\&&10 p p_{1}^2 p_{x} p_{y1} p_{y11}+30 p^2 p_{1} p_{x1} p_{y1} p_{y11}-20 p^2 p_{1}^2 p_{y} p_{y1} p_{y11}+15 p^3 p_{1} p_{y1}^2 p_{y11}-\\&&2 p_{1}^4 p_{x1} p_{yy}-p p_{1}^2 p_{11} p_{x1} p_{yy}+p p_{1}^3 p_{x11} p_{yy}+4 p_{1}^5 p_{y} p_{yy}-4 p p_{1}^4 p_{y1} p_{yy}-\\&&2 p^2 p_{1}^2 p_{11} p_{y1} p_{yy}+2 p^2 p_{1}^3 p_{y11} p_{yy}-2 p_{1}^4 p_{x} p_{yy1}-3 p p_{1}^2 p_{11} p_{x} p_{yy1}+6 p p_{1}^3 p_{x1} p_{yy1}+\\&&10 p^2 p_{1} p_{11} p_{x1} p_{yy1}-4 p^2 p_{1}^2 p_{x11} p_{yy1}-8 p p_{1}^4 p_y p_{yy1}-6 p^2 p_{1}^2 p_{11} p_{y} p_{yy1}+\\&&8 p^2 p_{1}^3 p_{y1} p_{yy1}+10 p^3 p_{1} p_{11} p_{y1} p_{yy1}-4 p^3 p_{1}^2 p_{y11} p_{yy1}+3 p p_{1}^3 p_{x} p_{yy11}-\\&&6 p^2 p_{1}^2 p_{x1} p_{yy11}+6 p^2 p_{1}^3 p_{y} p_{yy11}-6 p^3 p_{1}^2 p_{y1} p_{yy11}+2 p p_{1}^5 p_{yyy}-2 p^2 p_{1}^4 p_{yyy1}-\\&&p^3 p_{1}^2 p_{11} p_{yyy1}+p^3 p_{1}^3 p_{yyy11},
\end{eqnarray*}
and $\tfrac{\partial J}{\partial \theta^3}$ and $\tfrac{\partial
  K}{\partial \theta^1}$ denote the coframe derivatives of functions
$J$ and $K$ with respect to the coframe element $\theta^3$ and
$\theta^1$, respectively.
 
Two type $(1,1,1)$ para-CR manifolds $(M,[\lambda,\mu,\nu])$ and
$(M',[\lambda',\mu',\nu'])$, with $\der\lambda\dz\lambda\neq 0$ and
$\der\lambda'\dz\lambda'\neq 0$ are locally para-CR equivalent iff
there exists a local diffeomorphism $\phi:P\to P'$, of the
corresponding $8$-manifolds $P$ and $P'$, which pulls back the coframe
$(\theta'^0,\theta'^1,\theta'^3,$
$\Om'_1,\Om'_2,\Om'_3,\Om'_7,\Om'_8)$ 
to $(\theta^0,\theta^1,\theta^3,$ $\Om_1,\Om_2,\Om_3,\Om_7,\Om_8)$.  
\end{proposition}
In particular the vanishing of each of the functions $J$ and $K$ is a
para-CR invariant property. These functions are para-CR versions of
the classical two point invariants $w_1$ and $w_2$ (see
\cite{nurspar}) of the corresponding second order ODE, which were known to Lie and
Tresse \cite{lieodes,tresse}. This proposition solves the local
equivalence problem for type $(1,1,1)$ para-CR structures: they are
either locally equivalent to $[\lambda=\der y,\mu=\der x,\nu=\der
  a_2]$, or they are described by the above proposition.
\subsubsection{The simplest relative invariant for type $(1,1,2)$}\label{sse}
Passing to the $(1,1,2)$ case we consider a para-CR structure
$(M,[\lambda,\mu,\nu_1,\nu_2])$, and since all para-CR structures with
$\der\lambda\dz\lambda\equiv 0$ are locally equivalent to
$(\bbR^{(1+1+2)},[\lambda=\der y,\mu=\der x,\nu_1=\der a_1,\nu_2=\der
  a_2])$, we will assume
$\der\lambda\dz\lambda\neq 0$ in the following. As at the begining of
Section \ref{112} we may introduce a local coordinate system
$(x,y,a_1,a_2)$ on $M$ so that the para-CR
structure is represented by 
$$\lambda=\der y-p(x,y,a_1,a_2)\der x,\quad\quad\mu=\der x,\quad\quad\nu_1=\der
a_1,\quad\quad\nu_2=\der a_2.$$ 
Here $p$ is an appropriate function $p=p(x,y,a_1,a_2)$ on $M$ which satisfies $\der x\dz\der y\dz\der
p\neq 0$. Without loss of generality we can assume in the following that
$$p_1=\frac{\partial p}{\partial a_1}\neq 0.$$ Now we introduce the most
general forms $(\lambda',\mu',\nu'_1,\nu'_2)$ from the class
$[\lambda,\mu,\nu_1,\nu_2]$. These are:
\be
\bma
\la\\
\nu_1\\
\nu_2\\
\mu
\ema\to\bma
\la'\\
\nu_1'\\
\nu_2'\\
\mu'
\ema =\bma
a&0&0&0\\
c_{1}&h_{11}&h_{12}&0\\
c_{2}&h_{21}&h_{22}&0\\
b&0&0&f
\ema
\bma
\la\\
\nu_1\\
\nu_2\\
\mu
\ema\stackrel{\rm def}{=}\bma
\theta^0\\
\theta^1\\
\theta^2\\
\theta^3
\ema .\label{uf2}
\ee
Now we are in a position to determine the first relative invariant. We
do it using Cartan's equivalence method (see e.g. \cite{olver}) in the
following steps:
%/home/pawel/notebooks/paracr/paracr112.nb
%/home/pawel/notebooks/paracr/paracr112n1.nb
\begin{enumerate}
\item \label{nmb1} We first calculate the invariant form
  $\der\theta^0\dz\theta^0$. This is given by 
$$\der\theta^0\dz\theta^0=\frac{a(h_{22}p_1-h_{21}p_2)}{f(h_{12}h_{21}-h_{11}h_{22})}\theta^0\dz\theta^1\dz\theta^3+\frac{a(h_{11}p_2-h_{12}p_1)}{f(h_{12}h_{21}-h_{11}h_{22})}\theta^0\dz\theta^2\dz\theta^3.$$
 (Here and in the following the
  partial derivatives with respect to $a_i$ are denoted by a subscript
  $i$ at the differentiated function; derivatives with respect to $x$
  and $y$ are denoted by the respective subscript $x$ or $y$.)
\item Then we impose the invariant condition
  $\der\theta^0\dz\theta^0=-\theta^0\dz\theta^1\dz\theta^3$. This is
  achieved by taking 
\be
a=
  \frac{h_{11}f}{p_1}\quad\quad{\rm and}\quad\quad h_{12}=\frac{h_{11}p_2}{p_1}.\label{fo}\ee
\item Then, on an 11-dimensional manifold $M^{(1)}$ parametrized by
  $(x,y,a_1,a_2,c_1,c_2,$ $h_{11},h_{21},h_{22},b,f)$,  we introduce a 
1-form $\Omega_1$ so that we have 
$$\der\theta^0=\Om_1\dz\theta^0+\theta^3\dz\theta^1.$$
The form $\Omega_1$ is given by 
\be
\Om_1=\frac{\der f}{f}+\frac{\der h_{11}}{h_{11}}-\frac{\der
  p_1}{p_1}+\frac{b p_1}{h_{11} f}\theta^1+\frac{h_{11}
  p_y-c_1 p_1}{h_{11} f}\theta^3+f_0 \theta^0,\label{f0}\ee
where $f_0$ is an additional function on $M^{(1)}$. 
\item\label{nbm} It
  is easy to check that at this stage we have 
$$\der\theta^1\dz\theta^0\dz\theta^1=
-I\frac{h_{11}}{p_1(h_{22}p_1-h_{21}p_2)f}\theta^0\dz\theta^1\dz\theta^2\dz\theta^3,$$
where $$I=p_1(p_{x2}+p p_{y2})-p_2(p_{x1}+p p_{y1}).$$ 
Comparing this with 
$$\det \bma
a&0&0&0\\
c_{1}&h_{11}&h_{12}&0\\
c_{2}&h_{21}&h_{22}&0\\
b&0&0&f
\ema=(h_{22}p_1-h_{21}p_2)f^2h_{11}^2p_1^{-1}\neq 0,$$
we see that the condition that $I$ vanishes or not is a para-CR
invariant property of the class $[\lambda,\mu,\nu_1,\nu_2]$. This
shows that $I$ is a \emph{relative invariant} for the considered para-CR
structure. 
\item For
example if $I\neq 0$ in the considered neighborhood, we can
normalize $\der\theta^1\dz\theta^0\dz\theta^1$ to
$\der\theta^1\dz\theta^0\dz\theta^1=-\theta^0\dz\theta^1\dz\theta^2\dz\theta^3$,
by choosing 
$$h_{11}=\frac{(h_{22}p_1-h_{21}p_2)p_1f}{I}.$$
Such a normalization is obviously impossible if $I=0$ in the
considered region. 
\item It can be checked that it is this invariant that distinguishes
  between the $(1,1,2)$ para-CR structures that correspond to the
  extension of $(1,1,1)$ type structures by $\bbR$ and the type
  $(1,1,2)$ para-CR structures equivalent to the canonical para-CR
  structures corresponding to third order ODEs.
\end{enumerate}
\subsubsection{Branch $I\neq 0$} Actually, further application of Cartan's equivalence method proves 
the following theorem.
\begin{theorem}\label{53}
Every type $(1,1,2)$ para-CR structure $(M,[\lambda,\mu,\nu_1,\nu_2])$
with $\der \lambda\dz\lambda\neq 0$ for which the invariant $I$ is non
vanishing, is
locally para-CR equivalent to a canonical para-CR structure of a
certain 
point equivalence class of 3-rd order ODEs $y'''=F(x,y,y',y'')$.
\end{theorem}
In particular, if $I\neq 0$, all the local invariants of such 
para-CR structures are identical with the local
\emph{point} invariants of the corresponding point equivalence classes of 3rd order
ODEs. For example the lowest order relative invariant, next after
$I$, is the W\"unschmann invariant \cite{wun} of the corresponding class of
ODEs. This can be written explicitly in terms of the function
$p=p(x,y,a_1,a_2)$ used above. Although we calculated this invariant in
terms of $p$ we do not display it here. It is given by quite a lengthy
and complicated expression in terms of $p$ and its derivatives up to
the 5th order. 

The above proposition enables us to find the para-CR
structures with $I\neq 0$ and large symmetry groups. Since third order
ODEs with large symmetry groups of point symmetries are
classified in \cite{godphd,nurgod1}, we know that 
such para-CR manifolds have a maximal group of para-CR symmetries of
dimension \emph{seven}. They are locally para-CR equivalent to the
para-CR structure corresponding to the point equivalent class of
the simple equation $y'''=0$. The $I\neq 0$ para-CR structures with a 
group of symmetries of dimension $6$, $5$ and $4$ are also easily 
obtained from the results of \cite{godphd,nurgod1}. We have the 
following proposition.
\begin{proposition}
All homogeneous type $(1,1,2)$ para-CR structures are locally para-CR
equivalent to the canonical para-CR structure of the following point
equivalent classes of 3rd order ODEs:
\begin{itemize}
\item $y'''=0$; in this case the symmetry algebra is
  $\coa(2,1)\oplus\bbR^3$ of dimension 7;
\item $y'''=\tfrac32\frac{(y'')^2}{y'}$; symmetry algebra $\oa(2,2)$ of
dimension 6;
\item $y'''=\frac{3(y'')^2y'}{1+(y')^2}$; symmetry algebra $\oa(4)$ of
  dimension 6;
\item $y'''=-2\mu y'+y$; each $\mu\in\bbR$ defines a nonequivalent
  para-CR structure with a 5-dimensional symmetry algebra, with
  generators $V_i$ satisfying $[V_1,V_4]=-\mu V_2+V_3$,
  $[V_1,V_5]=V_1$, $[V_2,V_4]=V_1-\mu V_3$, $[V_2,V_5]=V_2$
  $[V_3,V_4]=V_2$, $[V_3,V_5]=V_3$;
\item $y'''=(y'')^3$; symmetry algebra of dimension 4 with generators
  $V_i$ satisfying $[V_1,V_4]=2V_1$, $[V_2,V_4]= \tfrac43 V_2$,
  $[V_2,V_3]=V_1$, $[V_3,V_4]=-\tfrac23 V_3$; 
\item $y'''=\mu\frac{(y'')^2}{y'}$; here each $\mu>\tfrac32$ such that
  $\mu\neq 3$, defines a nonequivalent para-CR structure having a 4-dimensional symmetry algebra,
  with generators $V_i$ satisfying $[V_1,V_2]=V_1$, $[V_3,V_4]=V_3$;
\item $y'''=\frac{3y'+\mu}{1+(y')^2}$; for each $\mu>0$ we have a
  nonequivalent para-CR structure with a 4-dimensional symmetry algebra;
  its generators $V_i$ satisfy $[V_1,V_2]=V_3$, $[V_3,V_1]=V_2$,
    $[V_3,V_4]=V_3$, $[V_2,V_4]=V_2$. 
\end{itemize} 
\end{proposition}

\subsubsection{Branch $I\equiv 0$} This case is a bit easier to
describe explicitly than the above $I\neq 0$ case. Thus we choose
this case to present all the details of constructing invariants for
such para-CR structures, rather then those with $I\neq 0$. 

When constructing these invariants we proceed as follows: 

%/home/pawel/notebooks/paracr/paracr112n1nii0.nb
Starting with the defining forms $(l=\der y-p\der x,n_1=\der
a_1,n_2=\der a_2, m=\der x)$ as in (\ref{qi}), for which the function
$p=p(x,y,a_1,a_2)$ satisfies $$p_1\neq 0\quad\quad{\rm and}\quad\quad
I\equiv 0,$$ we consider the most general forms
$(\theta^0,\theta^1,\theta^2,\theta^3)$ from the class
$[l,n_1,n_2, m]$ as in (\ref{uf2}). Then we repeat the entire Cartan's
procedure for these forms we performed in Section \ref{sse} from item
(\ref{nmb1}) up to item (\ref{nbm}). After this we have forms
$\theta^0$ and $\theta^1$ normalized so that
$$\der\theta^0=\Om_1\dz\theta^0+\theta^3\dz\theta^1$$
and 
$$\der\theta^1\dz\theta^0\dz\theta^1=0.$$ 
This second equation holds since we assumed that $$I\equiv 0.$$ The form
$\Om_1$ is given by (\ref{f0}), and the normalizations for $a$ and
$h_{12}$ are as in (\ref{fo}). Continuing with Cartan's equivalence
method we now make the following steps:
\begin{itemize}
\item First we introduce forms $\Om_2$ and $\Om_3$ so that the form
  $\theta^1$ satisfies:
 $$\der\theta^1=\Om_2\dz\theta^0+\Om_3\dz\theta^1.$$
This defines forms $\Om_2$ and $\Om_3$ to be:
\begin{eqnarray*}
&&\Om_2=\frac{p_1\der c_1}{fh_{11}}-\frac{c_1p_1\der
  h_{11}}{fh_{11}^2}+\frac{c_1p(p_1p_{12}-p_{11}p_2)+h_{11}(p_1p_{x2}-p_2p_{x1})}{fh_{11}p(h_{22}p_1-h_{21}p_2)}\theta^2-\\&&
\frac{c_1p_1(c_1p_1-h_{11}p_y)}{f^2h_{11}^2}\theta^3+c_{10}\theta^0+c_{11}\theta^1,\\~\\
&&\Om_3=\der\log(h_{11})+\\&&\Big(\frac{c_{11}f^2h_{11}^2-bc_1p_1^2}{f^2h_{11}^2}+\frac{c_2p(p_1p_{12}-p_{11}p_2)+h_{21}(p_1p_{x2}-p_2p_{x1})}{fh_{11}p(h_{22}p_1-h_{21}p_2)}\Big)\theta^0+\\&&
\frac{p_{11}p_2-p_{12}p_1}{p_1(h_{22}p_1-h_{21}p_2)}\theta^2+\frac{c_1p_1}{fh_{11}}\theta^3+h_{111}\theta^1.
\end{eqnarray*}
As we see the forms $\Omega_2$ and $\Omega_3$ are defined modulo the
terms $\theta^0$ and $\theta^1$ (the form $\Om_2$), and $\theta^1$ (the
form $\Om_3$),
respectively. Thus to write them down in full generality one has to introduce additional
parameters $c_{10}$, $c_{11}$ and $h_{111}$.
\item At the next step we introduce forms $\Om_4$, $\Om_5$ and $\Om_6$
  such that the form $\theta^2$ satisfies
$$\der\theta^2=\Om_6\dz\theta^0-\Om_5\dz\theta^1+\Om_4\dz\theta^2.$$
These forms are defined as follows:
\begin{eqnarray*}
&&\Om_4=\frac{p_2\der h_{21}}{h_{21}p_2-h_{22}p_1}+\frac{p_1\der
    h_{22}}{h_{22}p_1-h_{21}p_2}+h_{220}\theta^0+h_{221}\theta^1+h_{222}\theta^2\\
&&\Om_5=-\frac{\der h_{21}}{h_{11}}+\frac{h_{21}}{h_{11}}\Om_4-\frac{(h_{11}h_{221}+h_{21}h_{222})}{h_{11}}\theta^2-\frac{c_2p_1}{fh_{11}}\theta^3+h_{210}\theta^0+h_{211}\theta^1\\
&&\Om_6=\frac{p_1}{fh_{11}}\der
  c_2-\frac{c_2p_1}{fh_{11}}\Om_4+\frac{c_1p_1}{fh_{11}}\Om_5-\\&&
\frac{f^2h_{11}(h_{11}h_{210}+h_{21}h_{220})+p_1\big(c_1f(h_{11}h_{211}+h_{21}h_{221})-c_2(bp_1+fh_{11}h_{221})\big)}{f^2h_{11}^2}\theta^1+\\&&
\Big(h_{220}+\frac{p_1(c_1h_{221}+c_2h_{222})}{fh_{11}}\Big)\theta^2+\frac{c_2p_1p_y}{f^2h_{11}}\theta^3-c_{20}\theta^0. 
\end{eqnarray*}
Here we had to introduce new parameters $h_{220}$, $h_{221}$,
$h_{222}$, $h_{210}$, $h_{211}$ and $c_{20}$, which take care of 
the undefined terms in the expressions for $\Om_4$, $\Om_5$, and $\Om_6$.
\item Analysing $\der\theta^3$ we first observe that
\begin{eqnarray*}
&&\der\theta^3\dz\theta^0=(\Om_3-\Om_1)\dz\theta^0\dz\theta^3+\\&&\frac{fh_{11}h_{111}(h_{22}p_1-h_{21}p_2)+h_{22}(fp_{11}-2bp_1^2)+h_{21}(2bp_1p_2-fp_{12})}{fh_{11}(h_{22}p_1-h_{21}p_2)}\theta^0\dz\theta^1\dz\theta^3.\end{eqnarray*}
This enables us to fix $h_{111}$:
$$h_{111}=\frac{h_{22}(2bp_1^2-fp_{11})+h_{21}(fp_{12}-2bp_1p_2)}{fh_{11}(h_{22}p_1-h_{21}p_2)}.$$
\item After this normalization an introduction of a form
\begin{eqnarray*}
&&\Om_7=\frac{p_1\der
    b}{fh_{11}}+\frac{bp_1}{fh_{11}}(\Om_3-\Om_1)+b_0\theta^0-\Big(c_{11}-f_0+\frac{2bc_1p_1^2}{f^2h_{11}^2}+\\&&\frac{b}{f^2h_{11}}(p_{x1}-2p_1p_y+pp_{y1})+\frac{c_2p(p_1p_{12}-p_2p_{11})+h_{21}(p_1p_{x2}-p_2p_{x1})}{fh_{11}p(h_{22}p_1-h_{21}p_2)}\Big)\theta^3,
\end{eqnarray*}
brings $\der\theta^3$ into the form:
$$\der\theta^3=\Om_7\dz\theta^0+(\Om_1-\Om_3)\dz\theta^3.$$
Again we had to introduce a new parameter which we denoted by $b_0$ here.
\end{itemize}
Summarizing our efforts in this section so far, we conclude that the
invariant forms $\theta^0,\theta^1,\theta^2,\theta^3$ of a
para-CR structure with $I\equiv 0$ can be gauged in such a way that they
have the following differentials:
\begin{eqnarray}
&&\der\theta^0=\Om_1\dz\theta^0+\theta^3\dz\theta^1\nonumber\\
&&\der\theta^1=\Om_2\dz\theta^0+\Om_3\dz\theta^1\label{ssya}\\
&&\der\theta^2=\Om_6\dz\theta^0-\Om_5\dz\theta^1+\Om_4\dz\theta^2\nonumber\\
&&\der\theta^3=\Om_7\dz\theta^0+(\Om_1-\Om_3)\dz\theta^3.\nonumber
\end{eqnarray}
Now we pass to the analysis of this system in terms of Cartan's
characters and Cartan's test for the involutivity 
(see \cite{olver}, pp. 350-355 for definitions; for a one page
description of the procedure see e.g. \cite{davepaw}, pp. 4066-4067). Since the forms
$\theta^i$ are given up to the action of the residual
$(r=7)$-dimensional group
parametrized by $f, b, c_1, c_2, h_{11}, h_{21}, h_{22}$, we easily
calculate the four Cartan characters associated to
this system. They are $s_1'=4$,
$s_2'=2$, $s_3'=1$, $s_4'=0$. Moreover, since the new forms
$\Om_1,\Om_2,\Om_3, \Om_4, \Om_5, \Om_6,\Om_7$ transversal to the
respective residual group directions $\partial_f$, $\partial_{c_1}$, 
$\partial_{h_{11}}$, $\partial_{h_{22}}$, $\partial_{h_{21}}$,
$\partial_{c_2}$, $\partial_b$, are determined modulo $r^{(1)}=10$
parameters $f_0$, $c_{10}$,
$c_{11}$, $h_{220}$, $h_{221}$,
$h_{222}$, $h_{210}$, $h_{211}$, $c_{20}$, $b_0$, we have 
$$1 s_1'+2 s_2'+3 s_3'+4 s_4'=11\neq 10=r^{(1)}.$$
Thus the system (\ref{ssya}) is \emph{not} involutive, and has to be
prolonged. Calculating
$\der\Om_1$, $\der\Om_2$, $\der\Om_3$, $\der\Om_7$ we fix $c_{10}$,
$c_{11}$ and $b_0$ in such a way that the forms $\Om_1$, $\Om_2$,
$\Om_3$ and $\Om_7$ satisfy:
\begin{eqnarray}
&&\der\Om_1=2\Om_8\dz\theta^0+\Om_7\dz\theta^1-\Om_2\dz\theta^3\nonumber\\
&&\der\Om_2=\Om_2\dz(\Om_1-\Om_3)+\Om_8\dz\theta^1+K\theta^0\dz\theta^3\label{ssyb}\\
&&\der\Om_3=\Om_8\dz\theta^0+2\Om_7\dz\theta^1+\Om_2\dz\theta^3\nonumber\\
&&\der\Om_7=\Om_7\dz\Om_3+\Om_8\dz\theta^3+J\theta^0\dz\theta^1.\nonumber
\end{eqnarray}
Here the form $\Om_8$ and functions $J$ and $K$ are totally
determined by the above equations. The form $\Om_8$ is given by:
\begin{eqnarray*}
&&2\Om_8=\der
  f_0+f_0\Om_1+(\frac{bp_1}{fh_{11}}+\frac{h_{21}p_{12}-h_{22}p_{11}}{h_{11}(h_{22}p_1-h_{21}p_2)})\Om_2-\frac{p_1p_{12}-p_{2}p_{11}}{p_1(h_{22}p_1-h_{21}p_2)}\Om_6-\\
&&\frac{c_1p_1^2+h_{11}p_{x1}-h_{11}p_1p_y+h_{11}pp_{y1}}{fh_{11}p_1}\Om_7+(\dots)\theta^0+(\dots)\theta^1+(\dots)\theta^2+(\dots)\theta^3,\end{eqnarray*}
where we skip writing down very compliceted, yet still \emph{totally
determined}, coefficients at the terms $\theta^0$, $\theta^1$,
$\theta^2$ and $\theta^3$. It turns out, \emph{and this is the result of our
calculations}, that the functions $J$ and $K$
are given by \emph{the same formulae} as in Proposition
\ref{psss}. This is not surprising, if one notices the identical forms
of the systems (\ref{ssya})-(\ref{ssyb}) and (\ref{sss}) with the
equation for $\der\theta^2$ and $\der\Om_8$ removed.
Actually, after calculating $\der\Om_8$ in the present situation, we
get 
\be
\der\Om_8=\Om_8\dz\Om_1+\Om_7\dz\Om_2+\tfrac{\partial J}{\partial
  \theta^3} \theta^1\dz\theta^0+\tfrac{\partial K}{\partial
  \theta^1}\theta^3\dz\theta^0,\label{ssyc}\ee
which again agrees with the system (\ref{sss}). 
Now we are ready to perform the Cartan analysis of the the composed 
system (\ref{ssya})-(\ref{ssyc}). We have here $m=4+5$ differentials
$\der\theta^0$, $\der\theta^1$, $\der\theta^2$, $\der\theta^3$,
$\der\Omega_1$, $\der\Omega_2$, $\der\Omega_3$, $\der\Omega_7$,
$\der\Omega_8$, of the forms $\theta^0$, $\theta^1$, $\theta^2$,
$\theta^3$, $\Omega_1$, $\Omega_2$, $\Omega_3$, $\Omega_4$,
$\Omega_5$, which are given modulo the $(r=3)$-dimensional residual
group parametrized by $c_2$, $h_{21}$ and $h_{22}$. The new forms
$\Om_4$, $\Om_5$ and $\Om_6$, 
transversal to the respective vector fields 
$\partial_{h_{22}}$, $\partial_{h_{21}}$ and $\partial_{c_2}$ are
given up to $r^{(1)}=6$ parameters $h_{220}$, $h_{221}$, $h_{222}$
($\Om_4$), $h_{210}$, $h_{211}$ ($\Om_5$), and $c_{20}$ ($\Om_6$). Simple linear algebra gives the following Cartan's characters of the
system (\ref{ssya})-(\ref{ssyc}): $s_1'=s_2'=s_3'=1$, $s_i'=0$ for all
$i=4,\dots 9$. Thus for this system we have
$$1s_1'+2s_2'+3s_3+4s_4'+5s_5'+6s_6'+7s_7'+8s_8'+9s_9'=6=r^{(1)},$$
and, hence, the system \emph{is} involutive. This result together with 
Cartan's Theorem 11.16, \cite{olver}, p. 367, tells us that there is 
no para-CR invariant information encoded in the forms 
$\Omega_4$, $\Om_5$ and $\Om_6$. Hence we can take them in the most simple representation
$\Om_4=\Om_5=\Om_6\equiv 0$. (Note that this can be achieved by
setting
$c_2=h_{21}=h_{22}=c_{20}=h_{210}=h_{211}=h_{220}=h_{221}=h_{222}=0$,
$\theta^2=\der a_2$. Cartan's theorem says also that we can
do it in many ways. Since we are in the involutive case, the local
group of para-CR symmetries is infinite dimensional; it depends on $s_{k=3}'=1$
arbitrary real function of $k=3$ variables.) Concluding we have the
following theorem. 
\begin{theorem}\label{54}
All type $(1,1,2)$ para-CR structures $(M,[\lambda,\mu,\nu_1,\nu_2])$
with $\der\lambda\dz\lambda\neq 0$, and with the invariant $I\equiv
0$, 
are locally equivalent to one of the para-CR structures $(M,[\lambda=\der
  y-p\der x, \nu_1=\der p-Q(x,y,p)\der x, \nu_2=\der a_2,\mu=\der
  x])$. Thus they are obtained by extending by
$\nu_2=\der a_2$ the
type $(1,1,1)$ para-CR structure defined by $[\lambda=\der
  y-p\der x, \nu_1=\der p-Q(x,y,p)\der x,\mu=\der
  x]$. All local invariants of such $(M,[\lambda,\mu,\nu_1,\nu_2])$ 
are given by the point invariants of
the corresponding point equivalence class of second order ODEs represented by
$y''=Q(x,y,y'')$. 
\end{theorem}
It is convenient to introduce the following definition.
\begin{definition}
A type $(1,1,2)$ para-CR manifold $(M,[\lambda,\mu,\nu_1,\nu_2])$
with $\der\lambda\dz\lambda\neq 0$ is regular if the 
invariant $I$ is either not equal to
zero in $M$ or it is zero everywhere in $M$.
\end{definition}

Now comparing Theorems \ref{53} and \ref{54} we obtain:
\begin{corollary}\label{cll}
All regular type $(1,1,2)$ para-CR manifolds 
$(M,[\lambda,\mu,\nu_1,\nu_2])$ with $\der\lambda\dz\lambda\neq 0$ are
locally equivalent either to canonical para-CR structures of point
equivalence classes of 3rd order ODEs (if $I\neq 0$), or to the
trivial extensions of the cananical para-CR structures of point
equivalent classes of 2nd order ODEs (if $I\equiv 0$).  
\end{corollary}

\subsection{Local invariants for para-CR structures of type
  $(1,1,n-1)$}
We believe that the situation described in Cotrollary \ref{cll} 
is typical for any regular type $(1,1,n-1)$ para-CR 
structures $(M,[\lambda,\mu,\nu_\alpha])$ with
$\der\lambda\dz\lambda\neq 0$ and any $n> 3$. By this we mean the
following. In the \emph{generic} case, such para-CR structures should be locally
equivalent to the canonical para-CR structures associated with point
equivalent classes of $n$th order ODEs. This generic case should be
distinguished by the simultaneous \emph{non}vanishing of a finite number
$t$ of relative invariants $(I_1,\dots,I_t)$, generalizing our
invariant $I$. These invariants should have some hierarchical
structure, so that if all invariants above some level, say $n_0$, in the
hierarchy identically vanish, then the para-CR structure is a trivial
extension of a canonical para-CR structure of type $(1,1,n-n_0-1)$, 
by adding $n_0$ forms $\nu_{n-1}=\der a_{n-1}$, $\dots$,
$\nu_{n-n_0}=\der a_{n-n_0}$ to the canonical contact forms
$[\la,\mu,\nu_1,\dots,\nu_{n-n_0-1}]$. Proving or disproving our
belief goes beyond this article.

%At this stage, after determination of the lowest order invariant $I$,
%Cartan's equivalence method barnches: two para-CR structures with
%$I=0$ and $I\neq 0$ are inequivalent, and to obtain further invariants
%one one has to use it in each of the cases separately.
\section{Relations with other differential equations}
Given a para-CR structure of type $(k,r,s)$ we consider its local
embedding in $\bbR^{(k+r)+(k+s)}$, as
in Theorem \ref{mbe}. The obtained codimension-$k$ submanifold
$\Sigma$ we intend to interprete as a general
solution of a certain system of differential equations. We know how to
do it in the case of para-CR structures of type $(1,1,n-1)$: in this case
$\Sigma$ describes the general solution of an $n$th
order ODE considered modulo point transformations of variables. In the
case 
of a general $(k,r,s)$ we expect that $\Sigma$ corresponds to the 
general solution of a \emph{system of ODEs}, or more
generally, to the general solution of a \emph{system of PDEs of finite
type}. 
\subsection{Systems of ODEs} 
Given a system of first order ODEs
\be
\frac{\der y^i}{\der x}=F^i(x,y^1,\dots,y^n),\quad\quad i=1,2,\dots n,
\ee
we consider its general solution 
$$y^i=\psi^i(x,a_0,a_1,\dots,a_{n-1}),\quad\quad i=1,2,\dots n,$$
where the constants $a_\mu$, $\mu=0,1,\dots,n-1$, are the constants of 
integration. This defines a codimension $n$ submanifold 
$$\Sigma=\{\bbR^{(n+1,n)}\ni(x,y^1,\dots,y^n,a_0,\dots,a_{n-1})~|~y^i=\psi^i\}$$
in $\bbR^{(n+1,n)}$, which aquires a para-CR structure from the
split $2n+1=(n+1)+n$ in the ambient  $\bbR^{(n+1,n)}$, given by the
linear operator 
$\kappa(\partial_\mu)=-\partial_\mu$, $\kappa(\partial_x)=\partial_x$,
$\kappa(\partial_{y^i})=\partial_{y^i}$. Interestingly
this para-CR structure is of type $(n,1,0)$. 

Indeed, the tangent space
${\rm T}\Sigma$ 
to $\Sigma$ is spanned by
\begin{eqnarray*}
&&X=\partial_x+\psi^i_x\partial_{y^i}\\
&&Z_\mu=\partial_\mu+\psi^i_\mu\partial_{y^i}.
\end{eqnarray*} 
Since $\kappa(Z_\mu)\cap{\rm T}\Sigma=\{0\}$, for all
$\mu=0,\dots,n-1$, and $\kappa(X)=X$, then  
$\kappa({\rm T}\Sigma)\cap{\rm T}\Sigma=H^+={\rm Span}(X)$, and
the $k(=n)$ codimensions of the $(n,1,0)$-type para-CR structure on
$\Sigma$ are 
spanned by the $n$ vectors $Z_\mu$.

Hence a typical representantive of para-CR structures of type $(n,1,0)$
is a \emph{system of $n$ first order ODEs for $n$ scalar functions of
  one variable, considered modulo point transformations of the
  variables}. The study of invariants of such para-CR structures, as well
as para-CR structures representing systems of ODEs of higher orders, will be 
performed elsewhere.

\subsection{PDEs of finite type}
Recall that the finite type property of a system of PDEs means that its 
most general solution depends on a finite number of
parameters. Instead of studying the para-CR structures associated with the most
general PDEs of finite type, in the next few sections we will study the 
para-CR structures of type $(1,2,3)$ and $(3,2,1)$. They include, 
as the simplest
example, the para-CR structure
corresponding to $z_{xx}=0$ $\&$ $z_{yy}=0$, i.e. a system of two PDEs for one real function $z=z(x,y)$ of two real
variables $x$ and $y$, with the general solution $z=a_0+a_1x+a_2y+a_3
xy$, depending on four real parameters $a_0$, $a_1$, $a_2$ and $a_3$.
Generalization of this example to the finite type PDEs of the form 
$z_{xx}=R(x,y,z,z_x,z_y,z_{xy})$ $\&$
$z_{yy}=T(x,y,z,z_x,z_y,z_{xy})$, provides examples of $(1,2,3)$ and
$(3,2,1)$ type
para-CR structures with very nice properties.

\section{Para-CR structures of type $(1,2,3)$}\label{buur}
\subsection{The flat model}\label{fm1}
Consider a pair of second order
PDEs 
\be z_{xx}=0\quad\quad\quad\&\quad\quad\quad z_{yy}=0,\label{et}\ee
for a real function $z=z(x,y)$ of two real variables $x$ and $y$. The
general solution for this system is clearly
\be z=a_0+a_1x+a_2y+a_3 xy.\label{ets}\ee
This means that the solution space of this system 
is 4-dimensional, and that its points are parametrized by 
${\bf a}=(a_0,a_1,a_2,a_3)\in \bbR^4$.
Thus we have here a 
generically embedded hypersurface 
$$\Sigma=\{\bbR^7\ni (x,y,z,a_0,a_1,a_2,a_3)~|~z=a_0+a_1x+a_2y+a_3
xy\},$$
in the `correspondence space' 
$\bbR^7=\bbR^3\times\bbR^4$, with the respective 
coordinates $(x,y,z)$ and $(a_0,a_1,a_2,a_3)$. 
The linear map $\kappa: \bbR^7\to\bbR^7$, such that
$$\kappa(x,y,z,a_0,a_1,a_2,a_3)=(x,y,z,-a_0,-a_1,-a_2,-a_3),$$ induces
a para-CR-structure of type $(1,2,3)$ on $\Sigma$. Indeed, the tangent
space to $\Sigma$ is spanned by 
\begin{eqnarray*}
&&X_1=a_1\partial_y-a_2\partial_x,\quad\quad X_2=a_2\partial_z+\partial_y,\\
&&Y_1=x\partial_0-\partial_1,\quad\quad
Y_2=y\partial_1-x\partial_2,\quad\quad Y_3=x\partial_2-\partial_3,\\
&&Z=\partial_z+\partial_0,
\end{eqnarray*}
and we have a $(k,r,s)=(1,2,3)$-type para-CR structure, with 
$k=1$ corresponding to $\Span(Z)$, $r=2$ corresponding to 
the eigenspace $H^+=\Span(X_1,X_2)$, and $s=3$ corresponding to the
eigenspace $H^-=\Span(Y_1,Y_2,Y_3)$. Obviously $H=H^+\oplus
H^-$. Any diffeomorphism of $\bbR^7$ of the form  
$$\Phi(x,y,z,a_i)=(\bar{x}(x,y,z),\bar{y}(x,y,z),\bar{z}(x,y,z),\bar{a}_i(a_j))$$ 
is, on the one hand, a para-CR diffeomorphism of the para-CR manifold
$\Sigma$, and on the other hand, can be interpreted as coming from a 
point transformation of the variables of the system (\ref{et}).   

Dually this para-CR manifold is defined on $\Sigma$ in terms of the forms
%\edz{from now on $\mu$s and $\nu$s should be interchanged. I'll do it
%  after writing up the entire business} 
\begin{eqnarray}
&&\lambda=\der a_0+x\der a_1+y\der a_2+xy\der a_3\nonumber\\
&&\mu_1=\der x\nonumber\\
&&\mu_2=\der y,\label{pf1}\\
&&\nu_1=\der a_1\nonumber\\
&&\nu_2=\der a_2\nonumber\\
&&\nu_3=\der a_3\nonumber
\end{eqnarray}
given up to the transformation
\be
\bma
\la\\
\nu_1\\
\nu_2\\
\nu_3\\
\mu_1\\
\mu_2
\ema\to\bma
\la'\\
\nu_1'\\
\nu_2'\\
\nu_3'\\
\mu_1'\\
\mu_2'
\ema =\bma
a&0&0&0&0&0\\
b_{1}&f_{11}&f_{12}&f_{13}&0&0\\
b_{2}&f_{21}&f_{22}&f_{23}&0&0\\
b_{3}&f_{31}&f_{32}&f_{33}&0&0\\
c_{1}&0&0&0&h_{11}&h_{12}\\
c_{2}&0&0&0&h_{21}&h_{22}
\ema
\bma
\la\\
\nu_1\\
\nu_2\\
\nu_3\\
\mu_1\\
\mu_2
\ema\stackrel{\rm def}{=}\bma
\theta^4\\
\theta^1\\
\theta^2\\
\theta^3\\
\Om_3\\
\Om_2
\ema .\label{pf2}
\ee
In this formulation the question of local equivalence of a given
para-CR structure of type $(1,2,3)$ to the one defined by
(\ref{pf1})-(\ref{pf2}) can be solved by using \emph{Cartan's equivalence
method}, see e.g. \cite{olver}. Using it we get the following theorem.
\begin{theorem}\label{pfa}
The para-CR structure (\ref{pf1})-(\ref{pf2}) defines a \emph{unique} 
11-dimensional manifold $P$ on which the forms
$(\theta^1,\theta^2,\theta^3,\theta^4,\Om_2,\Om_3)$, as defined
in (\ref{pf2}), can be supplemented by the unique 1-forms
$(\Om_1,\Om_4, \Om_5,\Om_6,A)$ in such a way 
that the eleven 1-forms 
$(\theta^i,\Om_\mu, A)$, $i=1,2,3,4$, $\mu=1,2,3,4,5,6$, constitute a
coframe on $P$, and that they satisfy the exterior
differential system
\begin{eqnarray}
&&\der \theta^i+\Gamma^i_{~j}\dz\theta^j=0\label{sy1}\\
&&\der \Gamma^i_{~j}+\Gamma^i_{~k}\dz\Gamma^k_{~j}=0\label{sy11},
\end{eqnarray}  
with
\be\Gamma^i_{~j}=g^{ik}\Gamma_{kj},\quad\quad
\Gamma_{ij}=\Gamma_{[ij]}+\tfrac12 A g_{ij},\label{sy2}\ee
where
\be
g^{ik}=g_{ik}=\bma 0&1&0&0\\1&0&0&0\\0&0&0&1\\0&0&1&0
\ema,\label{sy3}\ee
and
\be
\Gamma_{[ij]}=\bma 0&\Om_1&\Om_2&\Om_4\\-\Om_1&0&\Om_3&\Om_5\\-\Om_2&-\Om_3&0&\Om_6\\-\Om_4&-\Om_5&-\Om_6&0\ema.\label{sy4}\ee
Moreover, if $(\bar{\Sigma}, [(\bar{\la},\bar{\mu}_1,\bar{\mu}_2,\bar{\nu}_1,\bar{\nu}_2,\bar{\nu}_3)])$ is an
arbitrary 6-dimensional para-CR structure of type $(1,2,3)$, then it is locally
para-CR-equivalent to the para-CR structure (\ref{pf1})-(\ref{pf2}) if
and only if its corresponding forms   
$$\bma
\theta^4\\
\theta^1\\
\theta^2\\
\theta^3\\
\Om_3\\
\Om_2\ema=\bma
\bar{a}&0&0&0&0&0\\
\bar{b}_{1}&\bar{f}_{11}&\bar{f}_{12}&\bar{f}_{13}&0&0\\
\bar{b}_{2}&\bar{f}_{21}&\bar{f}_{22}&\bar{f}_{23}&0&0\\
\bar{b}_{3}&\bar{f}_{31}&\bar{f}_{32}&\bar{f}_{33}&0&0\\
\bar{c}_{1}&0&0&0&\bar{h}_{11}&\bar{h}_{12}\\
\bar{c}_{2}&0&0&0&\bar{h}_{21}&\bar{h}_{22}
\ema=\bma
\bar{\la}\\
\bar{\nu}_1\\
\bar{\nu}_2\\
\bar{\nu}_3\\
\bar{\mu}_1\\
\bar{\mu}_2
\ema$$
can be suplemented by five 1-forms $(\Om_1,\Om_4,\Om_5,\Om_6,A)$ in
such a way that on some 11-dimensional manifold $\bar{P}$ they satisfy
the exterior differential system (\ref{sy1})-(\ref{sy4}). 
\end{theorem}
\begin{proof}
The proof of this fact is a standard application of Cartan's method of
equivalence. It requires massive calculations to show that
the 1-forms (\ref{pf1})-(\ref{pf2}) can be \emph{uniquely} brought to
the form, in
which they satisfy (\ref{sy1})-(\ref{sy4}) with \emph{unique}
$(\Om_1,\Om_4,\Om_5,\Om_6,A)$. Actually Cartan's method of equivalence
constructs the manifold $P$ with a natural 
parametrization of $P$ by $(x,y,a_0,a_1,a_2,a_3,a,f_{11},f_{22},f_{31},f_{32})$, and 
gives, in an algorithmic
way, the explicit formulae for the coframe 1-forms 
$(\theta^i,\Om_\mu, A)$, $i=1,2,3,4$, $\mu=1,2,3,4,5,6$, which
correspond to (\ref{pf1})-(\ref{pf2}) on $P$. These coframe 1-forms read:
%%/home/pawel/notebooks/paracr/metryka123rok2009example1takjakwpapernew3.nb
%%and earlier
\begin{eqnarray*}
&&\theta^1=-\frac{af_{32}}{f_{22}}(\der a_0+y\der
  a_2)+\frac{f_{11}f_{22}-xaf_{32}}{f_{22}}(\der a_1+y\der a_3),\\
&&\theta^2=-\frac{af_{31}}{f_{11}}(\der a_0+x\der
  a_1)+\frac{f_{11}f_{22}-yaf_{31}}{f_{11}}(\der a_2+x\der a_3),\\
&&\theta^3=-\frac{af_{31}f_{32}}{f_{11}f_{22}}\der
a_0+\frac{f_{31}(f_{11}f_{22}-xaf_{32})}{f_{11}f_{22}}\der
a_1+\frac{f_{32}(f_{11}f_{22}-yaf_{31})}{f_{11}f_{22}}\der
a_2-\\&&\quad\quad\quad\frac{(f_{11}f_{22}-xaf_{32})(f_{11}f_{22}-yaf_{31})}{af_{11}f_{22}}\der a_3,\\
&&\theta^4=a(\der a_0+x\der a_1+y\der a_2+xy\der a_3),\quad\quad\Om_2=\frac{a}{f_{11}}\der x,\quad\quad\Om_3=\frac{a}{f_{22}}\der y,
\end{eqnarray*}
\begin{eqnarray*}
&&\Om_1=\tfrac12\der\log(\frac{f_{11}}{f_{22}})+\frac{a}{f_{11}f_{22}}(f_{31}\der y-f_{32}\der
  x),\\
&&\Om_4=\frac{f_{31}}{f_{11}}\der\log(\frac{af_{31}}{f_{11}f_{22}})+\frac{af_{31}^2}{f_{11}^2f_{22}}\der
  y,\\
&&\Om_5=\frac{f_{32}}{f_{22}}\der\log(\frac{af_{32}}{f_{11}f_{22}})+\frac{af_{32}^2}{f_{11}f_{22}^2}\der
  x,\\
&&\Om_6=\tfrac12\der\log(\frac{f_{11}f_{22}}{a^2})-\frac{a}{f_{11}f_{22}}(f_{31}\der y+f_{32}\der
  x),\\
&&A=-\der\log(f_{11}f_{22}).
\end{eqnarray*}
It can be checked by a direct calculation that these forms satsify 
(\ref{sy1})-(\ref{sy4}).  
\end{proof}

\subsection{Newman's construction}\label{fm2}
After E. Ted Newman \cite{et,nemo} we recall that  
the system (\ref{et}) has the interesting property that its solution space
$\bbR^4$ is naturally equipped with a conformal metric of split
signature. This is defined as follows.  

Consider two neighboring solutions of (\ref{et}) corresponding to two
points ${\bf a}$ and ${\bf a}+\der {\bf a}$ in $\bbR^4$. These two
solutions can be considered as two surfaces, the graphs of two functions, 
\begin{eqnarray*}  
z(x,y)&=&a_0+a_1x+a_2y+a_3 xy\quad\quad\&\\
(z+\der z)(x,y)&=&(a_0+\der a_0)+(a_1+\der a_1)x+(a_2+\der a_2)y+(a_3+\der
  a_3) xy,
\end{eqnarray*}
in $\bbR^3$ with coordinates $(x,y,z)$. One can ask what conditions 
the two points ${\bf a}$ and ${\bf a}+\der{\bf a}$ in the solution
space $\bbR^4$ must
satisfy for these two surfaces to be tangent at some point $(x,y,z)$ in 
$\bbR^3$. An elementary argument shows that the point $(x,y,z)$ at which
the two surfaces are tangent satisfies the following equations:
\begin{eqnarray}
&&\der z=\der a_0+\der a_1x+\der a_2y+\der a_3 xy=0,\nonumber\\
&&(\der z)_x=\der a_1+\der a_3 y=0\quad\&\quad(\der z)_y=\der a_2+\der
  a_3 x=0.\nonumber
\end{eqnarray}
The first of the above equations says that the two surfaces intersect
at a point $(x,y, z(x,y))$, and the second two equations say that
they are tangent at the same point $(x,y,z(x,y))$. These three equations for the two
unknowns $(x,y)$ have a solution if and only if $\der {\bf a}$
satisfies a compatibility condition, which is obtained by 
eliminating $x$ and $y$
from the two second equations, and by inserting the so determined $x$ and
$y$ in the first equation. This compatibility condition is:
$$\der a_0\der a_3-\der a_1\der a_2=0.$$
Thus: two neighboring solutions ${\bf a}$ and ${\bf a}+\der {\bf a}$
of (\ref{et}) are tangent to each other at some point
$(x,y,z)$ in $\bbR^3$ if and only if they are \emph{null separated} in
the flat split signature metric 
\be g= 2(\der a_0\der a_3-\der a_1\der a_2)\label{wnu}\ee
in $\bbR^4$. This shows that the solution space of (\ref{et}) is
naturally equipped with a conformal structure. This gives a
correspondence between the incidence relations between two solutions
of (\ref{et}) treated as surfaces in $\bbR^3$ and conformal
properties of points in the solution space $\bbR^4$. This description
is very similar to the well known correspondences in the \emph{Lie sphere
geometry}, or more generally, in \emph{Penrose's twistor theory}. 

A new view of Newman's construction, stressing the
\emph{Weyl geometric} aspect of it, follows from our Theorem
\ref{pfa}, and is included in the following theorem.

\begin{theorem}\label{pfa1}
Every para-CR structure of type $(1,2,3)$ which is para-CR equivalent
to the para-CR structure (\ref{pf1})-(\ref{pf2}) uniquely defines an
11-dimensional principal fiber bundle $\cog(2,2)\to P\to{\mathcal S}$, 
with the 7-dimensional homothetic structure group $\cog(2,2)$, over a 4-dimensional
manifold $\mathcal S$, which can be identified with the solution space
of a pair of PDEs on the plane: $z_{xx}=0=z_{yy}$. It also defines a flat Weyl
geometry $[g,A]$ on $\mathcal S$, in which the $(2,2)$-signature metric
$g$ and the 1-form $A$ change conformally, $g\to {\rm e}^{2\phi}g$, $A\to
A-2\der \phi$, when the system $z_{xx}=0=z_{yy}$ undergoes a point
transformation of the variables $(x,y,z)$. 
\end{theorem}
\begin{proof}
Given a para-CR manifold locally equivalent 
to (\ref{pf1})-(\ref{pf2}) we use the previous theorem and construct
an 11-dimensional manifold $P$ with the coframe
$(\theta^i,\Omega_\mu,A)$ satisfying
(\ref{sy1})-(\ref{sy3}). It is convenient to write down these
equations explicitly. Equations (\ref{sy1}), when written in the
coframe $(\theta^i,\Omega_\mu,A)$ read:
\begin{eqnarray}
&&\der\theta^1=(\Omega_1-\tfrac12
A)\dz\theta^1-\Omega_3\dz\theta^3-\Omega_5\dz\theta^4\nonumber\\
&&\der\theta^2=(-\Omega_1-\tfrac12
A)\dz\theta^2-\Omega_2\dz\theta^3-\Omega_4\dz\theta^4\label{sy1p}\\
&&\der\theta^3=\Omega_4\dz\theta^1+\Omega_5\dz\theta^2+(\Omega_6-\tfrac12A)\dz\theta^3\nonumber\\
&&\der\theta^4=\Omega_2\dz\theta^1+\Omega_3\dz\theta^2+(-\Omega_6-\tfrac12A)\dz\theta^4,\nonumber
\end{eqnarray}
whereas equations (\ref{sy11}) read:
\begin{eqnarray}
&&\der\Omega_1=\Omega_2\dz\Omega_5-\Omega_3\dz\Omega_4\nonumber\\
&&\der\Omega_2=\Omega_2\dz(\Omega_1+\Omega_6)\nonumber\\
&&\der\Omega_3=(\Omega_1-\Omega_6)\dz\Omega_3\nonumber\\
&&\der\Omega_4=\Omega_4\dz(\Omega_1-\Omega_6)\label{sy11p}\\
&&\der\Omega_5=(\Omega_1+\Omega_6)\dz\Omega_5\nonumber\\
&&\der\Omega_6=\Omega_2\dz\Omega_5+\Omega_3\dz\Omega_4\nonumber\\
&&\der A=0.\nonumber
\end{eqnarray}
The appearance of only constant coefficients in front of the 2-forms on the right hand sides
of equations (\ref{sy1p})-(\ref{sy11p}) enables us to identify the
coframe forms $(\theta^i,\Omega_\mu,A)$ with the Maurer-Cartan forms on an
11-dimensional Lie group with a Lie algebra having structure constants
equal to these coeficients. This shows that $P$ is a Lie group. A 
look at the structure constants of the corresponding Lie algebra
given by
(\ref{sy1p})-(\ref{sy11p}), shows that this Lie group is $P=\bbR^4 \rtimes\cog(2,2)$.   
The $\cog(2,2)$ principal fibre bundle 
structure on $P=\bbR^4 \rtimes\cog(2,2)$ corresponds to the fibration $\cog(2,2)\to\bbR^4 \rtimes\cog(2,2)\to
\bbR^4$, i.e. to the natural principal $\cog(2,2)$ fibration over the
homogeneous space $\bbR^4\simeq
(\bbR^4\rtimes\cog(2,2))/\cog(2,2)$. Existence of this fibration on
$P$ is guaranteed by the equations (\ref{sy1}) (or what is the same
(\ref{sy1p})). They say that the 1-forms $(\theta^1,\theta^2,\theta^3,\theta^4)$ form a
closed differential ideal, so that their annihilator defines a foliation
of $P$ by 7-dimensional manifolds. On each of these 7-dimensional
manifolds the forms $(\theta^1,\theta^2,\theta^3,\theta^4)$ vanish identically, and
the additional seven 1-forms
$(\Omega_1,\Omega_2,\Omega_3,\Omega_4,\Omega_5,\Omega_6,A)$ form a
coframe. The differentials of this coframe, on each leaf of the
foliation, satisfies a closed exterior differential system with constant
coefficients (\ref{sy11p}). Thus each leaf can be identified with the
same Lie group, whose Lie algebra has structure constants determined
by (\ref{sy11p}). It is easy to see that this 7-dimensional Lie
algebra 
is the homothetic Lie algebra $\coa(2,2)$ of homothetic
motions in 4-dimensions associated with a metric of signature
$(+,+,-,-)$.

The appearance of the Lie group $\cog(2,2)$ as a subgroup of $P$
suggests that ${\mathcal S}=\bbR^4\simeq
(\bbR^4\rtimes\cog(2,2))/\cog(2,2)$ is naturally equipped with a
conformal metric of signature $(+,+,-,-)$. This is indeed the case. 
The metric
is obtained as follows: consider the bilinear form $G$ on
$P$ defined by:
$$G=2(\theta^1\theta^2+\theta^3\theta^4).$$
This form is highly degenerate on $P$, but its degenerate directions
are precisely along the fiber directions of the foliation
$\cog(2,2)\to P\to M$; actually $G$ has signature
$(+,+,-,-,0,0,0,0,0,0,0)$. 
Morever, using the sytem (\ref{sy1}) it can be
easily checked that the Lie derivatives of $G$ along all the
directions tangent to the fibres are just multiples of $G$. In
particular, if $Z$ is any vector field on $P$ tangent to the fibres,
we have $Z\hook\theta^i\equiv 0$, and as a consequence of (\ref{sy1p})
we get $${\mathcal L}_ZG=-(Z\hook A) G.$$ Thus $G$
descends to a conformal metric $[g]$ of signature $(+,+,-,-)$ on the
quotient space ${\mathcal S}=P/\cog(2,2)$.

Using the last equation (\ref{sy11p}) we also get 
$${\mathcal L}_ZA=\der (Z\hook A),$$
so we see that the pair $(G,A)$, changes as $(G,A)\to(G',A')=({\rm
  e}^{2\phi}G,A-2\der\phi)$ when it is Lie dragged along the fibres of
$\cog(2,2)\to P\to{\mathcal S}$. Thus it descends to a split signature
Weyl geometry $[g,A]$ on $\mathcal S$. The equations (\ref{sy11p}),
when pulled back to $\mathcal S$, show that this Weyl geometry is flat.    

To interpret the quotient ${\mathcal S}=P/\cog(2,2)$ as the solution
space of the pair of equations $z_{xx}=0=z_{yy}$ we use the
corresponding para-CR forms (\ref{pf1}), together with the 
explicit expressions for the invariant forms
$(\theta^1,\theta^2,\theta^3,\theta^4)$ and $A$ in coordinates
$(x,y,a_0,a_1,a_2,a_3,a,f_{11},f_{22},f_{31},f_{32})$ on $P$, as in the proof of
Theorem \ref{pfa}. A short calculation shows that 
$$G=2(\theta^1\theta^2+\theta^3\theta^4)=-2f_{11}f_{22}(\der a_0\der
a_3-\der a_1\der a_2).$$
This, together with $A=-\der\log(f_{11}f_{22})$, shows that the representative $(g,A)\in[g,A]$ can be taken as 
$$g=2(\der a_0\der a_3-\der a_1\der a_2),\quad\quad A=0,$$ and that $\mathcal S$ is
parametrized by $(a_0,a_1,a_2,a_3)$. Since these parameters constitute all the integration constants
of the equations $z_{xx}=0=z_{yy}$, the quotient $\mathcal S$ can be naturally
identified with the solution space of these equations.
\end{proof}
%\begin{remark}
%Note that the last few lines of the proof of the above theorem is an
%\emph{abstract nonsense} explanation of the conformal structure on the
%solution space of the system $z_{xx}=0=z_{yy}$ we encountered at the
%begining of this Section in (\ref{wnu}).  
%\end{remark}

\subsection{The principal bundle
  point of view and Weyl geometry}\label{fm3}
In the previous section we have shown how to associate an
11-dimensional principal fiber bundle $\cog(2,2)\to P\to\mathcal S$ to
any flat para-CR structure of type $(1,2,3)$. Here we reverse this
construction.  
\begin{proposition}\label{fiu}
Every 11-dimensional manifold $P$ with a coframe $(\theta^i,\Omega_\mu,A)$, $i=1,2,3,4$, $\mu=1,2,3,4,5,6$, satisfying
the differential system
\begin{eqnarray}
&&\der\theta^1=(\Omega_1-\tfrac12
A)\dz\theta^1-\Omega_3\dz\theta^3-\Omega_5\dz\theta^4\nonumber\\
&&\der\theta^2=(-\Omega_1-\tfrac12
A)\dz\theta^2-\Omega_2\dz\theta^3-\Omega_4\dz\theta^4\nonumber\\
&&\der\theta^3=\Omega_4\dz\theta^1+\Omega_5\dz\theta^2+(\Omega_6-\tfrac12A)\dz\theta^3\nonumber\\
&&\der\theta^4=\Omega_2\dz\theta^1+\Omega_3\dz\theta^2+(-\Omega_6-\tfrac12A)\dz\theta^4,\nonumber\\
&&\der\Omega_1=\Omega_2\dz\Omega_5-\Omega_3\dz\Omega_4+\tfrac12\kappa_{1ij}\theta^i\dz\theta^j\label{fyu}\\
&&\der\Omega_2=\Omega_2\dz(\Omega_1+\Omega_6)+\tfrac12\kappa_{2ij}\theta^i\dz\theta^j\nonumber\\
&&\der\Omega_3=(\Omega_1-\Omega_6)\dz\Omega_3+\tfrac12\kappa_{3ij}\theta^i\dz\theta^j\nonumber\\
&&\der\Omega_4=\Omega_4\dz(\Omega_1-\Omega_6)+\tfrac12\kappa_{4ij}\theta^i\dz\theta^j\nonumber\\
&&\der\Omega_5=(\Omega_1+\Omega_6)\dz\Omega_5+\tfrac12\kappa_{5ij}\theta^i\dz\theta^j\nonumber\\
&&\der\Omega_6=\Omega_2\dz\Omega_5+\Omega_3\dz\Omega_4+\tfrac12\kappa_{6ij}\theta^i\dz\theta^j\nonumber\\
&&\der A=\tfrac12F_{ij}\theta^i\dz\theta^j,\nonumber
\end{eqnarray}
with $\kappa_{aij}$, $F_{ij}$ being functions on $P$, is locally a
principal fiber bundle $\cog(2,2)\to P\to {\mathcal S}$ over a
4-dimensional manifold $\mathcal S$ naturally equipped with a Weyl
geometry $[g,A]$, in which the split signature conformal metric $g$ is
determined by a bilinear form $G=2(\theta^1\theta^2+\theta^3\theta^4)$
on $P$, and the Weyl potential 1-form is determined by the 1-form $A$ on
$P$. The curvature of this Weyl geometry is given by 
$${\mathcal R}=\bma
0&\kappa_1+\tfrac12{\mathcal F}&\kappa_2&\kappa_4\\-\kappa_1+\tfrac12{\mathcal F}&0&\kappa_3&\kappa_5\\-\kappa_2&-\kappa_3&0&\kappa_6+\tfrac12{\mathcal F}\\-\kappa_4&-\kappa_5&-\kappa_6+\tfrac12{\mathcal F}&0\ema,$$
where $\kappa_a=\tfrac12 \kappa_{aij}\theta^i\dz\theta^j$ and ${\mathcal
  F}=\tfrac12 F_{ij}\theta^i\dz\theta^j$.  
\end{proposition}
\begin{proof}
As in the proof of Theorem \ref{pfa1} we easily see that the
7-dimensional distribution annihilating
$(\theta^1,\theta^2,\theta^3,\theta^4)$ is integrable, and hence we
have a local projection $\pi: P\to {\mathcal S}$, identifying points
along the same leaves of the corresponding foliation. Since on the
leaves the forms $\theta^i$ vanish, and since the differentials 
$\der\Omega_\mu$s differ from those in (\ref{sy11p}) by terms that
vanish on the leaves, every leave is a local Lie group isomorphic to
$\cog(2,2)$. This proves that the manifold $P$ is locally a principal fiber
bundle $\cog(2,2)$ $\to P\to\mathcal S$. 

To prove that $\mathcal S$ has a natural Weyl structure $[g,A]$, one
repeats the argument from the previous proof. Although in (\ref{fyu}),
when compared to (\ref{sy1p})-(\ref{sy11p}), the new terms 
$\kappa_a$ and $\mathcal F$ appear, 
the argument from the previous proof is not altered. This
is because (1) the new terms do not appear in the `conformal
metricity/torsion' part of the equations (i.e. $\der\theta^i$
equations) and (2) they appear in $\der A$ only in harmless terms
which are annihilated by any vertical direction. 

The curvature of this Weyl structure can be calculated, by observing
that on any section $\sigma({\mathcal S})$ of $P$ the Weyl
connection is given by
$\Gamma^i_{~j}=g^{ik}\sigma^*(\Gamma_{[kj]}+\tfrac12 Ag_{kj})$, where
$\Gamma_{ij}$ is expressed in terms of the forms $\Omega_\mu$
appearing in (\ref{fyu}) via formula (\ref{sy4}), and $g_{ij}$,
$g^{ji}$ are as in (\ref{sy3}). The rest of the proof consists in
calculating ${\mathcal
  R}^i_{~j}=\der\Gamma^i_{~j}+\Gamma^i_{~k}\dz\Gamma^k_{~j}$ using
(\ref{fyu}) and lowering one index.   
\end{proof}
This proposition is crucial for the remaining sections. In particular
it can be used to prove the theorem, which gives the converse of
Newman's construction:
\begin{theorem}\label{ffm}
Every 11-dimensional manifold $P$ which is equipped with a coframe
$(\theta^i,\Omega_\mu,A)$, $i=1,2,3,4$, $\mu=1,2,3,4,5,6$, satisfying
the differential system (\ref{sy1p})-(\ref{sy11p}), is locally a principal
fiber bundle $\cog(2,2)\to P\to\mathcal S$, originating from a flat
para-CR manifold $(\Sigma,[\lambda,\mu_1,\mu_2,\nu_1,\nu_2,\nu_3])$ of type
$(1,2,3)$, via the procedure described by Theorem \ref{pfa}.
\end{theorem}
\begin{proof}
That $P$ with a system (\ref{sy1p})-(\ref{sy11p}) is locally a principal fiber
bundle $\cog(2,2)$ $\to P\to\mathcal S$ is an immediate consequence of
Proposition \ref{fiu} with $\kappa_a\equiv 0$ and ${\mathcal F}\equiv
  0$. Here we show that apart from the foliation
$\cog(2,2)\to P$, the system (\ref{sy1p})-(\ref{sy11p}) defines another
foliation of the manifold $P$, whose leaf space can be identified with
a 6-dimensional flat para-CR structure $\Sigma$. 

To see this consider the forms
$(\theta^1,\theta^2,\theta^3,\theta^4,\Omega_2,\Omega_3)$, and observe
that the system (\ref{sy1p}) and the second and the third equations
from system (\ref{sy11p}) guarantee that these six forms constitute a
closed differential ideal. Therefore their annihilator is a
5-dimensional integrable distribution on $P$, whose integral manifolds
define a 6-parameter foliation of $P$. Putting $\Omega_2\equiv 0\equiv
\Omega_3$ in equations (\ref{sy11p}) we see that the coframe
$(\Omega_1,\Omega_4,\Omega_5,\Omega_6,A)$ on these integral manifolds
satsifies a closed differential system with all the coefficients being
constants. Thus all these integral manifolds can be identified with a
unique Lie group $K$, which turns out to be a direct product $K={\bf Aff}(1)\times{\bf
  Aff}(1)\times \bbR^*$ of two independent groups of affine
transformations of the real line, ${\bf Aff}(1)$, and the
multiplicative group of the real numbers $\bbR^*$. This shows that the
manifold $P$, with the system of 1-forms (\ref{sy1p})-(\ref{sy11p}), can be
also locally viewed as a principal fibre bundle $K\to P\to
\Sigma$. Here $\Sigma$ is the 6-dimensional leaf space of the foliation
whose leaves are identified with $K$. Any manifold $\bar{\Sigma}$
transversal to the fibres of these fibration is equipped with a
coframe
$(\bar{\theta}^i,\bar{\Omega}_2,\bar{\Omega}_3)=(\theta^i,\Omega_2,\Omega_3)_{|\bar{\Sigma}}$,
$i=1,2,3,4$, which satisfies the system 
\begin{eqnarray}
&&\der\bar{\theta}^1=(\bar{\Omega}_1-\tfrac12
\bar{A})\dz\bar{\theta}^1-\bar{\Omega}_3\dz\bar{\theta}^3-\bar{\Omega}_5\dz\bar{\theta}^4\nonumber\\
&&\der\bar{\theta}^2=(-\bar{\Omega}_1-\tfrac12
\bar{A})\dz\bar{\theta}^2-\bar{\Omega}_2\dz\bar{\theta}^3-\bar{\Omega}_4\dz\bar{\theta}^4\nonumber\\
&&\der\bar{\theta}^3=\bar{\Omega}_4\dz\bar{\theta}^1+\bar{\Omega}_5\dz\bar{\theta}^2+(\bar{\Omega}_6-\tfrac12\bar{A})\dz\bar{\theta}^3\label{syk}\\
&&\der\bar{\theta}^4=\bar{\Omega}_2\dz\bar{\theta}^1+\bar{\Omega}_3\dz\bar{\theta}^2+(-\bar{\Omega}_6-\tfrac12\bar{A})\dz\bar{\theta}^4,\nonumber\\
&&\der\bar{\Omega}_2=\bar{\Omega}_2\dz(\bar{\Omega}_1+\bar{\Omega}_6)\nonumber\\
&&\der\bar{\Omega}_3=(\bar{\Omega}_1-\bar{\Omega}_6)\dz\bar{\Omega}_3,\nonumber
\end{eqnarray}
with forms $\bar{\Omega}_1$, $\bar{\Omega}_4$, $\bar{\Omega}_5$,
$\bar{\Omega}_6$ and $\bar{A}$ on $\bar{\Sigma}$. That these forms are
the restrictions of $\Omega_1$, $\Omega_4$, $\Omega_5$, $\Omega_6$ and
$A$ to $\bar{\Sigma}$ is not important in the following. What is
important, is that the system (\ref{syk}) on $\bar{\Sigma}$ is
satisfied by a \emph{coframe}
$(\bar{\theta}^i,\bar{\Omega}_2,\bar{\Omega}_3)$, and that it implies
that the quartet of forms
$(\bar{\theta}^1,\bar{\theta}^2,\bar{\theta}^3,\bar{\theta}^4)$, as well as the triplet of forms
$(\bar{\theta}^4,\bar{\Omega}_2,\bar{\Omega}_3)$, \emph{both} form 
closed differential ideals of 1-forms on $\bar{\Sigma}$. Thus the 2-dimensional 
anihilator $\bar{H}^+$ of
$(\bar{\theta}^1,\bar{\theta}^2,\bar{\theta}^3,\bar{\theta}^4)$, as
well as the 3-dimensional anihilator $\bar{H}^-$ of
$(\bar{\theta}^4,\bar{\Omega}_2,\bar{\Omega}_3)$, define foliations of
$\bar{\Sigma}$ by, respectively, a 4-parameter family of 2-dimensional
leaves, and a 3-parameter family of 3-dimensional leaves. The integrable
distributions $\bar{H}^+$ and $\bar{H}^-$ obviously have
$\bar{H}^+\cap\bar{H}^-=\{0\}$, equipping each $\bar{\Sigma}$ with a
para-CR structure $(\bar{\Sigma},\bar{H}^+,\bar{H}^-)$. It is matter
of checking that the $(1,2,3)$-type para-CR structures on each
$\bar{\Sigma}$ are localy equivalent to each other, and that they
descend to the unique $(1,2,3)$-type para-CR structure
$(\Sigma,H^+,H^-)$ on the quotient $\Sigma=P/K$. Obviously this
para-CR structure is the flat one of Theorem \ref{pfa}.   

\end{proof}
\subsection{Non flat case}\label{fm4}
Now we generalize the flat example of Sections \ref{fm1}-\ref{fm3}  
to systems of PDEs on the plane of
the form 
\be 
z_{xx}=R(x,y,z,z_x,z_y,z_{xy})\quad\quad\&\quad\quad
z_{yy}=T(x,y,z,z_x,z_y,z_{xy}).\label{dhpndh} \ee      
We assume that they are finite type, or, what is the
same, we assume that their general solution can be written as 
$$z=\psi(x,y,a_0,a_1,a_2,a_3).$$
This is always the case \cite{nemo}, when the functions $R=R(x,y,z,p,q,s)$ and
$T=T(x,y,z,p,q,s)$ satisfy 
\be D^2_xT=D^2_yR,\label{pnd}\ee 
where the differential 
operators $D_x$ and $D_y$ are implicitly given by 
\be 
D_x=\partial_x+p\partial_z+R\partial_p+s\partial_q+D_yR\partial_s\quad\quad\&\quad\quad
D_y=\partial_y+q\partial_z+s\partial_p+T\partial_q+D_xT\partial_s.\label{pndha}
\ee
To make this implicit definition of $D_x$ and $D_y$ explicit 
one has to solve for $D_xT$ and $D_yR$ in
$D_xT=T_x+pT_z+RT_p+sT_q+(D_yR)T_s$ and
$D_yR=R_y+qR_z+sR_p+TR_q+(D_xT)R_s$. This is only possible if 
\be
T_sR_s\neq 1,\label{ndg}\ee which when assumed, defines $D_xT$ and $D_yR$
uniquely, and in turn after insertion in (\ref{pndha}), makes the
operators $D_x$ and $D_y$ explicit. Thus we assume (\ref{ndg}) from
now on.

To define a type $(1,2,3)$ para-CR structure associated with the 
system (\ref{dhpndh}), (\ref{pnd}), (\ref{ndg}) we do as follows. 
First, using the general solution $z=\psi(x,y,a_0,a_1,a_2,a_3)$, we
define the forms 
\begin{eqnarray}
&&\lambda=\psi_0\der a_0+\psi_1\der a_1+\psi_2\der a_2+\psi_3\der a_3\nonumber\\
&&\mu_1=\der x\nonumber\\
&&\mu_2=\der y\nonumber\\
&&\nu_1=\der a_1\label{134}\\
&&\nu_2=\der a_2\nonumber\\
&&\nu_3=\der a_3,\nonumber
\end{eqnarray}
Then we extend these forms to the class
$[\la,\mu_1,\mu_2,\nu_1,\nu_2,\nu_3]$ via (\ref{pf2}). This equips the
6-dimensional hypersurface
$$\Sigma=\{(x,y,z,a_0,a_1,a_2,a_3,a_4)\in\bbR^7~|~z=\psi(x,y,a_0,a_1,a_2,a_3)\}$$
in $\bbR^3\times\bbR^4$ with the $(1,2,3)$-para-CR structure $[\la,\mu_1,\mu_2,\nu_1,\nu_2,\nu_3]$. Alternatively, a \emph{para-CR equivalent} structure
may be defined on the second jets ${\mathcal J}^2$ of the system
(\ref{dhpndh})-(\ref{pnd}). Parametrizing this space by
$(x,y,z,p,q,s)$ we use the contact forms
\begin{eqnarray}
&&\lambda=\der z-p\der x-q\der y\nonumber\\
&&\mu_1=\der x\nonumber\\
&&\mu_2=\der y\nonumber\\
&&\nu_1=\der p-R\der x-s\der y\label{135}\\
&&\nu_2=\der q-s\der x-T\der y\nonumber\\
&&\nu_3=\der s-D_yR\der x-D_xT\der y,\nonumber
\end{eqnarray}
and define the type $(1,2,3)$ para-CR structure by extending these
forms to a class of para-CR forms 
$[\la,\mu_1,\mu_2,\nu_1,\nu_2,\nu_3]$ on ${\mathcal J}^2$ via:
\be
\bma
\la\\
\nu_1\\
\nu_2\\
\nu_3\\
\mu_1\\
\mu_2
\ema\to\bma
\la'\\
\nu_1'\\
\nu_2'\\
\nu_3'\\
\mu_1'\\
\mu_2'
\ema =\bma
a&0&0&0&0&0\\
b_{1}&f_{11}&f_{12}&f_{13}&0&0\\
b_{2}&f_{21}&f_{22}&f_{23}&0&0\\
b_{3}&f_{31}&f_{32}&f_{33}&0&0\\
c_{1}&0&0&0&h_{11}&h_{12}\\
c_{2}&0&0&0&h_{21}&h_{22}
\ema
\bma
\la\\
\nu_1\\
\nu_2\\
\nu_3\\
\mu_1\\
\mu_2
\ema,\label{po}
\ee
where $a,b_A,c_\alpha,f^A_{~B},h^\alpha_{~\beta}$ are arbitrary parameters
such that $a\det(f^A_{~b})\det(h^\alpha_{~\beta})\neq 0$. Let us now define, as
before, the lifted coframe 
\be
\bma
\theta^4\\
\theta^1\\
\theta^2\\
\theta^3\\
\Om_3\\
\Om_2
\ema=\bma
a&0&0&0&0&0\\
b_{1}&f_{11}&f_{12}&f_{13}&0&0\\
b_{2}&f_{21}&f_{22}&f_{23}&0&0\\
b_{3}&f_{31}&f_{32}&f_{33}&0&0\\
c_{1}&0&0&0&h_{11}&h_{12}\\
c_{2}&0&0&0&h_{21}&h_{22}
\ema
\bma
\la\\
\nu_1\\
\nu_2\\
\nu_3\\
\mu_1\\
\mu_2
\ema.\label{poi}
\ee
%%from now on everything is in the files:
%%/home/pawel/notebooks/paracr/123/para123wawlast.nb
%%i pozniejszych: ..wawlast1, ..wawlast2, ..wawlastep1, ..ep1dalej, ..epm
We ask which conditions the functions $R$ and $T$ must satsify so that
the forms $(\theta^1,\theta^2,\theta^3,\theta^4,\Om_2,\Om_3)$ are 
forced to satisfy the system (\ref{fyu}) with some auxiliary forms
$(\Om_1,\Om_4,\Om_5,\Om_6,A)$, on a certain 11-dimensional manifold
$P$, where $(\theta^i,\Omega_\mu,A)$ would serve as a coframe. As a
first result in this respect we have 
the following theorem.    
\begin{theorem}\label{poj}
A neccessary condition for the 
equations $z_{xx}=R(x,y,z,z_x,z_y,z_{xy})$ $\&$
$z_{yy}=T(x,y,z,z_x,z_y,z_{xy})$ satisfying $D_x^2T=D_y^2R$,
$1-R_sT_s>0$ 
to admit forms (\ref{po})-(\ref{poi}) with  
\begin{eqnarray}
&&\der\theta^4\dz\theta^4=(\Om_2\dz\theta^1+\Om_3\dz\theta^2)\dz\theta^4\label{uiu1}\\
&&\der\theta^1\dz\theta^1\dz\theta^2\dz\theta^4\dz\Om_3=0\label{uiu2}\\
&&\der\theta^2\dz\theta^1\dz\theta^2\dz\theta^4\dz\Om_2=0\label{uiu3}\\
&&\der\theta^1\dz\theta^1\dz\theta^3\dz\theta^4\dz\Om_3=0\label{uiu4}\\
&&\der\theta^2\dz\theta^2\dz\theta^3\dz\theta^4\dz\Om_2=0\label{uiu5}
\end{eqnarray} 
is that the functions $R=R(x,y,z,z_x,z_y,z_{xy})$ and
$T=T(x,y,z,z_x,z_y,z_{xy})$ satisfy 
$$J_1\equiv 0,\quad\quad\&\quad\quad J_2\equiv 0,$$
where 
\begin{eqnarray*}
&&J_1=(R_sT_s-4)D_xR_s+R_s(2D_yR_s-R_sD_xT_s)+\\
&&8R_q-6R_qR_sT_s+4R_pR_s+2R_s^2T_q-2R_pR_s^2T_s+2R_s^3T_p\\
&&\\
&&J_2=(R_sT_s-4)D_yT_s+T_s(2D_xT_s-T_sD_yR_s)+\\
&&8T_p-6R_sT_pT_s+4T_qT_s+2R_pT_s^2-2R_sT_qT_s^2+2R_qT_s^3.
\end{eqnarray*}
\end{theorem}
\begin{proof}
We force the forms $(\theta^1,\theta^2,\theta^3,\theta^4,\Om_2,\Om_3)$
to satisfy (\ref{uiu1})-(\ref{uiu5}) in the following steps:

First we fix coefficients $f_{23}$, $f_{33}$, $h_{11}$, $h_{12}$,
  $h_{21}$ and $h_{22}$ by forcing $\der\theta^4$ to satsify (\ref{uiu1}). For
  this to be satisfied we must take: \be f_{13}=f_{23}=0,\label{iwi}\ee
and 
\be
\begin{aligned}
&&h_{11}=\frac{af_{12}}{f_{12}f_{21}-f_{11}f_{22}},\quad
  h_{12}=-\frac{af_{11}}{f_{12}f_{21}-f_{11}f_{22}},\\
&&h_{21}=-\frac{af_{22}}{f_{12}f_{21}-f_{11}f_{22}},\quad 
h_{22}=\frac{af_{21}}{f_{12}f_{21}-f_{11}f_{12}}.
\end{aligned}\label{iwi1}\ee

After these normalizations we have
$$\der\theta^1\dz\theta^1\dz\theta^2\dz\theta^4\dz\Om_3=\frac{2f_{11}f_{12}+R_sf_{11}^2+T_sf_{12}^2}{af_{33}}\Om_2\dz\Om_3\dz\theta^1\dz\theta^2\dz\theta^3\dz\theta^4$$
and
$$\der\theta^2\dz\theta^1\dz\theta^2\dz\theta^4\dz\Om_2=\frac{2f_{21}f_{22}+R_sf_{21}^2+T_sf_{22}^2}{af_{33}}\Om_3\dz\Om_2\dz\theta^1\dz\theta^2\dz\theta^3\dz\theta^4.$$
Thus to satisfy (\ref{uiu2}) and (\ref{uiu3}) we must equate to zero
the right hand sides of these equations. It is the moment, when we
need the assumption $$1-R_sT_s>0.$$ When this is assumed we achieve
(\ref{uiu2}) and (\ref{uiu3}) by normalizing:
\be
f_{21}=\frac{-1\pm w}{R_s}f_{22},\quad\quad f_{11}=\frac{-1\mp
  w}{R_s}f_{12},\quad\quad w=\sqrt{1-R_sT_s}.\label{iwi2}\ee
%\edz{we have to comment about the case $R_s=0$ but it is a pain in the ass} 

With these normalizations we now have
\be
\begin{aligned}
&\der\theta^1\dz\theta^1\dz\theta^3\dz\theta^4\dz\Om_3=\\
&f_{12}^2\frac{\big(1+3w^2\pm3w\pm w^3\big)J_1- R_s^3J_2}{4
    af_{22}R_s^2w^2}\Om_2\dz\Om_3\dz\theta^1\dz\theta^2\dz\theta^3\dz\theta^4\\
&\\
&\der\theta^2\dz\theta^2\dz\theta^3\dz\theta^4\dz\Om_2=\\
&f_{22}^2\frac{\big(1+3w^2\mp 3w\mp w^3\big)J_1-
  R_s^3J_2}{4
    af_{12}R_s^2w^2}\Om_2\dz\Om_3\dz\theta^1\dz\theta^2\dz\theta^3\dz\theta^4.
\end{aligned}\label{iwi3}\ee
%\edz{here $w=0$, i.e. $R_sT_s=1$ should be excluded, but i have no
%  patience to do it}
The right hand sides of these equations vanish identically if
and only if 
\begin{eqnarray*}
&&\big(1+3w^2\pm3w\pm w^3\big)J_1-R_s^3J_2\equiv 0\\
&&\big(1+3w^2\mp3w\mp w^3\big)J_1-
  R_s^3J_2\equiv 0.
\end{eqnarray*}
Since
$$\det\bma1+3w^2\pm 3w\pm w^3&-R_s^3 \\1+3w^2\mp 3w\mp w^3&-
  R_s^3 \ema=\mp 2R_s^3w(3+w^2)\neq 0$$
this is only possible if and only if $J_1\equiv J_2\equiv 0$, which finishes the proof.
\end{proof}

The meaning of vanishing of both $J_1$ and $J_2$, known as \emph{Newman's
  metricity conditions} \cite{et,nemo}, is given in the following theorem.  
%%and here we have 
%%/home/pawel/notebooks/paracr/123/para123wawlast2.nb, last2ep1,
%%ep1dalej, last2epm1
\begin{theorem}\label{alb}
If conditions
$$J_1\equiv 0\quad\quad\&\quad\quad J_2\equiv 0$$
are satisfied then one can normalize the forms
$(\theta^1,\theta^2,\theta^3,\theta^4,\Om_2,\Om_3)$ in such a way
that they, together
with the auxiliary forms $(\Om_1,\Om_2,\Om_5,\Om_6,A)$, satisfy
\begin{eqnarray}
&&\der\theta^1=(\Omega_1-\tfrac12
A)\dz\theta^1-\Omega_3\dz\theta^3-\Omega_5\dz\theta^4+ t^1_{23}\theta^2\dz\theta^3\nonumber\\
&&\der\theta^2=(-\Omega_1-\tfrac12
A)\dz\theta^2-\Omega_2\dz\theta^3-\Omega_4\dz\theta^4+ t^2_{13}\theta^1\dz\theta^3\label{bla}\\
&&\der\theta^3=\Omega_4\dz\theta^1+\Omega_5\dz\theta^2+(\Omega_6-\tfrac12A)\dz\theta^3\nonumber\\
&&\der\theta^4=\Omega_2\dz\theta^1+\Omega_3\dz\theta^2+(-\Omega_6-\tfrac12A)\dz\theta^4,\nonumber
\end{eqnarray}
where \begin{eqnarray}
&&t^1_{23}=-\frac{a}{8f^2_{22}w^4}\big(R_{ss}(1\pm w)^2+T_{ss}R_s^2\big),\label{bla1}\\
&&t^2_{13}=-\frac{a}{8f^2_{32}w^4}\big(R_{ss}(1\mp w)^2+T_{ss}R_s^2\big).\nonumber
\end{eqnarray}
With this normalization the bilinear form
$G=2(\theta^1\theta^2+\theta^3\theta^4)$ descends to a conformal
$(+,+,-,-)$ signature metric $[g]$ on the 4-dimensional leaf space $\mathcal
S$ of the foliation
defined by the integrable distribution anihilating
$(\theta^1,\theta^2,\theta^3,\theta^4)$.

Modulo a discrete point transformation, interchanging $\theta^1$ with
$\theta^2$, the vanishing or not of at least one of 
$$K_1=R_{ss}(1-\sqrt{1-R_sT_s})^2+T_{ss}R_s^2,\quad\quad K_2=R_{ss}(1+\sqrt{1-R_sT_s})^2+T_{ss}R_s^2,$$
is a point invariant property of the
corresponding system $z_{xx}=R(x,y,z,z_x,z_y,z_{xy})$ $\&$ 
$z_{yy}=T(x,y,z,z_x,z_y,z_{xy})$. In particular, the simultaneous
vanishing of $R_{ss}$ and $T_{ss}$, $R_{ss}\equiv T_{ss}\equiv 0$, is
a point invariant property of the system.
\end{theorem}
\begin{proof}
%%/home/pwel/notebooks/paracr/123/mat5waw.nb
If we prove that the forms $(\theta^1,\theta^2,\theta^3,\theta^4)$ can
be forced to satisfy the system
(\ref{bla}) on some $11$-dimensional manifold $P$, where the forms
$(\theta^i,\Omega_\mu,A)$ are linearly independent, then similarly as
in the proof of Proposition \ref{fiu}, we will have a foliation of $P$
by the integral leaves of a 7-dimensional \emph{integrable} distribution annihilated by
$(\theta^1,\theta^2,\theta^3,\theta^4)$. Moreover because 
(\ref{bla}) differs from (\ref{sy1p}) by only the
appearence of $\theta^i\dz\theta^k$ terms, the Lie derivtives of $G$
with respect to the vectors tangent to the foliation, will be given by
the same expressions as in the proof of Theorem \ref{pfa1}. Thus, if
we prove (\ref{bla}), we will get the conclusion that 
the leaf space $\mathcal S$ is equipped with the conformal split
signature metrics $[g]$ to which $G$ descends.

The procedure of bringing the forms $(\theta^i)$ to the form in which they
satisfy (\ref{bla}) is based on Cartan's equivalence method. The
Cartan process of normalizing the group coefficients
$a,b_i,c_i,f_{ij},h_{ij}$ has two loops, the first of which ends
after normalization of the coeficient $b_1$.

{\bf The first loop}: 
We first impose the conditions (\ref{uiu1})-(\ref{uiu5}), as in the
previous proof, and as before reduce the possible freedom in the choice of
$(\theta^1,\theta^2,\theta^3,\theta^4,\Om_2,\Om_3)$  to
\be
\bma
\theta^4\\
\theta^1\\
\theta^2\\
\theta^3\\
\Om_3\\
\Om_2
\ema=\bma
a&0&0&0&0&0\\
b_{1}&\frac{-1\mp w}{R_s}f_{12}&f_{12}&0&0&0\\
b_{2}&\frac{-1\pm w}{R_s}f_{22}&f_{22}&0&0&0\\
b_{3}&f_{31}&f_{32}&f_{33}&0&0\\
c_{1}&0&0&0&\pm \frac{aR_s}{2wf_{22}}&\pm \frac{a(1\pm w)}{2wf_{22}}\\
c_{2}&0&0&0&\mp \frac{aR_s}{2wf_{12}}&\pm\frac{a(-1\pm w)}{2wf_{12}}
\ema
\bma
\la\\
\nu_1\\
\nu_2\\
\nu_3\\
\mu_1\\
\mu_2
\ema\label{lop}
\ee
In the next step we impose the condition 
$\der\theta^1\dz\theta^1\dz\theta^2\dz\theta^4=-\Om_3\dz\theta^1\dz\theta^2\dz\theta^3\dz\theta^4$. 
This
gives the normalization
\be
f_{33}=\frac{2w^2f_{12}f_{22}}{aR_s}\label{iwi4}\ee
and implies also that  
$\der\theta^2\dz\theta^1\dz\theta^2\dz\theta^4=-\Om_2\dz\theta^1\dz\theta^2\dz\theta^3\dz\theta^4$.

Then we require that
$\der\theta^2\dz\theta^2\dz\theta^3\dz\theta^4=0$. This determines
$b_2$ as:
\be
\begin{aligned}
&b_2=\pm \frac{a(f_{32}+f_{31}R_s\mp f_{32}w)}{2f_{12}w}\pm\frac{f_{22}}{R_s^2w(3+w^2)}\times\\
&\Big((1\pm w)(1-w^2)^2R_q + (1\mp w)^2R_s(D_yR_s+T_pR_s^2+R_s
  T_q)+\\&(1\pm w)R_s(R_sD_xT_s+R_p(1-w^2))\Big).\end{aligned}\label{iwi5}   
\ee
Similarly the condition 
$\der\theta^1\dz\theta^1\dz\theta^3\dz\theta^4=0$ determines $b_1$ as:
\be
\begin{aligned}
&b_1=\mp \frac{a(f_{32}+f_{31}R_s\pm f_{32}w)}{2f_{22}w}\mp\frac{f_{12}}{R_s^2w(3+w^2)}\times\\
&\Big((1\mp w)(1-w^2)^2R_q + (1\pm w)^2R_s(D_yR_s+T_pR_s^2+R_s
  T_q)+\\&(1\mp w)R_s(R_sD_xT_s+R_p(1-w^2))\Big) .\end{aligned}\label{iwi6}
\ee
After these normalizations have been imposed, we have to associate the
remaining undetermined parameters
$a,f_{12},f_{22},f_{31},f_{32},b_3,c_1$ and $c_2$ with the auxiliary
forms $\Om_1,\Om_4,\Om_5,\Om_6$ and $A$. 

This is done by first observing that the equation 
\be\der\theta^4=\Om_{2}\dz\theta^1+\Om_3\dz\theta^2-(\Om_6+\tfrac12A)\dz\theta^4.\label{1u}\ee
is equivalent to 
\be
-\Om_6-\tfrac12A=\frac{\der a}{a}-\frac{b_1}{a}\Om_2-\frac{b_2}{a}\Om_3+\frac{c_2}{a}\theta^1+\frac{c_1}{a}\theta^2+u_{111}\theta^4,\label{u11}
\ee
with $b_1$ and $b_2$ as above, and an unspecified new parameter $u_{111}$. 

From now on we only sketch the proof, which is based on massive
computer calculations using Mathematica.

After relating $\der a$ to $\Om_6+\tfrac12 A$ we pass to the condition 
\be
\der\theta^3\dz\theta^1\dz\theta^2\dz\theta^3=0.\label{onc}\ee 
It follows that
this can only be satisfied if the differential $\der b_3$ is
\begin{eqnarray}
&&\der b_3=b_{302}\Om_2+b_{303}\Om_3+b_{306}(\Om_6+\tfrac12 A)+\label{u12}\\
&&b_{322}\der
f_{22}+b_{312}\der f_{12}+b_{331}\der f_{31}+b_{332}\der
f_{32}+b_{31}\theta^1+b_{32}\theta^2+b_{33}\theta^3+b_{34}\theta^4.\nonumber\end{eqnarray}
The functions $b_{302}$, $b_{303}$, $b_{306}$, $b_{322}$,
$b_{312}$, $b_{331}$ and $b_{332}$ are \emph{uniquely} determined by
(\ref{onc}), and are expressible in terms of $R,T$, their derivatives up to 
order two, and the
free parameters $a,f_{12},f_{22},f_{31},f_{32},b_3$. The parameters
$b_{31}$, $b_{32}$, $b_{33}$ and $b_{34}$ are arbitrary. 
Using Mathematica we found 
explicit expressions for this differential up to the undetermined
$\theta^i$ terms. Due to the enormous size
of this formula we do not quote it here. We note, however, that the
free parameters $c_1$ and $c_2$ are not present in $\der b_3$.

Now, using all the normalizations obtained so far, and $\der b_3$ as
above, we impose the condition
\be\der\theta^1\dz\theta^3\dz\theta^4=(\Om_1-\tfrac12
A)\dz\theta^1\dz\theta^3\dz\theta^4.\label{pop}\ee
This gives 
\be
\Om_1-\tfrac12 A=\der\log
f_{12}+f_{122}\Om_2+f_{123}\Om_3+f_{121}\theta^2+\dots.\label{u13}\ee
The dots here denote the undetermined $(\theta^1,\theta^3,\theta^4)$
  terms. The functions $f_{122}$, $f_{123}$ and $f_{121}$ are
\emph{uniquely} and \emph{explicitly} determined by (\ref{pop}). Similarly, imposition of 
\be\der\theta^2\dz\theta^3\dz\theta^4=(-\Om_1-\tfrac12
A)\dz\theta^2\dz\theta^3\dz\theta^4.\label{pop1}\ee
gives 
\be-\Om_1-\tfrac12 A=\der\log
f_{22}+f_{222}\Om_2+f_{223}\Om_3+f_{221}\theta^1+\dots\label{u14}\ee
with \emph{uniquely} determined functions $f_{222}$, $f_{223}$ and
$f_{221}$, and dots denoting the undetermined 
$(\theta^2,\theta^3,\theta^4)$ terms.

Now the condition
\be
\der\theta^3\dz\theta^1\dz\theta^2=(\Om_6-\tfrac12 A)\dz\theta^1\dz\theta^2\dz\theta^3\label{popa}\ee
is used to reduce the freedom in the choice of the undetermined $\theta^4$
terms in (\ref{u11}), (\ref{u13}), (\ref{u14}) and the undetermined
$\theta^3$ term in (\ref{u12}). This
\emph{one scalar} condition gives a linear relation between the coefficient
$u_{111}$, the coefficient $b_{33}$ at the $\theta^3$ term in (\ref{u12}), and the two coefficients at $\theta^4$ in (\ref{u13}) and
(\ref{u14}). Denoting the last two 
coefficients by $f_{124}$ and $f_{224}$  respectively, we use
(\ref{popa}) to obtain $b_{33}$ as a linear combination (with coefficients
depending on $R$, $T$, their derivatives, and the free parameters such
as $a$, etc.) of $u_{111}$, $f_{124}$
and $f_{224}$. 
 
At this stage we have associated the forms $\Om_6$, $\Om_1$, and $A$ to 
nonsingular linear combinations of the differentials $\der a$, $\der
f_{22}$ and $\der f_{12}$. The still unknown forms $\Om_4$ and $\Om_5$ can
now be related to $\der f_{31}$ and $\der f_{32}$ by imposing the
condition 
\be
\der\theta^3=\Om_4\dz\theta^1+\Om_5\dz\theta^2+(\Om_6-\tfrac12
A)\dz\theta^3.\label{2u}\ee
The imposition of this condition results in 
\begin{eqnarray}
&&\Om_4=\mp\frac{R_s}{2f_{12}w}\der f_{31}\mp\frac{1\mp w}{2f_{12}w}\der f_{32}+\dots+\al\theta^1+\beta\theta^2\label{u15}\\
&&\Om_5=\pm\frac{R_s}{2f_{22}w}\der f_{31}\pm\frac{1\pm w}{2f_{22}w}\der f_{32}+\dots+\beta\theta^1+\gamma\theta^2,\label{u16}
\end{eqnarray}
where the doted terms are totally and \emph{uniquely} determined by $R$, $T$, their
derivatives, and the previous choices. Here $\alpha,\beta,\gamma$ are
new free parameters. 

We stress that we calculated explicitly the right hand sides of
equations (\ref{u12}), (\ref{u13}), (\ref{u14}), (\ref{u15}) and
(\ref{u16}). We do not quote them here in full generality due to the
lack of space. But now, having these right hand sides calculated, we
can calculate $\der\theta^1$ and $\der\theta^2$. It follows from these
calculations that
$$\der\theta^1\dz\theta^1\dz\theta^2\dz\theta^3\dz\Om_5=H\Om_3\dz\Om_5\dz\theta^1\dz\theta^2\dz\theta^3\dz\theta^4,$$
and 
$$\der\theta^2\dz\theta^1\dz\theta^2\dz\theta^3\dz\Om_4=H\Om_2\dz\Om_4\dz\theta^1\dz\theta^2\dz\theta^3\dz\theta^4.$$
The function $H$ appearing in these equations has the form 
$$H=A b_3+ B,$$
where $A\neq 0$ and $B$ are functions of $R$, $T$, their
derivatives up to order \emph{three}, and only \emph{five} free
parameters $a, f_{12}, f_{22}, f_{31}$ and $f_{32}$. To satisfy
the first two of the equations (\ref{bla}) we need $H\equiv 0$. This
gives the normalization of the parameter $b_3$ as
$$b_{3}=-\frac{B}{A}.$$
This, when compared with $\der b_3$ given by (\ref{u12}), and
everything after this equation, might bring 
compatibility conditions. Thus we are at the end of the first loop: we have to
return to the formula (\ref{u12}) with $b_3=-B/A$ and repeat all the
steps after this formula, inserting this $b_3$ everywhere.

Note that as the result of the first loop we have forms
$(\theta^1,\theta^2,\theta^3,\theta^4)$ satisfying the last two
equations (\ref{bla}). 
 
%mat5wawcommutator2.nb
{\bf The second loop:} Now we start with the forms (\ref{lop}), in
which we use $f_{33}$, $b_1$, $b_2$ and $b_3$ determined in the first
loop. Then, as before, (\ref{u11}) guarantees that (\ref{1u}) is
valid, and (\ref{onc}) is satisfied \emph{automatically}. This means 
that we do not need equation (\ref{u12}) anymore. Equations
 (\ref{pop}) and (\ref{pop1}) as before determine $\Om_1-\tfrac12 A$
and $-\Om_2-\tfrac12 A$, so that (\ref{u13}) and (\ref{u14}) are
satisfied, with new but still explicitly determined $f_{122},f_{123},f_{121},f_{222},f_{223},f_{221}$. Since now we do not have (\ref{u12}), we use
(\ref{popa}) to determine $u_{111}$. After
this, we calculate $\der\theta^3$. This satisfies (\ref{2u}) provided
that $\Om_4$ and $\Om_5$ are as in (\ref{u15}) and (\ref{u16}), with
everything determined except the parameters
$\alpha,\beta,\gamma$. Choosing these $\Om_4$ and $\Om_5$ we have also
have 
$$\der\theta^1\dz\theta^1\dz\theta^2\dz\theta^3\dz\Om_5=\der\theta^2\dz\theta^1\dz\theta^2\dz\theta^3\dz\Om_4=0.$$
It turns out that out of the \emph{nine} undetermined parameters:
$\al,\beta,\gamma$ and the ones in the dotted terms in (\ref{u14}) and
(\ref{u13}), \emph{eight} are totally determined by the requirement
that $\der\theta^1=(\Omega_1-\tfrac12
A)\dz\theta^1-\Omega_3\dz\theta^3-\Omega_5\dz\theta^4+
t^1_{23}\theta^2\dz\theta^3$ $\&$ 
$\der\theta^2=(-\Omega_1-\tfrac12
A)\dz\theta^2-\Omega_2\dz\theta^3-\Omega_4\dz\theta^4+
t^2_{13}\theta^1\dz\theta^3$. If this condition is imposed the
remaining free parameters are
$a,f_{12},f_{22},f_{31},f_{32},c_{1},c_2,\beta$. It also follows that
this condition forces the coefficients $t^1_{23}$ and $t^2_{13}$ 
to be given by (\ref{bla1}). This finishes the proof. 
\end{proof}
\begin{remark}\label{rbl}
Further conditions
\begin{eqnarray*}
&&\der\Om_2\dz\theta^1\dz\theta^3\dz\theta^4\dz\Om_2=0\\
&&\der\Om_3\dz\theta^2\dz\theta^3\dz\theta^4\dz\Om_3=0,
\end{eqnarray*}
imposed on the system
$(\theta^1,\theta^2,\theta^3,\theta^4,\Om_2,\Om_3)$ uniquely determine
parameters $c_1$ and $c_2$. To fix the parameter $\beta$ we use the
requirement that the differential $\der\Om_2$ does not involve a 
$\Om_2\dz\theta^4$ term. After imposing this, the
remaining free parameters in the definitions of $(\theta^i,\Om_\mu,A)$
are \emph{only}: 
$a,f_{12},f_{22},f_{31},f_{32}$. This shows that the system for a
$(1,2,3)$ type para-CR structure with $J_1\equiv J_2\equiv 0$
naturally 
\emph{closes} on $P$, and that $P$ can be locally parametrized by
$(x,y,z,p,q,s)$ (the base) and $(a,f_{12},f_{22},f_{31},f_{32})$ (fibers). 
\end{remark}
\begin{remark}
Theorem \ref{alb} assures that the solution space of a pair of PDEs 
$z_{xx}=R(x,y,z,z_x,z_y,z_{xy})$ $\&$ $z_{yy}=T(x,y,z,z_x,z_y,z_{xy})$
satisfying $D^2_xT\equiv D^2_yR$ and $J_1\equiv J_2\equiv 0$ is 
naturally equipped with a 
$(+,+,-,-)$ signature conformal structure, and that this conformal
structure is a \emph{point invariant} of the corresponding pair of
PDEs. However the appearence of the \emph{torsion} terms $t^1_{23}$
and $t^2_{13}$ in (\ref{bla}), as well as the \emph{nonhorizontal} terms,
such as e.g. $\Om_2\dz\theta^1$ in $\der\Om_2$, show, that there might
be \emph{many point nonequivalent} PDEs $z_{xx}=R(x,y,z,z_x,z_y,z_{xy})$ $\&$
$z_{yy}=T(x,y,z,z_x,z_y,z_{xy})$ with 
$D^2_xT\equiv D^2_yR$ and $J_1\equiv J_2\equiv 0$, which correspond to
\emph{the same} conformal class of metrics. Although the forms 
$(\Om_1,\Om_2,\Om_3,\Om_4,\Om_5,\Om_6,A)$ together with
$(\theta^1,\theta^2,\theta^3,\theta^4)$, as constructed in the proof of
Theorem \ref{alb} and in the Remark \ref{rbl}, \emph{solve the
  equivalence problem} for the $(1,2,3)$ type para-CR structures
in question, they in general 
do \emph{not} define a Weyl
connection on $\mathcal S$. For this to be possible the torsion
coefficients $t^1_{23}$, $t^2_{13}$, as well as the nonhorizontal
terms in $\der\Om_2$ and $\der\Om_3$ must vanish. 
In the rest of this section we will 
find those point nonequivalent classes of equations $z_{xx}=R(x,y,z,z_x,z_y,z_{xy})$ $\&$
$z_{yy}=T(x,y,z,z_x,z_y,z_{xy})$ for which this is the case.     
\end{remark}
\begin{lemma}\label{le}
The forms (\ref{135})-(\ref{po})-(\ref{poi}) satisfy the differential
system (\ref{fyu}) if and only if they can be brought to the form in
which they satisfy:
\begin{eqnarray}
&&\der\theta^1=(\Omega_1-\tfrac12
A)\dz\theta^1-\Omega_3\dz\theta^3-\Omega_5\dz\theta^4\nonumber\\
&&\der\theta^2=(-\Omega_1-\tfrac12
A)\dz\theta^2-\Omega_2\dz\theta^3-\Omega_4\dz\theta^4\nonumber\\
&&\der\theta^3=\Omega_4\dz\theta^1+\Omega_5\dz\theta^2+(\Omega_6-\tfrac12A)\dz\theta^3\nonumber\\
&&\der\theta^4=\Omega_2\dz\theta^1+\Omega_3\dz\theta^2+(-\Omega_6-\tfrac12A)\dz\theta^4,\nonumber\\
&&\der\Omega_1=\Omega_2\dz\Omega_5-\Omega_3\dz\Omega_4-\varkappa\theta^1\dz\theta^2\nonumber\\
&&\der\Omega_2=\Omega_2\dz(\Omega_1+\Omega_6)+\varkappa\theta^2\dz\theta^4\nonumber\\
&&\der\Omega_3=(\Omega_1-\Omega_6)\dz\Omega_3+\varkappa\theta^1\dz\theta^4\label{fyul}\\
&&\der\Omega_4=\Omega_4\dz(\Omega_1-\Omega_6)+\varkappa\theta^2\dz\theta^3\nonumber\\
&&\der\Omega_5=(\Omega_1+\Omega_6)\dz\Omega_5+\varkappa\theta^1\dz\theta^3\nonumber\\
&&\der\Omega_6=\Omega_2\dz\Omega_5+\Omega_3\dz\Omega_4-\varkappa\theta^3\dz\theta^4\nonumber\\
&&\der A=0,\nonumber\\
&&\der\varkappa=\varkappa A.\nonumber
\end{eqnarray}
\end{lemma}
%home/pawel/notebooks/paracr/123/spr123closure.nb
\begin{proof}
As we noticed in Theorem \ref{alb} the forms
(\ref{135})-(\ref{po})-(\ref{poi}) may satsify the first two of
equations (\ref{fyu}) if and only if $K_1\equiv K_2\equiv 0$, or what
is the same, if and only if $R_{ss}\equiv T_{ss}\equiv 0$. Moreover,
because the forms $(\theta^1,\Om_2,\Om_3,\theta^2,\theta^3,\theta^4)$
are in the class of forms $(\lambda,\mu_1,\mu_2,\nu_1,\nu_2,\nu_3)$
defining the $(1,2,3)$ para-CR structure, the forms
$(\theta^4,\Om_2,\Om_3)$ form a closed differential ideal
corresponding to the integrable distribution $H^-$. Thus, since 
$\der\Om_2\dz\Om_2\dz\Om_3\dz\theta^4\equiv 0$ and
$\der\Om_3\dz\Om_2\dz\Om_3\dz\theta^4\equiv 0$, the only possibility
of satisfaction of the sixth and seventh
equations in (\ref{fyu}) is that:
\begin{eqnarray}
&&\der\Omega_2=\Omega_2\dz(\Omega_1+\Omega_6)+\Gamma_1\dz\theta^4\label{oi}\\
&&\der\Omega_3=(\Omega_1-\Omega_6)\dz\Omega_3+\Gamma_2\dz\theta^4,\nonumber
\end{eqnarray}
with two 1-forms $\Gamma_1,\Gamma_2$ on $P$, which can be chosen such that 
$\Gamma_1=\gamma_{11}\theta^1+\gamma_{12}\theta^2+\gamma_{13}\theta^3$
and
$\Gamma_2=\gamma_{21}\theta^1+\gamma_{22}\theta^2+\gamma_{23}\theta^3$. 
Here $\gamma_{ij}$ are some functions on $P$. Now, one successively
imposes the condition that the differentials of the right hand sides of
the first four of equations (\ref{fyul}), the differentials of the
right hand sides of equations (\ref{oi}), and the differentials of the
right hand sides of the last five
of equations (\ref{fyu}) are zero (they must be, as they are
differentials of the coframe forms $(\theta^i,\Om_\mu)$). This
straightforwardly leads to the conclusion that it is possible 
if and only if (\ref{fyul}) is satisfied. This finishes the proof.
\end{proof}
\begin{theorem}
All finite type systems of PDEs on the
plane $$z_{xx}=R(x,y,z,z_x,z_y,z_{xy})\quad\&\quad z_{yy}=T(x,y,z,z_x,z_y,z_{xy}),$$ which in a natural way define a split
signature Weyl geometry $[g,A]$ on their 4-dimensional solution space,
are locally point equivalent to the system:
\begin{eqnarray}
&&z_{xx}=-\frac{2y z_x z_{xy}}{z+x z_x-y z_y}\quad\&\quad\label{si}\\
&&z_{yy}=-\frac{2\kappa}{y}\frac{z_{xy}}{z+x
  z_x-yz_y}-\frac{2x}{y}\frac{(z-yz_y)z_{xy}}{z+xz_x-yz_y},\nonumber\end{eqnarray}
with $\kappa$ being a real number. All such systems with $\kappa\neq
0$ are locally point equivalent to the system with $\kappa=1$. They are
point nonequivalent with the system with $\kappa\equiv 0$. 
For each $\kappa$ system (\ref{si}) has 
$$z=\frac{\kappa(a_0a_1+a_2a_3)y+\kappa
  a_1-y-a_0y^2-a_3xy}{a_2y-a_1x}$$
as its general solution. The Weyl geometry $[g_\kappa,A_\kappa]$ on the 4-dimensional
solution space, with points parametrized by $(a_0,a_1,a_2,a_3)$, is represented by
$$g_\kappa=\frac{2\big(\der a_0\der a_1+\der a_2\der
  a_3\big)}{\big(1+\kappa(a_0a_1+a_2a_3)\big)^2},\quad\quad\quad\quad
A_\kappa=0.$$
The type $(1,2,3)$ para-CR structures corresponding to the 
two different values $1$ or $0$ are locally nonequivalent. If
$\kappa=0$, then the corresponding $(1,2,3)$ type para-CR structure 
has an 11-dimensional group of symmetries $\cog(2,2)$, and is equivalent to
the $(1,2,3)$ para-CR structure corresponding to the system
$z_{xx}=z_{yy}=0$. If $\kappa\neq 0$, the corresponding type $(1,2,3)$
para-CR
structures have a 10-dimensional group of symmetries isomorphic to 
$\sog(2,3)$. This group acts naturally as the group of motions  
on the solution space, which is equipped with a metric of constant curvature 
$g_\kappa$.  
\end{theorem}
%/home/pawel/notebooks/paracr/123/mat5wawcommutator3exampleinnarep1.nb
\begin{proof}
We first show that the $(1,2,3)$ type para-CR structure associated
with the system (\ref{si}) defines forms $(\theta^i,\Om_\mu,A)$
satisfying (\ref{fyul}). Since we have the general solution of the
system (\ref{si}), it is convenient to use the representation
(\ref{134}), 
rather than (\ref{135}), for the defining forms
$(\lambda,\mu_1,\mu_2,\nu_1,\nu_2,\nu_3)$. Thus, inserting $$\psi=\frac{\kappa(a_0a_1+a_2a_3)y+\kappa
  a_1-y-a_0y^2-a_3xy}{a_2y-a_1x}$$ in (\ref{134}), we have
\begin{eqnarray}
\lambda=-\frac{(a_1\kappa-y)y}{a_1 x-a_2 y}\der
  a_0+\frac{(a_2\kappa-x)y(1+a_3 x+a_0y)}{(a_1 x-a_2 y)^2}\der a_1-
\nonumber\\\frac{(a_1\kappa-y)y(1+a_3 x+a_0y)}{(a_1 x-a_2 y)^2}\der a_2-\frac{(a_2\kappa-x)y}{a_1 x-a_2 y}\der a_3,\nonumber\\
\nu_1=\der a_1,\quad 
\nu_2=\der a_2,\quad 
\nu_3=\der a_3\nonumber\\
\mu_1=\der x,\quad
\mu_2=\der y.\nonumber
\end{eqnarray}   
We now take the forms (\ref{poi}) with these
$(\lambda,\mu_1,\mu_2,\nu_1,\nu_2,\nu_3)$ and apply the procedure of
fixing the gauge as in the proof of Theorems \ref{poj}, \ref{alb} and
Remark \ref{rbl}. This procedure leads to the following choices for the
free parameters $b_1,b_2,b_3,c_1,c_2$, $f_{13}$, $f_{22},f_{23},
f_{31},f_{33}, h_{11},h_{12},h_{21},h_{22}$:
\begin{eqnarray*}
&&b_1=\frac{f_{12}u}{f_{21}f_{32}(a_1\kappa-y)y(1+a_3 x+a_0y)}\\
&&b_2=-\frac{a f_{12}}{f_{32}}\\
&&b_3=\frac{u}{f_{21}(a_1\kappa-y)y(1+a_3 x+a_0y)}\\
&&c_1=-\frac{a\kappa(1+a_3x+a_0
    y)}{f_{21}(1+\kappa(a_0a_1+a_2a_3))(a_1\kappa-y)}\\
&&c_2=-\frac{a(a_1\kappa-y)}{f_{32}(1+\kappa(a_0a_1+a_2a_3))(a_1x-a_2y)}\\
&&f_{13}=\frac{(a_1x-a_2 y)u}{a(a_1\kappa-y)y(1+a_3 x+a_0 y)^2}\\
&&f_{22}=0\\
&&f_{23}=-\frac{(a_1 x-a_2 y)f_{21}}{1+a_3x+a_0y}\\
&&f_{31}=-\frac{(a_2\kappa-x)f_{32}}{a_1\kappa-y}\\
&&f_{33}=0\\
&&h_{11}=-\frac{ay(1+a_3x+a_0y)}{f_{21}(a_1x-a_2 y)^2}\\
&&h_{12}=\frac{a(a_2\kappa-x)y(1+a_3x+a_0y)}{f_{21}(a_1\kappa-y)(a_1x-a_2y)^2}\\
&&h_{21}=\frac{a(a_1\kappa-y)y(a_1+y(a_0a_1+a_2a_3))}{f_{32}(a_1x-a_2y)^3}\\
&&h_{22}=-\frac{a(a_1\kappa-y)y(a_2+x(a_0a_1+a_2a_3))}{f_{32}(a_1x-a_2y)^3}.
\end{eqnarray*}
Here 
\begin{eqnarray*}
u=a_1^2f_{21}f_{32}x^2-a_1y\Big(2a_2f_{21}f_{32}x+af_{11}\kappa(1+a_3x+a_0y)\Big)+\\
y\Big(a_2^2f_{21}f_{32}y+a(1+a_3x+a_0y)\big(f_{12}(x-a_2\kappa)+f_{11}y\big)\Big).
\end{eqnarray*}
It follows from the construction that these normalizations force the
forms $(\theta^1,\theta^2,\theta^3$, $\theta^4,\Om_2,\Om_3)$ to satisfy
the system (\ref{bla}) and the three conditions from remark \ref{rbl}. 
Because of the choice of $z=z(x,y,a_0,a_1,a_2,a_3)$ as the general 
solution to (\ref{si}), it turns out that in these normalizations the forms (\ref{poi}) 
satisfy, in addition (\ref{fyul}), with  
$$\varkappa=-\frac{2\kappa(1+a_3
  x+a_0y)}{f_{21}f_{32}(1+\kappa(a_0a_1+a_2 a_3))^2(a_1x-a_2 y)}.$$

If $\kappa\equiv 0$, we get $\varkappa\equiv 0$, and the system
(\ref{fyul}) becomes (\ref{sy1p})-(\ref{sy11p}). This proves that if
$\kappa\equiv 0$, then the system (\ref{si}) is point equivalent to
$z_{xx}=z_{yy}=0$, or what is the same, that the corresponding
$(1,2,3)$ type para-CR structure is locally equivalent to the flat one
described by Theorem \ref{pfa1}. 

%/paracr/123/mat5wawcommutator3exampleinnarep1redukcja.nb
If $\kappa\neq 0$ we normalize $\varkappa$ to $\varkappa=1$ by choosing 
$$f_{32}=-\frac{2\kappa(1+a_3
  x+a_0y)}{f_{21}(1+\kappa(a_0a_1+a_2 a_3))^2(a_1x-a_2 y)}.$$
This choice reduces $P$ to a 10-dimensional manifold $P_0$, with
coordinates $(x,y,z,p,q$, $s,a,f_{11},f_{12},f_{21},f_{22})$, on which
$A=0$ and the ten linearly independent 
1-forms $(\theta^i,\Om_\mu)$ satisfy the system
\begin{eqnarray}
&&\der\theta^1=\Omega_1\dz\theta^1-\Omega_3\dz\theta^3-\Omega_5\dz\theta^4\nonumber\\
&&\der\theta^2=-\Omega_1\dz\theta^2-\Omega_2\dz\theta^3-\Omega_4\dz\theta^4\nonumber\\
&&\der\theta^3=\Omega_4\dz\theta^1+\Omega_5\dz\theta^2+\Omega_6\dz\theta^3\nonumber\\
&&\der\theta^4=\Omega_2\dz\theta^1+\Omega_3\dz\theta^2-\Omega_6\dz\theta^4,\nonumber\\
&&\der\Omega_1=\Omega_2\dz\Omega_5-\Omega_3\dz\Omega_4-\theta^1\dz\theta^2\label{fyull}\\
&&\der\Omega_2=\Omega_2\dz(\Omega_1+\Omega_6)+\theta^2\dz\theta^4\nonumber\\
&&\der\Omega_3=(\Omega_1-\Omega_6)\dz\Omega_3+\theta^1\dz\theta^4\nonumber\\
&&\der\Omega_4=\Omega_4\dz(\Omega_1-\Omega_6)+\theta^2\dz\theta^3\nonumber\\
&&\der\Omega_5=(\Omega_1+\Omega_6)\dz\Omega_5+\theta^1\dz\theta^3\nonumber\\
&&\der\Omega_6=\Omega_2\dz\Omega_5+\Omega_3\dz\Omega_4-\theta^3\dz\theta^4.\nonumber
\end{eqnarray} 
Since in these relations only constant coefficients appear on the
right hand sides, $P_0$ is locally a Lie group, with the forms
$(\theta^i,\Om_\mu)$ as its left invariant forms. This group is
isomorphic to $\sog(2,3)$ and, it follows from the Cartan equivalence
method, that it is the \emph{full} symmetry group of the type
$(1,2,3)$ para-CR structure corresponding to (\ref{si}) with
$\kappa\neq 0$. Accordingly it is also the full group of local point
symmetries of the system (\ref{si}) with $\kappa\neq 0$. The
appearence of the group $\sog(2,3)$ is not accidental, since one can
check that the so normalized forms
$(\theta^1,\theta^2,\theta^3,\theta^4)$ satisfy
$$G=2(\theta^1\theta^2+\theta^3\theta^4)=\frac{4\kappa(\der a_0\der
  a_1+\der a_2\der a_3)}{(1+\kappa(a_0a_1+a_2a_3))^2}.$$
This means that the 4-dimensional solution space $\mathcal S$ of the system
(\ref{si}) with $\kappa\neq 0$ is naturally equipped with a
split-signature \emph{constant curvature} metric $G$. The symmetry
group of the pseudoriemannian structure $({\mathcal S},G)$ is
obvioulsy $\sog(2,3)$.   
 
Since the parameter $\kappa$ does not appear in the equations
(\ref{fyull}), we conclude that $\kappa\neq 0$ can always
be brought to $\kappa=1$ by a point transformation of (\ref{si}), or
what is the same, by a para-CR diffeomorphism of the corresponding
para-CR structure. This proves that among type $(1,2,,3)$ para-CR
structures associated with (\ref{si}) there are only two para-CR
nonequivalent ones: the one with $\kappa=0$, and those with $\kappa\neq
0$, which are all locally equivalent to the one with $\kappa=1$.

To prove that these two structures, modulo para-CR equivalence, are the only
ones that satisfy Lemma \ref{le}, we proceed as follows:

Suppose that we have a finite type system of PDEs
$z_{xx}=R(x,y,z,z_z,z_y,z_{xy})$ $\&$
$z_{yy}=T(x,y,z,z_z,z_y,z_{xy})$, which via the procedure described in
Theorems \ref{poj}, \ref{alb} and Remark \ref{rbl}, leads to the
differential system (\ref{fyul}), as in Lemma \ref{le}. If we have 
$\varkappa\equiv 0$, then our PDEs are point equivalent to
$z_{xx}=z_{yy}=0$. If $\varkappa\neq 0$ then the last equation
(\ref{fyul}) says that $A=\frac{\der\varkappa}{\varkappa}$. Then
putting $\epsilon={\rm sign}\varkappa$ we rescale the forms
$(\theta^1,\theta^2,\theta^3,\theta^4)$ to
$$(\bar{\theta}^1,\bar{\theta}^2,\bar{\theta}^3,\bar{\theta}^4)=\big(\epsilon\varkappa\big)^{\frac{1}{2}}(\theta^1,\theta^2,\theta^3,\theta^4).$$
Obviously this rescaling is a para-CR transformation. The advantage of
this rescaling is that, after it, the form $A$ disappears from the
first ten equations (\ref{fyul}). Explicitly, after the rescaling,  
the system (\ref{fyul}) becomes:
\begin{eqnarray}
&&\der\bar{\theta}^1=\Omega_1\dz\bar{\theta}^1-\Omega_3\dz\bar{\theta}^3-\Omega_5\dz\bar{\theta}^4\nonumber\\
&&\der\bar{\theta}^2=-\Omega_1\dz\bar{\theta}^2-\Omega_2\dz\bar{\theta}^3-\Omega_4\dz\bar{\theta}^4\nonumber\\
&&\der\bar{\theta}^3=\Omega_4\dz\bar{\theta}^1+\Omega_5\dz\bar{\theta}^2+\Omega_6\dz\bar{\theta}^3\nonumber\\
&&\der\bar{\theta}^4=\Omega_2\dz\bar{\theta}^1+\Omega_3\dz\bar{\theta}^2-\Omega_6\dz\bar{\theta}^4,\nonumber\\
&&\der\Omega_1=\Omega_2\dz\Omega_5-\Omega_3\dz\Omega_4-\epsilon\bar{\theta}^1\dz\bar{\theta}^2\nonumber\\
&&\der\Omega_2=\Omega_2\dz(\Omega_1+\Omega_6)+\epsilon\bar{\theta}^2\dz\bar{\theta}^4\label{fyulll}\\
&&\der\Omega_3=(\Omega_1-\Omega_6)\dz\Omega_3+\epsilon\bar{\theta}^1\dz\bar{\theta}^4\nonumber\\
&&\der\Omega_4=\Omega_4\dz(\Omega_1-\Omega_6)+\epsilon\bar{\theta}^2\dz\bar{\theta}^3\nonumber\\
&&\der\Omega_5=(\Omega_1+\Omega_6)\dz\Omega_5+\epsilon\bar{\theta}^1\dz\bar{\theta}^3\nonumber\\
&&\der\Omega_6=\Omega_2\dz\Omega_5+\Omega_3\dz\Omega_4-\epsilon\bar{\theta}^3\dz\bar{\theta}^4\nonumber\\
&&A=\frac{\der\varkappa}{\varkappa}.\nonumber
\end{eqnarray} 
This shows that if $\varkappa\neq 0$ we can always reduce the system
to 10 dimensions, and that there are \emph{at most} two different
para-CR structures with such $\varkappa$, corresponding to the
different signs of $\varepsilon$. However, a discrete para-CR transformation
on this system, transforming 
$$(\bar{\theta}^1,\bar{\theta}^3,\Om_2,\Om_5)\to
(-\bar{\theta}^1,-\bar{\theta}^3,-\Om_2,-\Om_5),$$
and being the identity on the rest of the coframe forms, brings the system
(\ref{fyulll}) into the form (\ref{fyull}), in which
$\epsilon=+1$. This shows that the para-CR structures with different values of
$\varepsilon$ are equivalent, and that there are only two, locally
nonequivalent type $(1,2,3)$ para-CR structures satisfying system
(\ref{fyul}). We found the representatives of both of them, as the
para-CR structures corresponding to $\kappa=0$ or $\kappa=1$ in
(\ref{si}). This finishes the proof.
\end{proof}
\section{Para-CR structures of type $(3,2,1)$}\label{buru}
\subsection{Type $(3,2,1)$ versus $(1,2,3)$}\label{321}
As noted in Section
\ref{n111}, the flip $(1,1,n-1)\to(n-1,1,1)$, changes 
a para-CR structure corresponding to an $n$th order ODE considered
modulo \emph{point} transformations, to a para-CR structure corresponding to an $n$th
order ODE considered modulo \emph{contact} transformations. In this 
section we further investigate the meaning of the flip 
$$(k,r,s)\to(s,r,k),$$
on an example of type $(k=1,r=2,s=3)$ 
para-CR structures corresponding to PDEs
(\ref{dhpndh}). We expect that the passage $(1,2,3)\to(3,2,1)$ will
again change the geometric setting in such a way that the type
$(1,2,3)$ para-CR structure corresponding to PDEs (\ref{dhpndh})
considered modulo \emph{point} transformations will become 
a para-CR structure corresponding to the same pair of PDEs but
considered modulo
\emph{contact} transformations.

That this is really the case can be seen from the following:

Given a pair of equations
$$z_{xx}=R(x,y,z,z_x,z_y,z_{xy})\quad\quad\&\quad\quad
z_{yy}=T(x,y,z,z_x,z_y,z_{xy})$$
we use the contact forms $\lambda=\der z-p\der x-q\der y$, 
$\nu_1=\der p-R\der x-s\der y$, $\nu_2=\der q-s\der x-T\der y$, 
$\nu_3=\der s-D_yR\der x-D_xT\der y$, $\mu_1=\der x$, 
$\mu_2=\der y$ on the 6-dimensional jet space $\mathcal J$ parametrized by
$(x,y,z,p,q,s)$ . It is easy to see that when the equations undergo a
\emph{point} transformation of variables, then the forms change according to:
\be
\bma
\la\\
\nu_1\\
\nu_2\\
\nu_3\\
\mu_1\\
\mu_2
\ema\to\bma
\la'\\
\nu_1'\\
\nu_2'\\
\nu_3'\\
\mu_1'\\
\mu_2'
\ema =\bma
a&0&0&0&0&0\\
b_{1}&f_{11}&f_{12}&0&0&0\\
b_{2}&f_{21}&f_{22}&0&0&0\\
b_{3}&f_{31}&f_{32}&f_{33}&0&0\\
c_{1}&0&0&0&h_{11}&h_{12}\\
c_{2}&0&0&0&h_{21}&h_{22}
\ema
\bma
\la\\
\nu_1\\
\nu_2\\
\nu_3\\
\mu_1\\
\mu_2
\ema,
\label{poii}\ee
and when the equations undergo a \emph{contact} transformation of
variables the forms change as:
\be
\bma
\la\\
\nu_1\\
\nu_2\\
\nu_3\\
\mu_1\\
\mu_2
\ema\to\bma
\la'\\
\nu_1'\\
\nu_2'\\
\nu_3'\\
\mu_1'\\
\mu_2'
\ema =\bma
a&0&0&0&0&0\\
b_{1}&f_{11}&f_{12}&0&0&0\\
b_{2}&f_{21}&f_{22}&0&0&0\\
b_{3}&f_{31}&f_{32}&f_{33}&0&0\\
c_{1}&u_{11}&u_{12}&0&h_{11}&h_{12}\\
c_{2}&u_{21}&u_{22}&0&h_{21}&h_{22}
\ema
\bma
\la\\
\nu_1\\
\nu_2\\
\nu_3\\
\mu_1\\
\mu_2
\ema.\label{ctn}
\ee
Introducing vector fields $(Z,X_1,X_2,Y_1,Y_2,Y_3)$, which are
respective duals to the coframe
$(\lambda,\mu_1,\mu_2,\nu_1,\nu_2,\nu_3)$, we easily see that under the
\emph{point} transformations they transform according to:
$$\bma
Z\\
Y_1\\
Y_2\\
Y_3\\
X_1\\
X_2
\ema\to\bma
Z'\\
Y_1'\\
Y_2'\\
Y_3'\\
X_1'\\
X_2'
\ema =\bma
*&*&*&*&*&*\\
0&*&*&*&0&0\\
0&*&*&*&0&0\\
0&0&0&*&0&0\\
0&0&0&0&*&*\\
0&0&0&0&*&*
\ema
\bma
Z\\
Y_1\\
Y_2\\
Y_3\\
X_1\\
X_2
\ema,$$
and under the \emph{contact} transformations they transform according
to:
$$\bma
Z\\
Y_1\\
Y_2\\
Y_3\\
X_1\\
X_2
\ema\to\bma
Z'\\
Y_1'\\
Y_2'\\
Y_3'\\
X_1'\\
X_2'
\ema =\bma
*&*&*&*&*&*\\
0&*&*&*&*&*\\
0&*&*&*&*&*\\
0&0&0&*&0&0\\
0&0&0&0&*&*\\
0&0&0&0&*&*
\ema
\bma
Z\\
Y_1\\
Y_2\\
Y_3\\
X_1\\
X_2
\ema,$$
where by $*$ we denoted the matrix entries that are nonzero. This
shows that the \emph{point} transformations preserve the two vector
spaces: 2-dimensional 
$H^+=\Span(X_1,X_2)$ and 3-dimensional $H^-_{\rm point}=\Span(Y_1,Y_2,Y_3)$, while the contact
transformations preserve $H^+$ and only a 1-dimensional $H^-_{\rm contact}=\Span(Y_3)$. We
have the following proposition:
\begin{proposition}
Assume that a pair of equations $$z_{xx}=R(x,y,z,z_x,z_y,z_{xy})\quad\quad\&\quad\quad
z_{yy}=T(x,y,z,z_x,z_y,z_{xy})$$ satisfies the compatibility
conditions $D^2_xT=D^2_yR$, where 
$D_x=\partial_x+p\partial_z+R\partial_p+s\partial_q+D_yR\partial_s$,
$D_y=\partial_y+q\partial_z+T\partial_q+s\partial_p+D_xT\partial_s$ 
and $p=z_x,q=z_y,s=z_{xy}$. Let $H^+=\Span(D_x,D_y)$, $H^-_{\rm
  point}=\Span(Y_1,Y_2,Y_3)$, and $H^-_{\rm contact}=\Span(Y_3)$, with 
$Y_1=\partial_p, Y_2=\partial_q, Y_3=\partial_s$, be three
distributions, with respective dimensions 2, 3, 1, on the
6-dimensional jet space $\mathcal J$ parametrized by
$(x,y,z,p,q,s)$. Then:

If this pair of equations is considered modulo \emph{point} transformation of
variables, it defines a type 
$(1,2,3)$ para-CR structure $({\mathcal J},H^+,H^-_{\rm
  point})$ on $\mathcal J$.

If this pair of equations is considered
modulo \emph{contact} transformations, it defines a type $(3,2,1)$ para-CR
structure $({\mathcal J},H^+,H^-_{\rm
  contact})$ on $\mathcal J$.
\end{proposition}   
\begin{proof}
In view of the discussion preceeding the Proposition, the only thing to
be proven is that the distributions $H^+$ and $H^-_{\rm point}$ are
integrable on $\mathcal J$. Using the local coordinates
$(x,y,z,p,q,s)$ we see that the duals to a coframe
$(\lambda,\mu_1,\mu_2,\nu_1,\nu_2,\nu_3)$ are
$(Z=\partial_z,X_1=D_x,X_2=D_y,Y_1=\partial_p,Y_2=\partial_q,Y_3=\partial_s)$. Hence, obviously, $H^-_{\rm point}$
is integrable. Calculating the commutator $[D_x,D_y]$ we get
$[D_x,D_y]=(D^2_xT-D^2_yR)\partial_s$, which vanishes due to our
assumptions. Thus also $H^+$ is integrable. 
\end{proof}
\subsection{Towards invariants for type $(3,2,1)$}\label{tow}
The contact transformations (\ref{ctn}) are \emph{more restrictive}
than the most general para-CR transformations 
\be
\bma
l_1\\
l_2\\
l_3\\
n\\
m_1\\
m_2
\ema\to\bma
l_1'\\
l_2'\\
l_3'\\
n'\\
m_1'\\
m_2'
\ema =\bma
a&a_{11}&a_{12}&0&0&0\\
b_{1}&f_{11}&f_{12}&0&0&0\\
b_{2}&f_{21}&f_{22}&0&0&0\\
b_{3}&f_{31}&f_{32}&f_{33}&0&0\\
c_{1}&u_{11}&u_{12}&0&h_{11}&h_{12}\\
c_{2}&u_{21}&u_{22}&0&h_{21}&h_{22}
\ema
\bma
l_1\\
l_2\\
l_3\\
n\\
m_1\\
m_2
\ema=\bma
\theta^4\\
\theta^1\\
\theta^2\\
\theta^3\\
\Om_3\\
\Om_2
\ema,\label{pio}
\ee
of a $(3,2,1)$-para-CR structure $[l_1,l_2,l_3,m_1,m_2,n]$ defined on
$\mathcal J$ by $l_1=\lambda=\der z-p\der x-q\der y$, 
$l_2=\nu_1=\der p-R\der x-s\der y$, $l_3=\nu_2=\der q-s\der x-T\der y$, 
$n=\nu_3=\der s-D_yR\der x-D_xT\der y$, $m_1=\mu_1=\der x$, 
$m_2=\mu_2=\der y$. However, when looking for the local
invariants for such structures, we can easily normalize the unwanted
$a_{11}$ and $a_{12}$ parameters in these transformations to $a_{11}=0$ and $a_{22}=0$ by the
requirement that the invariant forms $(\theta^i)$ satisfy a
consequence of (\ref{uiu1}), i.e.:
\be\der\theta^4\dz\theta^1\dz\theta^2\dz\theta^4=0.\label{zeo}\ee
%/home/pawel/notebooks/paracr/123/para321a12a13zero.nb
It is easy to see that (\ref{zeo}) necessarily implies $a_{11}=0$ and
$a_{22}=0$. Since condition (\ref{uiu1}) is needed to have a conformal
metric on the solution space, from now on we will assume (\ref{zeo}), and as a
consequence
$$a_{11}=a_{12}=0.$$
In such a case the para-CR transformations (\ref{pio}) become the contact
transformations\footnote{Note that
  the situation here is similar to the situation in the \emph{point}
  invariant case. There the para-CR transformations (\ref{poi}) of a
  $(1,2,3)$-type para-CR structure associated with the system of PDEs
  (\ref{dhpndh}) differed from the point transformations (\ref{poii}),
  by the appearence of the nonzero parameters $f_{13}$ and $f_{23}$ in
  (\ref{poi}. But one of the consequences of equations
  (\ref{uiu1})-(\ref{uiu5}) was that $f_{13}=f_{23}=0$, (see
  (\ref{iwi})), which proved that the para-CR transformations
  (\ref{poi}) and the point transformations (\ref{poii}) were equivalent.} for the associated system of PDEs
$z_{xx}=R(x,y,z,z_x,z_y,z_{xy})$, $z_{yy}=T(x,y,z,z_x,z_y,z_{xy})$. As
in the previous sections we assume in addition that
$D_x^2T=D_y^2R$, but release the $1-R_sT_s>0$ condition to
$1-R_sT_s\neq 0$. We have the following theorem.
%/home/pawel/notebooks/paracr/123/metricalanewman.nb
\begin{theorem}\label{tnf}
Given a pair of PDEs on the plane $z_{xx}=R(x,y,z,z_x,z_y,z_{xy})$
$\&$ $z_{yy}=T(x,y,z,z_x,z_y,z_{xy})$ satisfying $D_x^2T=D_y^2R$ and
$1-R_sT_s\neq 0$, the
condition 
$$J_1\equiv 0,\quad\quad\&\quad\quad J_2\equiv 0,$$
where 
\begin{eqnarray*}
&&J_1=(R_sT_s-4)D_xR_s+R_s(2D_yR_s-R_sD_xT_s)+\\
&&8R_q-6R_qR_sT_s+4R_pR_s+2R_s^2T_q-2R_pR_s^2T_s+2R_s^3T_p\\
&&\\
&&J_2=(R_sT_s-4)D_yT_s+T_s(2D_xT_s-T_sD_yR_s)+\\
&&8T_p-6R_sT_pT_s+4T_qT_s+2R_pT_s^2-2R_sT_qT_s^2+2R_qT_s^3,
\end{eqnarray*}
is preserved under the \emph{contact} transformations of the
variables. If this condition is satisfied the 4-dimensional solution space of the
PDEs is naturally equipped with a conformal class $[g]$ of metrics. If $$1-R_sT_s>0$$ these conformal metrics have \emph{split}
signature. If $$1-R_sT_s<0$$ the metrics have \emph{Lorentzian}
signature. The conformal class $[g]$ is invariant under the contact
transformations of the variables of the PDEs. 
\end{theorem}
We also have a useful Proposition, which gives local representatives of
the conformal class $[g]$ from the above Theorem:
\begin{proposition}\label{pnf}
If 
$R_sT_s\neq 4$ a representative $g$ of the conformal class
$[g]$ can be chosen so that it is given by
\be
g=2 \la \om~+~2(R_s T_s-4)( T_s\nu_1^2-2 \nu_1\nu_2+R_s\nu_2^2),\label{mne1}\ee
where
\begin{eqnarray*}
&&\om=(4D_xT_s-2T_sD_yR_s +4 R_pT_s-2R_s^2T_pT_s-2R_sT_qT_s+4R_qT_s^2)\nu_1+\\
&&\quad(4D_yR_s-2R_sD_xT_s+4R_sT_q-2R_qR_sT_s^2-2R_pR_sT_s+4R_s^2T_p)\nu_2+\\
&&\quad2(4-R_sT_s)(R_sT_s-1)\nu_3+ v \la,\\
&&\lambda=\der z-p\der x-q\der y, \\
&&\nu_1=\der p-R\der x-s\der
y,\quad \nu_2=\der q-s\der x-T\der y,\quad  \nu_3=\der s-D_yR\der x-D_xT\der y,\end{eqnarray*}
and 
\begin{eqnarray*}
&&2v=8 D_xT_q-4 D_y^2R_s+4 (D_xT_s) D_yR_s+4 R_sD_xT_p -4 R_sD_yT_q -4
 R_s^2D_yT_p+\\&&
8 R_q T_p-14 R_s T_pD_yR_s +4 R_p R_s T_p+3 R_s^2 T_pD_xT_s -6 R_s^3
T_p^2-4 T_qD_yR_s +\\&&
4 R_s T_qD_xT_s -6 R_s^2 T_p T_q+8 T_sD_yR_q -2 (D_yR_s)^2 T_s+4 R_p T_sD_yR_s
-2 R_s T_sD_xT_q +\\&&
 R_s T_sD_y^2R_s+4R_s T_s D_yR_p -R_s^2 T_sD_xT_p +R_s^2 T_sD_yT_q
 +R_s^3 T_sD_yT_p +8 R_z T_s+\\&&2 R_q R_s T_p T_s+2 R_p R_s^2 T_p
 T_s+8 R_q T_q T_s-3R_s T_q T_s D_yR_s +4 R_p R_s T_q T_s-\\&&2 R_s^3
 T_p T_q T_s-2 R_s^2 T_q^2 T_s+4 R_q T_s^2D_yR_s -2 R_s T_s^2D_yR_q -
 R_s^2 T_s^2D_yR_p-2 R_s R_z T_s^2+\\&&2 R_q R_s^2 T_p T_s^2+2 R_q R_s
 T_q T_s^2+8 R_s T_z-2 R_s^2 T_s T_z.
\end{eqnarray*}
If $R_sT_s\neq 0$ another representative $g$ of $[g]$ may be chosen so that:
\be
g=2 \la \om'~+~ T_s\nu_1^2-2 \nu_1\nu_2+R_s\nu_2^2,\label{mne2}\ee
where 
\begin{eqnarray*}
&&\om'=\frac{-D_yT_s+2T_p - R_sT_pT_s+T_qT_s}{T_s}\nu_1
+\frac{-D_xR_s+2R_q - R_qR_sT_s+R_pR_s}{R_s}\nu_2+\\
&&\quad\quad(1-R_sT_s)\nu_3- \frac{v'}{2R_s^3T_s} \la,\end{eqnarray*}
with $\la$, $\nu_1$, $\nu_2$ and $\nu_3$ as before,  
and 
\begin{eqnarray*}
&&v'=
2 R_s^2(D_xR_s)D_yT_s -4 R_q R_s^2D_yT_s -R_p R_s^3D_yT_s -4 R_s^2 T_pD_xR_s
+\\&&8 R_q R_s^2 T_p+2 R_p R_s^3 
T_p+2 (D_xR_s)^2 T_s-8  R_q T_sD_xR_s+8 R_q^2 T_s-\\&&2 R_p R_s T_sD_xR_s +4 R_p R_q R_s T_s-R_s^2 T_sD_xD_yR_s +2 R_s^2 T_sD_yR_q +R_s^3 T_sD_xT_q +\\&&R_s^3 T_sD_yR_p -R_q R_s^3 T_p T_s-3 R_s^2 T_q T_sD_xR_s +6 R_q R_s^2 T_q T_s+R_p R_s^3 T_q T_s+\\&&2 R_q R_s T_s^2D_xR_s -4 R_q^2 R_s T_s^2+R_s^3 R_z T_s^2+R_s^4 T_s T_z.
\end{eqnarray*}  
\end{proposition} 
%/home/pawel/notebooks/paracr/123/para321a12a13zeroj1j2.nb
%/home/pawel/notebooks/paracr/123/para321rstseq4waw1.nb i takze bez numeru
\begin{proof} (of the Proposition and the Theorem).
We start by forcing the \emph{contact} invariant
forms $(\theta^1,\theta^2,\theta^3,\theta^4,\Om_2,\Om_3)$ given in 
(\ref{pio}) to satisfy the \emph{first four} equations
(\ref{fyu}). We do it in several steps. The first step consists in
the requirement that $(\theta^1,\theta^2,\theta^3,\theta^4,\Om_2,\Om_3)$
satisfy consequences of equations (\ref{fyu}), namely equations
(\ref{uiu1})-(\ref{uiu5}). The first of these conditions implies
$\der\theta^4\dz\theta^1\dz\theta^2\dz\theta^4=0$, and this, as noted
before, implies  $a_{11}=a_{12}=0$. 

Let us now, unless otherwise stated, assume that $1-R_sT_s>0$. 
Then the conditions
(\ref{uiu1})-(\ref{uiu5}) can be easily fulfilled by taking
$u_{11}=u_{12}=u_{21}=u_{22}=0$ in (\ref{pio}), since this enables us to
identify forms (\ref{pio}) with (\ref{poi}). After this identification
the imposition of the rest of conditions (\ref{uiu1})-(\ref{uiu5}) may
be obtained by making the same 
normalizations of parameters $h_{11}$, $h_{21}$, $h_{12}$, $h_{21}$,
$f_{21}$ and $f_{11}$ as in the proof of Theorem \ref{poj}. It follows
however, that one can achieve (\ref{uiu1})-(\ref{uiu5}) \emph{without}
the restriction $u_{11}=u_{12}=u_{21}=u_{22}=0$ on the parameters
$u_{11},u_{12},u_{21}$ and $u_{22}$. We checked that the most general
normalizations to achieve (\ref{uiu1})-(\ref{uiu3}) is to take $h_{11}$, $h_{21}$, $h_{12}$, $h_{21}$,
$f_{21}$ and $f_{11}$ as in (\ref{iwi1})-(\ref{iwi2}) and to restrict
$u_{11},u_{12},u_{21}$ and $u_{22}$ by only \emph{one} constraint  
\be
u_{22} f_{11}-u_{21} f_{12}+u_{12} f_{21}-u_{11} f_{22}=0\label{ivi}.\ee
If this is not zero, equation (\ref{uiu1}) has an unwanted term
proportional to $\theta^1\dz\theta^2\dz\theta^4$  on the right hand side. Even without the
restriction (\ref{ivi}), but assuming (\ref{iwi1})-(\ref{iwi2}), we
get that $\der\theta^1\dz\theta^1\dz\theta^3\dz\theta^4\dz\Om_3$ and
$\der\theta^2\dz\theta^2\dz\theta^3\dz\theta^4\dz\Om_2$ are still
given by (\ref{iwi3}). This proves that the conditions $J_1\equiv
J_2\equiv 0$ are neccessary for a conformal metric $g$ to be defined
on the solution space. It also proves that these conditions are
\emph{contact invariant}. This surely holds when our assumption
$1-R_sTs>0$ is satisfied. (That this assumption is only a technical
one will be clear soon). So from now on we assume the normalizations  
(\ref{iwi1})-(\ref{iwi2}), (\ref{ivi}) and that the invariants $J_1$ and $J_2$ are
both zero, $J_1\equiv J_2\equiv 0$. 

Now it follows that the conditions 
$\der\theta^1\dz\theta^1\dz\theta^2\dz\theta^4=-\Om_3\dz\theta^1\dz\theta^2\dz\theta^3\dz\theta^4$, 
$\der\theta^2\dz\theta^1\dz\theta^2\dz\theta^4=-\Om_2\dz\theta^1\dz\theta^2\dz\theta^3\dz\theta^4$,
$\der\theta^2\dz\theta^2\dz\theta^3\dz\theta^4=0$ and $\der\theta^1\dz\theta^1\dz\theta^3\dz\theta^4=0$ 
are equivalent to precisely \emph{the same} normalizations
(\ref{iwi4}), (\ref{iwi5}) and (\ref{iwi6}) of
$f_{33}$, $b_2$ and $b_1$ as in the proof of  Theorem
\ref{alb}. Further repetition, step by step, of the
absorbtion/normalization procedure described in the proof of
Theorem \ref{alb} leads to the last relevant normalization, which
determines the coefficient $b_3$. Here, again this coeffcient turns out
to be precisely as in the proof of Theorem \ref{alb}. That
the present expressions for the determined parameters $h_{11}$, $h_{21}$, $h_{12}$, $h_{21}$,
$f_{21}$, $f_{11}$, $f_{33}$, $b_1$, $b_2$ and $b_3$ do \emph{not depend} on the parameters $u_{11},
u_{12}, u_{21}$ and $u_{22}$ is remarkable. They are invisible because
they turn out to parametrize only that part of the contact
transformations, which is related to the orthogonal group preserving the metric $g$ we are going to
construct. 

Indeed, assuming $J_1\equiv J_2\equiv 0$ and the above discussed
normalizations 
for $h_{11}$, $h_{21}$, $h_{12}$, $h_{21}$,
$f_{21}$, $f_{11}$, $f_{33}$, $b_1$, $b_2$, $b_3$, we calculate
$G=2(\theta^1\theta^2+\theta^3\theta^4)$. A direct calculation shows
then, that the resulting expression
for $G$ has \emph{no} $u_{11},
u_{12}, u_{21}, u_{22}$ dependence! Moreover, the so obtained $G$ is
also \emph{independent} of still undetermined parameters $a$, $b_3$,
$f_{31}$, $f_{32}$, $c_1$ and $c_2$. Its dependence on the
parameters $f_{12}$
and $f_{22}$ is only \emph{conformal}. By this we mean that the parameters $f_{12}$ and
$f_{22}$ only appear as a common
factor $f_{12}f_{22}$ in front of the entire expression for $G$. This means that\emph{all} the
remaining free parameters $a$, $b_3$,
$f_{31}$, $f_{32}$, $c_1$, $c_2$, $u_{11}, u_{12}, u_{21}$ and
$u_{22}$ are \emph{group parameters} of the 
dilation group ${\bf CO}(G)$ preserving conformally the bilinear form
$G$. 

If one wants the explicit expressions for $G$, with the above
normalizations for  $h_{11}$, $h_{21}$, $h_{12}$, $h_{21}$,
$f_{21}$, $f_{11}$, $f_{33}$, $b_1$, $b_2$ and $b_3$, in terms of the
functions $R$ and $T$ defining the system $z_{xx}=R$ $\&$ $z_{yy}=T$,
one has to decide how to mod the resulting formula by the constraints
$J_1\equiv J_2\equiv 0$.

It follows that if we write
$J_1\equiv J_2\equiv 0$ in the form $D_xR_s=\dots$ and $D_yT_s=\dots$,
and eliminate these derivatives from $G$, 
then we obtain $G$, which up to a factor, coincides with $g$ from
formula (\ref{mne1}). Similarly, if we write these
condtions as $D_xT_s=\dots$ and $D_yR_s=\dots$, we get the result that
$G$ differs from formula (\ref{mne2}) only by a factor. This proves 
that the bilinear forms $g$ as in (\ref{mne1}) and
(\ref{mne2}) are conformally invariant on ${\mathcal J}$, and that they change
conformally when the system $z_{xx}=R$ $\&$ $z_{yy}=T$ undergoes
\emph{contact} transformation of the variables.

The last thing is to prove that $[g]$ is actually defined on the
solution space of the PDEs, and that it is nondegenerate there with signature
depending on the sign of $1-R_sT_s$.

Let us start to comment on these last issues with a remark about
the technicality of our assumption $1-R_sT_s>0$. 
We needed the assumption $1-R_sT_s>0$ starting with the
normalization (\ref{iwi2}). It was needed there to maintain the invariant
forms $\theta^i$ to be \emph{real}. But this was only made for
simplicity, since we did not want to deal with the complex numbers in
the proof. Moreover, from the point of view of the conformal metric we
wanted to construct, this was a good simplification since in the
resulting formulae (\ref{mne1}), (\ref{mne2}) for $g$ the square root 
$\sqrt{1-R_sT_s}$ does not appear at all! Concluding this issue, we
say that if we were in the situation when $1-R_sT_s<0$, our
normalizing procedure for the forms $\theta^i$ would make them
complex, but the resulting $G$ would nevertheless be real and given by
(\ref{mne1}) or (\ref{mne2}). Thus all the conformal properties of $g$
established so far are also valid in the $1-R_sT_s<0$ case.   

There is one more technical issue here. The reason for having two
different expressions for $g$, as in (\ref{mne1}) and (\ref{mne2}), is
to have local expressions valid everywhere off the set $1-R_s T_s=0$. 
Since solving for $D_xR_s$ and $D_yT_s$ in
$J_1\equiv J_2\equiv 0$ we devide by $(4-R_sT_s)$, the metric 
(\ref{mne1}) is only defined if $R_sT_s\neq 4$; similarly, because of
the division by $R_sT_s$, the metric (\ref{mne2}) is defined only if $R_sT_s\neq
0$. Off the set $R_sT_s=0=4-R_sT_s$ the conformal metrics (\ref{mne1})
and (\ref{mne2}) 
coincide, since they are local manifestations of the same formula
$G=2(\theta^1\theta^2+\theta^3\theta^4)$ on $\mathcal J$.

Finally we comment on how $G$ descends to the solution space of the
PDEs. 

We start with an observation that the bilinear form (\ref{mne1}) 
satisfies $g(D_x,\cdot)=g(D_y,\cdot)\equiv 0$, i.e. it is \emph{degenerate} along the vector
fields $D_x$ and $D_y$ on $\mathcal J$. The first product $2\la\om$ in
(\ref{mne1}) has obviously signature $(+,-)$. Thus to determine the signature
of (\ref{mne1}) we need to determine 
the signature of the product $$2(R_s T_s-4)( T_s\nu_1^2-2
\nu_1\nu_2+R_s\nu_2^2).$$ Since the quadratic form $T_s\nu_1^2-2
\nu_1\nu_2+R_s\nu_2^2$ has $\Delta=4(1-R_sT_s)$ as its discriminant,
then the signature of the product $2(R_s T_s-4)( T_s\nu_1^2-2
\nu_1\nu_2+R_s\nu_2^2)$ is: $\pm(+,-)$ iff $1-R_sT_s>0$ and $\pm(+,+)$ iff
$1-R_sT_s<0$. Thus, assuming that $R_sT_s\neq 4$, we conclude that, modulo
the degenerate directions $D_x$ and $D_y$ along which $g$ is
vanishing, the bilinear form (\ref{mne1}) 
has either \emph{split} (iff $1-R_sT_s>0$), or \emph{Lorentzian}
signature (iff $1-R_sT_s<0$) on $\mathcal J$. 

%/dos/123/para321rstseq4waw1metryka.nb
%i takze ten sam file pod linuxem
A straightforward, but lengthy (!), calculation shows that the Lie
derivatives of $g$, from formula (\ref{mne1}), with respect to the degenerate directions $D_x$ and
$D_y$ are:
$${\mathcal L}_{D_x}g=\al(D_x)g,\quad\&\quad {\mathcal
  L}_{D_y}g=\al(D_y)g,$$
where 
\begin{eqnarray*}
&&\al(D_x)=(4-R_sT_s)^{-2}\times\\
&&\Big(8 D_yR_s+16 R_p-8 R_s D_xT_s+8 R_s^2 T_p+8 R_s T_q-24 R_q T_s-4
  R_s T_sD_yR_s -\\&&
16 R_p R_s T_s+3 R_s^2 T_sD_xT_s -4 R_s^3 T_p T_s-4 R_s^2 T_q T_s+10 R_q R_s T_s^2+4 R_p R_s^2 T_s^2\Big)\end{eqnarray*}
and 
\begin{eqnarray*}
&&\al(D_y)=(4-R_sT_s)^{-2}\times\\
&&\Big(8 D_xT_s+16 T_q-8  T_sD_yR_s+8 R_q T_s^2+8 R_p T_s-24 R_s T_p-4 R_s T_sD_xT_s -\\&&16 R_s T_q T_s+3  R_s T_s^2D_yR_s-4 R_q R_s T_s^3-4 R_p R_s T_s^2+10 R_s^2 T_p T_s+4 R_s^2 T_q T_s^2\Big)\end{eqnarray*}
Recalling the fact that the distribution $H^+=\Span(D_x,D_y)$ is
integrable on $\mathcal J$, we see that the bilinear form $g$ descends
to a \emph{conformal metric} $g$ on the \emph{4-dimensional leaf space} ${\mathcal
  J}/H^+$, and that the descended metric has \emph{split} signature
iff $1-R_sT_s>0$ and \emph{Lorentzian} signature iff $1-R_sT_s<0$ and
$R_sT_s\neq 4$. Obviously the leaf space ${\mathcal
  J}/H^+$ may be identified with the 4-dimensional solution space of
the PDEs.

Analogous considerations can be performed for the metric (\ref{mne2})
if $R_sT_s\neq 0$. This is also degenerate along $D_x$ and $D_y$ in
$\mathcal J$. It also, apart from the degenerate directions $D_x$ and
$D_y$, has signature Lorentzian/split. For this metric we have 
$${\mathcal L}_{D_x}g=\frac{D_xR_s-2R_q}{R_s}g,\quad\&\quad {\mathcal
  L}_{D_y}g=\frac{D_yT_s-2T_p}{T_s}g,$$
so again (\ref{mne2}) descends to a conformal metric of \emph{split} (iff
$1-R_sT_s>0$ and $R_sT_s\neq 0$) or \emph{Lorentzian} signature (if
$1-R_sT_s<0$) on  ${\mathcal J}/H^+$. If $R_sT_s\neq 0$ and
$R_sT_s\neq 4$, these two conformal classes 
coincide on ${\mathcal J}/H^+$ as we explained before.

This finishes the proofs of the Theorem and the Proposition.
\end{proof}

\begin{remark}
In the proof we did not show that, contrary to the $(1,2,3)$ para-CR
forms (\ref{poi}) which satisfy (\ref{bla}),  we can force the
$(3,2,1)$ para-CR forms (\ref{pio}) to satisfy their \emph{torsionless}
counterpart, i.e. the first four of equations (\ref{fyu}). But this is
very easy: one first makes the normalizations
$u_{11}=u_{12}=u_{21}=u_{22}=0$ and all the other ones from Theorems
\ref{poj} and \ref{alb}, and after achieving (\ref{bla}) uses a
transformation, which is an identity on the obtained
$(\theta^1,\theta^2,\theta^3,\theta^4)$ and changes the obtained $\Omega_2$ and
$\Om_3$ according to:
\be
\Om_3\to\Om_3'= \Om_3-t^1_{~23}\theta^2,\quad\quad \Om_2\to\Om_2'=\Om_2-t^2_{~13}\theta^1,
\label{avt}\ee    
where $t^1_{~23}$ and $t^2_{~13}$ are torsions given
by (\ref{bla1}). Since the obtained $\theta^1$ and $\theta^2$ are 
linear combinations of $l_1$, $l_2$ and $l_3$ \emph{only} (because
$f_{13}=f_{23}=0$ is the chosen normalization (\ref{iwi1})!), then
transformation (\ref{avt}) is an allowed $(3,2,1)$-para-CR
transformation\footnote{Note however that this is \emph{not} a type
  $(1,2,3)$ para-CR transformation, and that if only such
  transformations are considered one can not absorb the torsion terms in
(\ref{bla}).} for the type $(3,2,1)$ para-CR forms
$(\theta^1,\theta^2,\theta^3,\theta^4,\Om_2,\Om_3)$. But this
transformation \emph{absorbs} the torsion terms in (\ref{bla}) and
makes the forms $(\theta^1,\theta^2,\theta^3,\theta^4,\Om_2',\Om_3')$
to satisfy the torsionless part of equations (\ref{fyu}). This means
that the type $(3,2,1)$ para-CR structures originating from the system
$z_{xx}=R$ $\&$ $z_{yy}=T$ with $D_x^2T=D_y^2R$, $J_1\equiv J_2\equiv
0$, $R_sT_s\neq 1$, contrary to the corresponding $(1,2,3)$ para-CR
structures, define quite a general 
conformal geometry on the solution space, and that their \emph{invariants}
can be described in terms of the \emph{curvature of the 
Cartan normal conformal connection} associated with this conformal
geometry. 
This observation, and an equivalent statement of
Theorem \ref{tnf} and Proposition \ref{pnf}, in a slightly different 
language, was first made by E.T. Newman and his collaborators
\cite{et}. According to Newman \cite{et}, using all the type $(3,2,1)$ para-CR
structures coming from the system $z_{xx}=R$ $\&$ $z_{yy}=T$
satisfying $J_1\equiv J_2\equiv 0$, one can obtain \emph{all} the
conformal classes of the Lorentzian 4-metrics. This statement is not
clear to us, and requires further justification. For example,
similarly to the attempts in \cite{perk}, we were
unable to calculate the Weyl tensor of the metrics (\ref{mne1}) and
(\ref{mne2}). This was mainly because of the huge length of the 
intermediate expressions encountered during the calculations of the
the Cartan normal conformal connection. 
Thus we were unable to see if it is general enough 
to cover all the conformal Lorentzian/split signature
4-metrics. Finding the conformally Einstein or Bach conditions for
these metrics in terms of the defining functions $R$ and $T$ would be
very interesting, and would complete the Newman programme.
\end{remark}

Although, we were unable to calculate the \emph{full} Weyl tensor of the
metric (\ref{mne2}), we succeded in calculating \emph{two} of its
components. These components \emph{must} vanish if we want the metric
(\ref{mne1}) to be conformally flat. Thus vanishing of these
components is a conformal property, and in turn, is a \emph{contact invariant}
property of the equations $z_{xx}=R$ $\&$ $z_{yy}=T$ satisfying
$D_x^2T=D_y^2R$ $\&$ $J_1\equiv J_2\equiv 0$. It is also a
\emph{para-CR invariant} property of the corresponding type $(3,2,1)$
para-CR structure. Defining the forms $(\om_1,\om_2,\om_3,\om_4)$ by
$(\om_1,\om_2,\om_3,\om_4)=(\nu_1,\nu_2,\la,\om')$, so that the metric
(\ref{mne2}) can be written as:
$$g=2\om_2\om_4+T_s\om_1^2-2\om_1\om_2+R_s\om_2^2,$$
we calculated the components $C^1_{~424}$ and $C^2_{~414}$ of the Weyl
tensor of this metric to be:
%/dos/paulandicotton/para321confflat.nb
$$C^1_{~424}=\frac{2R_{sss}(1-R_sT_s)+3R_{ss}(R_sT_s)_s}{4(1-R_sT_s)^4},\quad\quad
C^2_{~414}=\frac{2T_{sss}(1-R_sT_s)+3T_{ss}(R_sT_s)_s}{4(1-R_sT_s)^4}.$$
This proves the following theorem.
\begin{theorem}
For the system of PDEs $z_{xx}=R(x,y,z,z_x,z_y,z_{xy})$
$\&$ $z_{yy}=T(x,y,z,z_x,z_y,z_{xy})$ satisfying $D_x^2T=D_y^2R$ and
the metricity conditions 
$$J_1\equiv 0,\quad\quad\&\quad\quad J_2\equiv 0,$$
each of the conditions 
$$K_1=2R_{sss}(1-R_sT_s)+3R_{ss}(R_sT_s)_s=0,\quad\quad
K_2=2T_{sss}(1-R_sT_s)+3T_{ss}(R_sT_s)_s=0,$$
is invariant with respect to \emph{contact} transformations 
of the variables.
\end{theorem}

The new invariants $K_1$ and $K_2$ from the above Theorem justify the title of this section: although we were
unable to define the invariants of the type $(3,2,1)$ para-CR
structures in full generality, we discussed a class of such structures
whose invariants are just the conformal invariants of certain
4-metrics. In the next section 
we provide an example of the system $z_{xx}=R$ $\&$ $z_{yy}=T$
satisfying $J_1\equiv J_2\equiv 0$, whose corresponding conformal
4-metrics are quite interesting.  
\subsection{An example of (3,2,1) para-CR structures with nontrivial
  conformally Einstein metrics} 
Given a pair of PDEs $z_{xx}=R$ $\&$ $z_{yy}=T$ it is not easy to find
the most general solution of the \emph{integrability conditions} 
$D_x^2T=D_y^2R$ and the \emph{metricity conditions} $J_1\equiv
J_2\equiv 0$. But particular examples of functions $R$ and $T$
satisfying both sets of conditions can be given. The simplest of them,
but as we will see, still nontrivial, is given in the following
proposition.
\begin{proposition}
Let the functions $R=R(x,y,z,p,q,s)$ and $T=T(x,y,z,p,q,s)$ be functions
of variable $s$ alone, 
$$R=r(s)\quad\quad\&\quad\quad T=t(s),$$
and assume that their derivatives $r'$ and $t'$ satisfy $$1-r't'\neq
0.$$
Then such $R$ and $T$ satisfy simultaneously equations $D_x^2T=D_y^2R$
and $J_1\equiv
J_2\equiv 0$.
\end{proposition}
\begin{proof}
Applying the operators $D_x$ and $D_y$ from definitions
(\ref{pndha}) on functions $R=r(s)$ and $T=t(s)$, we obtain
$$D_xR=r'D_yR,\quad D_yR=r'D_xT,\quad D_xT=t'D_yR\quad \&\quad
D_yT=t'D_xT.$$
These are \emph{linear} equations for functions $D_xR$, $D_yR$, $D_xT$
and $D_yT$. hence, by an elementary argument they have a unique
solution 
$$D_xR=0,\quad D_yR=0,\quad D_xT=0,\quad D_yT=0,$$
when $1-r't'\neq 0$. Thus, with our assumptions, the operators $D_x$
and $D_y$, when acting on differentiable functions $f=f(s)$ of only
variable $s$, are identically vanishing. This, in particular, means
that $D_x^2T=0=D_y^2R$. Looking at the definitions of $J_1$ and $J_2$, in
which each term involves at least one derivative of $R$ or $T$ with
respect to $p$, $q$ and $D_x$ or $D_y$, we see that $J_1$ and $J_2$
are identically zero as well.\end{proof} 
Now, having a solution $R=r(s)$, $T=t(s)$ to the integrability and the metricity
conditions, we apply the theory from Section \ref{tow}, and 
calculate the conformal metric on the solution space of the system
$$z_{xx}=r(z_{xy})\quad\quad\&\quad\quad z_{yy}=t(z_{xy}).$$ 
Modulo a conformal factor the explicit formula for the metric $g$ as in
(\ref{mne1}) reads:
%/home/pawel/notebooks/paracr/123/para321example.nb
\be
\begin{aligned}
&g_0=~2~(1-r't')~\big(\der z-p\der x-q\der y\big)
  \der s~+~t'~(\der p-r\der x-s\der y)^2~ -\\&\quad \quad \,\,\,2 ~(\der p-r\der x-s\der y)(\der
  q-s\der x-t\der y)~+~
r'~(\der q-s\der x-t\der y)^2,\end{aligned}\label{bn}\ee
where $x,y,z,p,q,s$ are coordinates on $\mathcal J$, $r=r(s)$, and 
$t=t(s)$, $r'=\der r/\der s$, $t'=\der t/\der s$. We know from the
previous section that although this bilinear form is manifestly 
defined on $\mathcal J$, it transforms conformally when Lie dragged
along $D_x=\partial_x+p\partial_z+r\partial_p+s\partial_q$ and 
$D_y=\partial_y+q\partial_z+s\partial_p+t\partial_q$, and descends to
a conformal metric on the 4-dimensional solution space ${\mathcal J}/H^+$. It has
split signature iff $1-r't'>0$ and Lorentzian signature iff
$1-r't'<0$.  

The conformal invariants of this metric are para-CR invariants of the
$(3,2,1)$ para-CR structure 
 $[l_1,l_2,l_3,m_1,m_2,n]$ with $l_1=\der z-p\der x-q\der y$, 
$l_2=\der p-r\der x-s\der y$, $l_3=\der q-s\der x-t\der y$, 
$n=\der s$, $m_1=\der x$, $m_2=\der y$. These conformal invariants are
given in terms of the Cartan normal conformal connection for the class 
$[g_0]$. It is described by the following theorem
\begin{theorem}\label{blu1} 
Consider a metric $g={\rm e}^{2h}g_0$, where $h=h(s)$ is an arbitrary
smooth function and $g_0$ is as in (\ref{bn}). Let $(\om_1,\om_2,\om_3,\om_4,\om_5,\om_6)$ be a
coframe on $\mathcal J$ defined by
$\om_1=\der q-s\der x-t\der y$, $\om_2=\der s$, $\om_3=\der z-p\der
x-q\der y$, $\om_4=\der p-r\der x-s\der y$, $\om_5=\der x$,
$\om_6=\der y$, 
so that the metric is \be
g={\rm
  e}^{2h}\Big(~2(1-r't')\om_2\om_3+r'\om_1^2-2\om_1\om_4+t'\om_4^2~\Big).\label{nmb}\ee
Then the curvature of the Cartan normal conformal connection for $g$, when written on
$\mathcal J$, reads:
\be
\begin{aligned}
&{\mathcal R}=\bma 0&0&0&0&0&0\\
0&0&Z_2&0&0&0\\
0&0&0&0&0&0\\
0&\frac{-Z_2r'-Z_1t'}{2(1-r't')}&0&0&\frac{2Z_2+Z_1t'^2-Z_2r't'}{2(1-r't')}&0\\
0&0&\tfrac12(Z_2r'-Z_1t')&0&0&0\\
0&0&0&0&0&0
\ema\om_2\dz\om_4+\\
&\quad\quad\quad\bma 0&0&0&0&0&0\\
0&0&\tfrac12(Z_2r'-Z_1t')&0&0&0\\
0&0&0&0&0&0\\
0&\frac{Z_1r't'-2Z_1-Z_2r'^2}{2(1-r't')}&0&0&\frac{Z_2r'+Z_1t'}{2(1-r't')}&0\\
0&0&-Z_1&0&0&0\\
0&0&0&0&0&0
\ema\om_1\dz\om_2,
\end{aligned}
\ee
where 
$$Z_1=\frac{2(r't'-1)r^{(3)}-3r''(t'r')'}{4(1-r't')^2},\quad Z_2=\frac{2(r't'-1)t^{(3)}-3t''(t'r')'}{4(1-r't')^2}.$$
In particular the metric $g$ is conformally flat iff $Z_1\equiv
Z_2\equiv 0$, i.e. iff the functions $r$ and $t$ satisfy the system of
third order ODEs:
$$r^{(3)}=\frac{-3r''(r't')'}{2(1-r't')}\quad\quad \&\quad\quad t^{(3)}=\frac{-3t''(r't')'}{2(1-r't')}.$$ 
In general the metric $g$ is of (conformal) Petrov type $N\oplus N'$ in
the split signature case, and of Petrov type $N\oplus\bar{N}$ in the Lorentzian case.
\end{theorem}
The proof of this theorem consists in a straightforward, but lengthy
calculation, which we made using Mathematica. We omit it here. With
the use of Mathematica we also were able to check that the following
theorem is true:
\begin{theorem}\label{blu}
For every choice of sufficiently smooth functions $r=r(s)$ and
$t=t(s)$ there exists a function $h=h(s)$ such that the metric
(\ref{nmb}) is Ricci flat. The function $h$ in which the metric
$g={\rm e}^{2h}g_0$ is Ricci flat is a solution to the 2nd order ODE:
\begin{eqnarray*}
&h''=h'^2-\frac{(t'r')'}{1-r't'}h'+\frac{2\big(r^{(3)}t'+t^{(3)}r'\big)(1-r't')+2r''t''+4r't'r''t''+3t'^2r''^2+3r'^2t''^2}{8(1-r't')^2}.
\end{eqnarray*}
\end{theorem}
Thus, among the type $(3,2,1)$ para-CR structures originating from
PDEs $z_{xx}=R$ $\&$ $z_{yy}=T$ we found conformally Ricci flat but
conformally non-flat metrics. It further follows that these metrics, 
in addition to being conformally Ricci flat and of type $N\oplus N'$,
have \emph{reduced holonomy}. This is because they 
have a \emph{covariantly constant null direction}, which is alligned
with the vector field $\partial_z$. In the Lorentzian case, i.e. when $1-r't'<0$, they are
known in General Relativity theory as $pp$-waves (see e.g. \cite{lew} 
for a definition and \cite{ln} for a discussion of their conformal properties). 

It would be very interesting to find type $(3,2,1)$ para-CR
structures defined by $z_{xx}=R$ $\&$ $z_{yy}=T$, which define
conformally Einstein metrics (\ref{mne1})-(\ref{mne2}) other than
$pp$-waves or their split signature counterparts discussed here.  
 

\begin{thebibliography}{99}
\bibitem{alex1} Alekseevsky D V, Medori C, Tomassini A (2008),
  ``Maximally homogeneous para-CR manifolds of semisimple type'',  to 
  appear in
  \emph{Handbook of Pseudo-Riemannian geometry and Supersymmetry},  
arXiv:0808.0431 
\bibitem{cartan} Cartan E (1910) ``Les systemes de Pfaff a cinq variables et
  les equations aux derivees partielles du seconde ordre''
  \emph{Ann. Sc. Norm. Sup.} {\bf 27} 109-192 
\bibitem{cartanode} Cartan E (1924) ``Varietes a connexion projective''
  {\it Bull. Soc. Math.} {\bf LII} 205-241
\bibitem{chern} Chern S S (1940) ``The geometry of the differential
  equations $y'''=F(x,y,y',y'')$'' {\it Sci. Rep. Nat. Tsing Hua
  Univ.} {\bf 4} 97-111
%\bibitem{docarmo} Do Carmo M P, (1992) \emph{Riemannian Geometry}, Birkhauser, Boston, Basel, Berlin
\bibitem{feff} Fefferman C, Graham C R (1985) ``Conformal invariants'', in
  \emph{Elie Cartan et mathematiques d'aujourd'hui}, Asterisque, hors
  serie (Societe Mathematique de France, Paris) 95-116  
\bibitem{et} Fritelli S, Kozameh C N, Newman E T (1995) ``GR via
  characteristic surfaces'' {\it J. Math. Phys.} {\bf 36} 4984-
\bibitem{nemo} Fritelli S, Newman E T, Nurowski P (2003)
  "Conformal Lorentzian metrics on the spaces of curves and
  2-surfaces" {\it Class. Q. Grav.} {\bf 20} 3649-3659
\bibitem{godphd} Godlinski M (2008) ``Geometry of Third-Order Ordinary
  Differential Equations and Its Applications in General Relativity''
  PhD Thesis, Warsaw University, arXiv: 0810.2234
\bibitem{nurgod1} Godlinski M, Nurowski P (2009) ``Geometry of
  third-order ODEs'', arXiv: 0902.4129 
\bibitem{gover} Gover A R, Nurowski P (2006)  ``Obstructions to conformally Einstein metrics in n dimensions'' {\it Journ. Geom. Phys.}  {\bf 56}  450-484 
\bibitem{lieodes} \textrm{Lie S (1924) ``Klassifikation und Integration von
gewohnlichen Differentialgleichungen zwischen $x$, $y$, die eine Gruppe von
Transformationen gestatten III'', in \textit{Gesammelte Abhandlungen,
  Vol. 5}, Teubner, Leipzig}
\bibitem{ln} Leistner Th, Nurowski P (2008) ``Ambient metrics for
  n-dimensional pp-waves'' arXiv:0810.2903
\bibitem{lew} Lewandowski J (1992) ``Reduced holonomy group and
  Einstein equations with a cosmological constant''
  \emph{Class. Q. Grav.} {\bf 9} L147-L151
\bibitem{nurdif} Nurowski P (2005) ``Differential equations and
  conformal structures'' {\it Journ. Geom. Phys.}  {\bf 55}  19-49 
\bibitem{davepaw} Nurowski P, Robinson D C (2000) ``Intrinsic geometry of a
  null hypersurface'' {\it Class. Q. Grav.} {\bf 17}, 4065-4084
\bibitem{nurspar} Nurowski P, Sparling G A J (2003) ``Three
  dimensional Cauchy-Riemann structures and second order ordinary
  differential equations'' {\it Class. Q. Grav.} {\bf 20}, 4995-5016
\bibitem{olver} Olver P J (1996) {\it Equivalence Invariants and
  Symmetry} (Cambridge University Press, Cambridge)
\bibitem{perk} Perkins K (2006) ``The Cartan-Weyl conformal geometry of a
  pair of second-order partial-differential equations'', PhD Thesis,
  Department of Physics $\&$ Astronomy, University of Pittsburgh
\bibitem{tanaka} Tanaka N (1985) ``On affine symmetric spaces and the
  automorphism groups of product manifolds, {\it Hokkaido Math. J.}
  {\bf 14}, 277-351
\bibitem{tresse} Tresse M A (1896) {\it 
Determinations des invariants ponctuels  de
l'equation differentielle ordinaire du second
ordre} $y''= \omega (x,y,y')$, Hirzel, Leipzig.
\bibitem{wun} W\"unschmann K (1905) ``{\it \"Uber Beruhrungsbedingungen
  bei Differentialgleichungen}'', Dissertation, Greifswald.
\end{thebibliography}
\end{document}